\documentclass{amsart}
\usepackage{amsmath,amsthm}
\usepackage{hyperref}
\usepackage[psamsfonts]{amssymb}
\usepackage[all]{xy}
\usepackage{graphicx}
\usepackage[T1]{fontenc}
\usepackage[bitstream-charter]{mathdesign}
\newtheorem{theorem}{Theorem}[section]
\newtheorem{prop}[theorem]{Proposition}
\newtheorem{lemma}[theorem]{Lemma}
\newtheorem{cor}[theorem]{Corollary}

\theoremstyle{remark}

\newtheorem{definition}[theorem]{Definition}
\newtheorem{remark}[theorem]{Remark}

\def\co{\colon\thinspace}

\def\ep{\epsilon}

\def\K{\mathbb{K}}
\def\R{\mathbb{R}}
\def\N{\mathbb{N}}

\def\Z{\mathbb{Z}}

\def\rk{\mathop{rank}\nolimits}

\def\Diff{\mathop{\mathrm{Diff}}\nolimits}
\title{Linking and the Morse complex}
\author{Michael Usher}
\date{\today}
\address{Department of Mathematics\\University of Georgia\\Athens, GA 30602}
\email{\href{mailto:usher@math.uga.edu}{usher@math.uga.edu}}

\begin{document}

\begin{abstract}
For a Morse function $f$ on a compact oriented manifold $M$, we show that $f$ has more critical points than the number required by the Morse inequalities if and only if there exists a certain class of link in $M$ whose components have nontrivial linking number, such that the minimal value of $f$ on one of the components is larger than its maximal value on the other.  Indeed we characterize the precise number of critical points of $f$ in terms of the Betti numbers of $M$ and the behavior of $f$ with respect to links.  This can be viewed as a refinement, in the case of compact manifolds, of the Rabinowitz Saddle Point Theorem.  Our approach, inspired in part by techniques of chain-level symplectic Floer theory, involves associating to collections of chains in $M$ algebraic operations on the Morse complex of $f$, which yields relationships between the linking numbers of homologically trivial (pseudo)cycles in $M$ and an algebraic linking pairing on the Morse complex.
\end{abstract}

\maketitle

\section{Introduction} \label{intro}

Let $f\co M\to \R$ be a Morse function on a compact $n$-dimensional manifold $M$; thus around each critical point $p$ of $f$ there are coordinates $(x_1,\ldots,x_n)$ in terms of which $f$ is given by the formula \[ f(x_1,\ldots,x_n)=-\sum_{i=1}^{k}x_{i}^{2}+\sum_{i=k+1}^{n}x_{i}^{2} \] for some integer $k$ called the \emph{index} of $p$ and denoted in this paper by $|p|_f$.  For each integer $k$ let $c_k(f)$ denote the number of critical points of $f$ having index $k$, and define the Morse polynomial of $f$ by \[ \mathfrak{M}_f(t)=\sum_{k=0}^{n}c_k(f)t^k.\]  Meanwhile if $\K$ is a field let $\mathfrak{b}_k(M;\K)$ be the rank of the $k$th homology $H_k(M;\K)$ with coefficients in $\K$ and define the Poincar\'e polynomial of $M$ with coefficients in $\K$ to be \[ \mathfrak{P}_M(t;\K)=\sum_{k=0}^{n}\frak{b}_{k}(M;\K)t^k.\]  One way of expressing the famous \emph{Morse inequalities} is to say that one has \[ \mathfrak{M}_f(t)=\mathfrak{P}_M(t;\K)+(1+t)\mathfrak{Q}_f(t;\K) \] for some polynomial $\mathfrak{Q}_f(t;\K)=\sum_{k=0}^{n-1}q_{k}(f;\K)t^k$ all of whose coefficients $q_{k}(f;\K)$ are nonnegative.  Indeed, using the gradient flow of 
$f$ it is possible to construct a chain complex $(CM_{*}(f;\K),d_f)$ such that $CM_{k}(f;\K)$ is a $\K$-vector space of dimension $c_k(f)$ and such that the homology of the complex is isomorphic to $H_{*}(M;\K)$, and then the coefficients $q_{k}(f;\K)$ of the polynomial $\mathfrak{Q}_f(\cdot;\K)$ are the ranks of the differentials $d_{f,k+1}\co CM_{k+1}(f;\K)\to CM_{k}(f;\K)$.

In particular, for any coeffcient field $\K$, the number of critical points of index $k$ for any Morse function $f$ obeys $c_k(f)\geq \frak{b}_{k}(f;\K)$, and equality holds in this inequality if and only if $q_{k}(f;\K)$ and $q_{k-1}(f;\K)$ are equal to zero.  Thus a nonzero value of $q_k(f;\K)$ corresponds to $f$ having ``extra'' critical points in indices $k$ and $k+1$.  This paper is concerned with giving alternate interpretations of the numbers $q_k(f;\K)$, in terms of the \emph{linking} of homologically trivial cycles in $M$.  (Actually, we will generally work with pseudocycles (see \cite[Section 6.5]{MS} and Section \ref{pseudolink} below for precise definitions) instead of cycles; in view of results such as \cite[Theorem 1.1]{Z} this will encode essentially the same information.  In particular it makes sense to ask whether a given pseudocycle is homologically trivial; our convention is that a \emph{pseudoboundary} is by definition a homologically trivial pseudocycle.)

Although our methods are rather different, our results are conceptually related to results along the lines the Saddle Point Theorem of \cite{R}, which assert under various rather general hypotheses that for a function $f\co M\to \R$ (where $M$ is, say, a Banach manifold) which satisfies the Palais--Smale condition, if there are null-bordant submanifolds $A,B\subset M$ such that $\inf_B f>\sup_A f$ and $A$ and $B$  are linked in the sense that any other submanifold whose boundary is $A$ must intersect $B$, then $f$ must have a critical point with critical value at least $\inf_B f$.  Various extensions and refinements of this result have appeared; for instance one can see from \cite[Theorems II.1.1$'$, II.1.5]{C} that if $\dim A=k$ then one can arrange to find a critical point of $f$ whose local Morse homology is nontrivial in degree $k$, and so the critical point will have index $k$ provided that it is nondegenerate.  However if $M$ has nontrivial singular homology in degree $k$ and if $f$ is \emph{any} Morse function then $f$ will automatically have critical points of index $k$, which might seem to indicate that in this case the linking condition in the hypothesis 
of the Saddle Point Theorem only leads to critical points whose existence can be explained just from the homology of $M$.  On the contrary, we show in this paper that, at least in the finite-dimensional, compact case, a certain type of linking is rather precisely associated with  ``extra'' critical points:

\begin{theorem}\label{main1} Let $f\co M\to \R$ be a Morse function on a compact oriented $n$-dimensional manifold $M$, and let $\K$ be any ring\footnote{In this paper ``ring'' means ``commutative ring with unity.''}.  The following are equivalent:
\begin{itemize}
\item[(i)] The Morse boundary operator\footnote{To construct the Morse boundary operator one needs to choose an auxiliary Riemannian metric; however its triviality or nontriviality is independent of this choice.} $d_{f,k+1}\co CM_{k+1}(f;\K)\to CM_k(f;\K)$ is nontrivial (\emph{i.e.}, in our earlier notation when $\K$ is a field, $q_k(f;\K)\neq 0$.)
\item[(ii)] There are pseudoboundaries $b_{\pm}\co B_{\pm}\to M$,  where $\dim B_+=k$ and $\dim B_-=n-k-1$, such that $\overline{Im(b_-)}\cap \overline{Im(b_+)}=\varnothing$ and the $\K$-valued linking number $lk_{\K}(b_-,b_+)$ is nonzero, and such that $\min(f|_{\overline{Im(b_-)}})>\max(f|_{\overline{Im(b_+)}})$.
\end{itemize}
Moreover, if (i) holds, then from the stable and unstable manifolds of $f$ associated to a suitable Riemannian metric, one may construct a pair of pseudoboundaries $b_{\pm}\co B_{\pm}\to M$ of dimensions $k$ and $n-k-1$ with $lk_{\K}(b_-,b_+)\neq 0$ such that the value of $\min(f|_{\overline{Im(b_-)}})-\max(f|_{\overline{Im(b_+)}})$ is as large as possible.
\end{theorem}

\begin{proof} The equivalence of (i) and (ii) follows from Theorem \ref{alggeom} and the first sentence of Proposition \ref{alg-betaprop}, since in the notation of Theorem \ref{alggeom} the statement (i) is equivalent to the statement that $\beta^{alg}_{k}(f;\K)> 0$, while statement (ii) is equivalent to the statement that $\beta^{geom}_{k}(f;\K)> 0$.
The final statement of the theorem follows from the constructions in Section \ref{crittolink} which are used to prove the implication `(i)$\Rightarrow$(ii).'
\end{proof}

We leave to Section \ref{pseudolink} the precise definitions related to linking numbers of pseudoboundaries.  Suffice it to note for the moment that a special case of a pair of pseudoboundaries $b_{\pm}\co B_{\pm}\to M$ is given by setting $B_{\pm}=\partial C_{\pm}$ for some compact manifolds with boundary $C_{\pm}$, and setting $b_{\pm}=c_{\pm}|_{C_{\pm}}$ for some pair of smooth maps $c_{\pm}\co C_{\pm}\to\R$.  Assuming that $b_+$ and $b_-$ have disjoint images, the $\mathbb{Z}$-valued linking number $lk(b_-,b_+)$ is obtained by perturbing $c_+$ to make it transverse to $b_-$ and then taking the intersection number of $b_-$ and $c_+$ (which, one can show, depends only on $b_{\pm}$ and not on $c_{\pm}$), and the $\K$-valued linking number $lk_{\K}(b_-,b_+)$ is just the image of $lk(b_-,b_+)$ under the unique unital ring morphism $\Z\to \K$.  The more general setup of pseudoboundaries generalizes this only in that the domains and images of $b_{\pm}$ and $c_{\pm}$ are allowed some mild noncompactness (the images should be precompact, and ``compact up to codimension two'' in a standard sense that is recalled in Section \ref{pseudolink}).  If we were to instead require the domains of $b_{\pm}$ to be compact, then of course the implication `(ii)$\Rightarrow$(i)' in Theorem \ref{main1} would follow \emph{a fortiori}, while `(i)$\Rightarrow$(ii)' would hold provided that $\K$ has characteristic zero by Remark \ref{betasmooth}.

The implication (i)$\Rightarrow$(ii) in Theorem \ref{main1} is perhaps surprisingly strong in that it yields not just a smoothly nontrivial link in $M$ but a link with nontrivial linking number.  For example, letting $M=S^3$, if $L=L_0\cup L_1$ is any link with $\dim L_0=\dim L_1=1$ such that every minimal-genus Seifert surface for $L_0$ intersects $L_1$ (there are many examples of links with this property that moreover have zero linking number, beginning with the Whitehead link), then standard minimax arguments along the lines of those used in \cite{R} and \cite{C} to prove the Saddle Point Theorem and its variants show that if $f\co S^3\to \R$ is any Morse function such that $\min(f|_{L_0})>\max(f|_{L_1})$ then $f$ must have a critical point which is not a global extremum, and so in view of the homology of $S^3$ the Morse boundary operator $d_f\co CM_{*}(S^3;\Z)\to CM_{*-1}(S^3;\K)$ must be nonzero in some degree.  So Theorem \ref{main1} then gives a link $L'_0\cup L'_1$ in $S^3$ such that $\min(f|_{L'_0})>\min(f|_{L'_1})$ and whose components $L'_0$ and $L'_1$ are homologically trivial and have nonzero linking number, even if the original link had zero linking number. (It is not immediately clear whether the $L'_i$ must be one-dimensional; conceivably one of them could be zero-dimensional and the other two-dimensional.)

\begin{remark}The orientability hypothesis on $M$ in Theorem \ref{main1} and in Theorem \ref{main2} below may be dropped if one restricts to rings $\K$ having characteristic two and modifies the definition of a pseudoboundary (see Definition \ref{pseudodef}) so that the domains of a pseudoboundary and of its bounding pseudochain need not be orientable.   This can be seen by direct inspection of the proofs of the theorems if one simply ignores all references to orientations therein.\end{remark}
 
Going beyond Theorem \ref{main1}, for any field $\K$ one can characterize the precise values of the coefficients $q_k(f;\K)$, not just whether or not they are zero, in terms of the linking of pseudoboundaries, though this requires a somewhat more complicated description and indeed requires some knowledge of the gradient flow of the function $f$ with respect to a suitable metric.  If $b_+\co B_+\to M$ and $b_-\co B_-\to M$ are pseudoboundaries of dimensions $k$ and $n-k-1$ respectively, from the general theory in Section \ref{ops} we obtain a quantity denoted there by $\Pi(M_{-f},I_{b_+,b_-}M_f)$.  This quantity may be intuitively described as a signed count of those trajectories $\gamma\co [0,T]\to M$ of the vector field $-\nabla f$ such that $\gamma(0)\in Im(b_+)$ and $\gamma(T)\in Im(b_-)$, where $T$ is a positive number (which is allowed to vary from trajectory to trajectory).  The quantity $\Pi(M_{-f},I_{b_+,b_-}M_f)$ should in general be expected to depend on the Riemannian metric used to define the gradient flow; however there is one case where it is obviously independent of the metric and also is easily computable: since the function $f$ decreases along its gradient flowlines, if one has $\sup(f|_{Im(b_+)})<\inf(f|_{Im(b_-)})$ then clearly 
$\Pi(M_{-f},I_{b_+,b_-}M_f)=0$ (indeed this is the reason that $\Pi(M_{-f},I_{b_+,b_-}M_f)$ did not appear in the statement of Theorem \ref{main1}).  We will see that in general $\Pi(M_{-f},I_{b_+,b_-}M_f)$ serves as a sort of correction term in the relationship between geometric linking of pseudoboundaries in $M$ and the algebraic linking pairing on the Morse complex of $f$ defined in (\ref{mlpair}).

\begin{theorem}\label{main2}  Let $f\co M\to \R$ be a Morse function on a compact oriented $n$-dimensional manifold $M$, and let $\K$ be a field.  There are Riemannian metrics on $M$ such that the following are equivalent for all nonnegative integers $k$ and $m$:
\begin{itemize}\item[(i)] The rank of the Morse boundary operator $d_{f,k+1}\co CM_{k+1}(f;\K)\to CM_{k}(f;\K)$ is at least $m$.
\item[(ii)] There are $k$-dimensional pseudoboundaries $b_{1,+},\ldots,b_{r,+}$ and $(n-k-1)$-dimensional pseudoboundaries $b_{1,-},\ldots,b_{s,-}$ such that the matrix $L$ whose entries are given by \[ L_{ij}=lk(b_{j,-},b_{i,+})-(-1)^{(n-k)(k+1)}\Pi(M_{-f},I_{b_{i,+},b_{j,-}}M_f) \] has rank at least $m$.
\end{itemize}
\end{theorem}  

\begin{proof} See Corollaries \ref{linkcor1} and \ref{cormain2}.
\end{proof}

In fact, as noted in Corollary \ref{cormain2}, if (ii) (or equivalently (i)) holds, then the pseudocycles $b_{i,+},b_{j,-}$ can be chosen to obey $\Pi(M_{-f},I_{b_{i,+},b_{j,-}}M_f)=0$.   

  
To rephrase Theorem \ref{main2}, for each $k$ the rank $q_k(f;\K)$ of the Morse boundary operator \linebreak $d_{f,k+1}\co CM_{k+1}(f;\K)\to CM_k(f;\K)$ can be expressed as the largest possible rank of a matrix whose entries are given by the $\K$-valued linking numbers of each member of a collection of $k$-dimensional pseudoboundaries with each member of a collection of $(n-k-1)$-dimensional pseudoboundaries, corrected by a term  arising from ``negative gradient flow chords'' from the former to the latter.  Moreover, there are collections of pseudoboundaries for which the maximal possible rank is attained and the correction term vanishes.  

By Poincar\'e duality, the Betti numbers $\frak{b}_k(M;\K)$ can somewhat similarly be described as the maximal rank of a certain kind of matrix: namely, a matrix whose entries are given by the $\K$-valued intersection numbers of each member of a collection of $k$-dimensional pseudocycles with each member of a collection of $(n-k)$-dimensional pseudocycles.   Thus a general Morse function $f$ on an oriented compact manifold  $M$ has $\left(\sum_{k}\frak{b}_k(M;\K)\right)$-many critical points which can be seen as resulting purely from the homology of $M$ and may be associated to the intersection of cycles in $M$, and also exactly $2\left(\sum_{k}q_k(f;\K)\right)$-many other critical points, and these other critical points are not accounted for by the homology of $M$ but may be associated to the behavior of $f$ with respect to linked, homologically trivial cycles in $M$.

\subsection{Outline of the paper and additional remarks}

The body of the paper begins with the following Section \ref{or}, which sets up some notation and conventions relating to orientations and Morse theory and works out some signs that are useful later; readers, especially those content to ignore sign issues and work mod 2, may prefer to skip this section on first reading and refer back to it as necessary.

Section \ref{pseudolink} introduces the formalism of pseudochains and pseudoboundaries that is used throughout the paper; these are natural modifications of the pseudocycles considered in \cite[Section 6.5]{MS}.  In particular we show that a pair of pseudoboundaries the sum of whose dimensions is one less than the dimension of the ambient manifold, and the closures of whose images are disjoint, has a well-defined linking number, about which we prove various properties.  We also prove Lemma \ref{smoothen}, which for some purposes allows one to work with homologically trivial maps of compact smooth manifolds into $M$ in place of pseudoboundaries; however if one wishes to work over $\Z$ or $\Z/p\Z$ rather than $\mathbb{Q}$ then restricting to  maps of compact smooth manifolds  will lead one to miss some topological information, consistently with results that date back to \cite[Th\'eor\`eme III.9]{T}.

Section \ref{ops} recalls the Morse complex $CM_{*}(f;\K)$ of a Morse function $f$ and introduces several operations on it.  Among these are rather standard ones corresponding after passing to homology to Poincar\'e duality and to the cap product. Importantly, these operations can be defined on chain level, and consideration of their chain-level definitions suggests some other operations that capture different information.  In particular the chain level Poincar\'e pairing can easily be modified to obtain a Morse-theoretic linking pairing, whose relation to the linking of pseudoboundaries is fundamental for this paper.  As for the cap product, it is described on chain level by considering negative gradient trajectories which pass through a given pseudochain, and this chain level operation has natural generalizations obtained from negative gradient trajectories which instead pass through several prescribed pseudochains at different times.  These more general operations (denoted $I_{g_0,\cdots,g_{k-1}}$) are not chain maps, so they do not pass to homology and, at least in and of themselves, do not encode topologically invariant information (though suitable combinations of them should give rise to Massey products).  While from some perspectives this lack of topological invariance would be seen as a defect, our focus in this paper is on the ``extra'' critical points that a given Morse function $f$ may or may not have, and these extra critical points are also not topologically invariant in that their existence typically depends on $f$ (and throughout the paper we are viewing the function $f$, not just the manifold on which it is defined, as the basic object of study).

While the $I_{g_{0},\cdots, g_{k-1}}$ are not chain maps for $k\geq 2$, they do satisfy some important identities which are obtained by examining boundaries of certain one-dimensional moduli spaces of gradient trajectories and are described in general in Remark \ref{ainfty}.  The only ones of these that are used for the main results of this paper are Propositions \ref{igprop}(ii) and \ref{fundid} (which concern the cases $k=1,2$), though it would be interesting to know if the identities for $k>2$ can be used to provide a relationship between Morse theory and Milnor's higher-order linking numbers.
We would like to take this opportunity to mention a broader perspective on these identities.  Given a finite set of pseudochains $g_i\co C_i\to M$ which are in suitably general position with respect to each other, a construction in the spirit of \cite[Section 3.4]{FOOO} should give rise to an $A_{\infty}$-algebra $\mathcal{C}(M)$ of pseudochains in $M$ with each $g_i\in \mathcal{C}(M)$, whose operations $\frak{m}_l$, when applied to tuples of distinct $g_i$ from the given collection, obey \begin{align*}  \frak{m}_1(g_i)&=\pm g_i|_{\partial C_i} \\ \frak{m}_2(g_{i_1},g_{i_2})&=\pm g_{i_{1}}\times_M g_{i_2}\,\,\mbox{ if }i_1\neq i_2 \\ \frak{m}_l(g_{i_1},\ldots,g_{i_l})&=0 \,\,\mbox{ if }l\geq 3\mbox{ and }i_1,\ldots,i_l\mbox{ are distinct}.\end{align*}  In this case Remark \ref{ainfty} would be a special case of the statement that the Morse complex $CM_{*}(f;\K)$ is an $A_{\infty}$-module over the $A_{\infty}$-algebra $\mathcal{C}(M)$, with part of the module action given up to sign by the operators $I_{g_{i_0}\ldots g_{i_{k-1}}}$.\footnote{To be clear, the existence of this $A_{\infty}$-module structure is not proven either in this paper or, as far as I know, anywhere else in the literature; this paper does however contain detailed proofs of the only consequences of the conjectural $A_{\infty}$-module structure that we require in the proofs of our main theorems, namely Propositions \ref{igprop}(ii) and \ref{fundid}.} This is reminiscent of, though distinct from, the discussion in \cite[Chapter 1]{F}, in which Fukaya organizes the Morse complexes associated to all of the various Morse functions on the manifold $M$ into an $A_{\infty}$-category; by contrast, we work with a single fixed Morse function $f$ on $M$, and the relevant $A_{\infty}$ structure on $CM_{*}(f;\K)$ arises not from other Morse functions but rather from the interaction of $f$ with an $A_{\infty}$-algebra of (pseudo)chains in $M$.  One could perhaps enlarge Fukaya's picture to incorporate ours by regarding $(\mathcal{C}(M),\frak{m}_1)$ as playing the role of the ``Morse complex'' of the (non-Morse) function $0$ on $M$. With this said, we will just prove those aspects of the $A_{\infty}$ structures that we require in a direct fashion, so the phrase ``$A_{\infty}$'' will not appear again in the paper. 

Section \ref{linkcrit} begins the process of establishing a relationship between the linking of pseudoboundaries described in Section \ref{pseudolink} with the operations on the Morse complex described in Section \ref{ops}; in particular the implications ``(ii)$\Rightarrow$(i)'' in Theorems \ref{main1} and \ref{main2} are established in Section \ref{linkcrit}.  The key ingredient in this regard is Proposition \ref{linklinkprop}, which uses Propositions \ref{igprop} and \ref{fundid} to associate to a pair $b_0,b_1$ of linked pseudoboundaries in $M$ a pair of boundaries in the Morse complexes of $\pm f$, whose Morse-theoretic linking pairing is determined by the linking number of the pseudoboundaries together with the correction term $\Pi(M_{-f},I_{b_0,b_1}M_f)$ alluded to earlier. If the pseudoboundary $b_0$ has dimension $k$, then its associated boundary in the Morse complex is obtained as a linear combination of those index-$k$ critical points which arise as the limit in forward time of a negative gradient flowline of $f$ which passes through the image of $b_0$ (generically there will be only finitely many such flowlines).  This is vaguely similar to the usual strategy of obtaining critical points the Saddle Point Theorem as in \cite{R}, wherein one essentially ``pushes down'' $b_0$ via the negative gradient flow until one encounters a critical point.  However, if one does follows Rabinowitz's approach naively then one should not even expect to locate a critical point of index $k$, since the critical points that one first encounters would be likely to have higher index. Although there exist ways of guaranteeing that one finds an index-$k$ critical point by a similar procedure (essentially by first replacing $b_0$ with a certain other chain which is homologous to it in an appropriate relative homology group, see \cite[Section II.1]{C}), these older methods still do not seem to suffice to yield the quantitative estimates on $q_k(f;\K)$ in Theorem \ref{main2}, or indeed the nonvanishing of $q_k(f;\K)$ in Theorem \ref{main1} if the ambient manifold has nonzero $k$th Betti number.  However, by taking the approach---familiar from Floer theory---of using not the entire gradient flow of $f$ but rather only certain zero-dimensional spaces of gradient trajectories, and by exploiting more fully the algebraic structures on the Morse complex, we are able to obtain these quantitative results.

 In Section \ref{linkcrit} we also formulate and begin to prove Theorem \ref{alggeom}, which can be seen as a more refined version of Theorem \ref{main1}.  Theorem \ref{alggeom} equates two quantities associated to a Morse function $f\co M\to\R$ on a compact oriented manifold and a ring $\K$: the \emph{geometric link separation} $\beta^{geom}_{k}(f;\K)$ and the \emph{algebraic link separation} $\beta^{alg}_{k}(f;\K)$.  The geometric link separation describes the maximal amount by which the function $f$ separates any pair of pseudoboundaries of appropriate dimensions whose linking number is nontrivial; thus Theorem \ref{main1} asserts that this quantity is positive if and only if the Morse boundary operator is nontrivial in the appropriate degree.  The algebraic link separation in general has a more complicated definition which we defer to Section \ref{linkcrit}, but when $\K$ is a field we show in Proposition \ref{alg-betaprop} that $\beta^{alg}_{k}(f;\K)$ is equal to a quantity introduced in the Floer-theoretic context in \cite{U11} called the \emph{boundary depth} of $f$: the Morse complex $CM_{*}(f;\K)$ has a natural filtration given by the critical values, and the boundary depth is the minimal quantity $\beta$ such that any chain $x$ in the image of the boundary operator must be the boundary of a chain $y$ whose filtration level is at most $\beta$ higher than that of $x$.  This paper had its origins in an attempt to obtain a more geometric interpretation of the boundary depth in the Morse-theoretic context---which in particular reflects the fact that the boundary depth depends continuously on the function $f$ with respect to the $C^0$-norm---and when $\K$ is a field that goal is achieved by Theorem \ref{alggeom}.

The implications ``(i)$\Rightarrow$(ii)'' in Theorems \ref{main1} and \ref{main2} are proven in Section \ref{crittolink}.   Our approach is to associate to any element $a$ in the Morse complex $CM_{k+1}(f;\Z)$ a pseudochain which represents $a$ in a suitable sense; see Lemma \ref{chainconstruct}.  In the case that the Morse differential of $a$ is trivial, such a construction already appears in \cite{S99}, where it is used to construct an equivalence between Morse homology and singular homology.  Our interest lies in the case that the Morse differential $d_{f,k+1}a$ of $a$ is nontrivial, and then the boundary of the pseudochain will be a pseudoboundary whose properties with respect to linking numbers and with respect to the function $f$ are patterned after corresponding properties of $d_{f,k+1}a$ in the Morse complex; for this purpose a somewhat more refined construction than that in \cite{S99} is required.  Our construction makes use of properties of the manifold-with-corners structure of the compactified unstable manifolds of $f$ with respect to metrics obeying a local triviality condition near the critical points, as established in \cite{BH}.  (The existence of such a structure has been proven for more general metrics in \cite{Q}; however we also require the evaluation map from the compactified unstable manifold into $M$ to be smooth, a property which currently seems to be known only in the locally trivial case.)

Finally, the closing Section \ref{app} contains proofs of three technical results deferred from Sections \ref{pseudolink} and \ref{ops}, two of which concern issues of transversality and the other of which works out in detail (with careful attention paid to orientations) the boundary of the moduli space which gives rise to the key identity in Proposition \ref{fundid}.

\subsection*{Acknowledgements}  The original inspiration for this project came partially from conversations with Leonid Polterovich during a visit to the University of Chicago in March 2011.  This work was supported by NSF grant DMS-1105700.

\section{Conventions for orientations and Morse theory} \label{or}

The most detailed and coherent treatment of the orientation issues that one encounters in dealing simultaneously with intersection theory and the Morse
complex that I have found is \cite[Appendix A]{BC}, so I will borrow most of my orientation conventions from there.

\subsection{Short exact sequences.}\label{or:ses}  Many of the vector spaces that one needs to orient in a discussion such as this are related to each other by short exact sequences, and so one should first decide on an orientation convention for short exact sequences; following \cite{BC}, our convention is that, given a short exact sequence of vector spaces \[ \xymatrix{ 0\ar[r] & A\ar[r]^{f} & B\ar[r]^{g}& C\ar[r]& 0 }\] in which two of $A,B,C$ are oriented, the other should be oriented in such a way that, if $\{a_1,\ldots,a_p\}$ and $\{c_1,\ldots,c_q\}$ are oriented bases for $A$ and $C$ respectively and if $b_i\in B$ are chosen so that $g(b_i)=c_i$, then \[ \{b_1,\ldots,b_q,f(a_1),\ldots,f(a_p)\} \] is an oriented basis for $B$.  

We orient a vector space given as a direct sum $V\oplus W$ of oriented vector spaces by using an oriented basis for $V$ followed by an oriented basis for $W$; in terms of our short exact sequence convention this amounts to orienting $V\oplus W$ by using the short exact sequence \[ \xymatrix{ 0\ar[r] & W \ar[r]& V\oplus W\ar[r]& V\ar[r]& 0 }\]  A product $M\times N$ of oriented manifolds is then oriented by 
means of the direct sum decomposition $T_{(m,n)}=T_mM\oplus T_n N$.

\subsection{Group actions.}\label{or:actions} If $G$ is an oriented Lie group with Lie algebra $\mathfrak{g}$ (for us $G$ will always be $\mathbb{R}$) acting freely on an oriented manifold $M$, the quotient $M/G$ is oriented according to the exact sequence on tangent spaces given by the action: \[ \xymatrix{ 0\ar[r]&\mathfrak{g}\ar[r]& T_mM\ar[r] & T_{[m]}(M/G)\ar[r]& 0 }\]

\subsection{Boundaries.}\label{or:bound}  For boundaries of oriented manifolds we use the standard ``outer-normal-first'' convention.

\subsection{Fiber products.}\label{or:fp} Many of the important spaces that we need to orient can be seen as \emph{fiber products}: if $f\co V\to M$ and $g\co W\to M$ are smooth maps, the fiber product $V {}_{f}\times_g W$ is given by \[ V {}_{f}\times_g W=\{(v,w)\in V\times W|f(v)=g(w)\}.\]  In other words, where $\Delta\subset M$ is the diagonal, we have $V {}_{f}\times_g W=(f\times g)^{-1}(\Delta)$.  So if $f\times g$ is transverse to $\Delta$ (in which case we say that ``the fiber product is cut out transversely'') then $V_f\times_g W$ will be a manifold of dimension $\dim V+\dim W-\dim M$.  The tangent space to $V_f\times_g W$ at $(v,w)$ may be canonically identified (under the projection $T_vV\oplus T_{f(v)}M\oplus T_wW\to T_vV\oplus T_wW$) with the kernel of the map $h\co T_vV\oplus T_{f(v)}M\oplus T_wW\to T_{f(v)}M\oplus T_{f(v)} M$ defined by $h(e_V,e_M,e_W)=(f_*e_V-e_M,e_M-g_*e_W)$.  Under this identification, if orientations on $V,W,M$ are given, then $V {}_f\times_g W$ is oriented at $(v,w)$ by means of the short exact sequence \[ \xymatrix{0\ar[r] & \ker h\ar[r] & T_vV\oplus T_{f(v)}M\oplus T_{w}W\ar[r]^{h}& T_{f(v)}M\oplus T_{f(v)}M\ar[r]&0 }\]  

As noted in \cite{BC}, this fiber product orientation convention results in a number of pleasant properties.  First, a Cartesian product  $V\times W$ can be viewed as a fiber product $V{}_{*}\times_{*}W$ by taking the target space $M$ to be a positively-oriented point, and the resulting fiber product orientation on $V\times W$ coincides with the standard orientation.  Also, for a smooth map $f\co V\to M$ the fiber products $V{}_{f}\times_{1_M}M$ and $M{}_{1_M}\times_{f}V$ are identified with $V$ by projection, and the orientations on $V{}_{f}\times_{1_M}M$ and $M{}_{1_M}\times_{f}V$ are consistent with this orientation.  Less trivially, given suitably transverse maps $f\co U\to X$, $g_1\co V\to X$, $g_2\co V\to Y$, $h\co W\to Y$, one has: \begin{equation}\label{associd} (U {}_{f}\times_{g_1} V)_{g_2}\times_h W=U {}_f\times {}_{g_1}(V_{g_2}\times_{h}W) \end{equation} as oriented manifolds.  Also, if $V$ and $W$ are manifolds with boundary and $f\co V\to M$, $g\co W\to M$ are smooth maps such that $f(\partial V)\cap g(\partial W)=\varnothing$, and if all fiber products below are cut out transversely, one has, as oriented manifolds, \begin{equation} \label{bdryid} \partial(V {}_f\times_g W)=\left((\partial V) {}_f\times_g W\right)\coprod (-1)^{\dim M-\dim V}\left(V {}_f\times_g\partial W\right).\end{equation}

Moreover, again assuming $f\times g\co V\times W\to M\times M$ to be transverse to $\Delta$, the obvious diffeomorphism $(v,w)\mapsto (w,v)$ from $V\times W$ to $W\times V$ restricts as a diffeomorphism of oriented manifolds as follows: \begin{equation} \label{commutid} V{}_{f}\times_g W\cong (-1)^{(\dim M-\dim V)(\dim M-\dim W)}W{}_g\times_f V.\end{equation}

(The proofs of (\ref{associd}), (\ref{bdryid}), and (\ref{commutid}) can all be read off from \cite[A.1.8]{BC} and references therein.)

If $\delta\co M\to M\times M$ is the diagonal embedding, and if $V{}_{f}\times_g W$ is cut out transversely, then $V{}_{f}\times_g W$ is diffeomorphic by the map $(v,w)\mapsto (v,w,f(v))$ to the fiber product $(V\times W){}_{f\times g}\times_{\delta} M$.  This diffeomorphism affects the orientations by \begin{equation}\label{diagid} V{}_{f}\times_g W\cong (-1)^{\dim M(\dim M-\dim W)}(V\times W){}_{f\times g}\times_{\delta} M.\end{equation}  To see this, one can use the fact that the tangent space to $V{}_{f}\times_g W$ may be oriented as the kernel of the map $h_1\co TV\oplus TM\oplus TW\oplus TM\oplus TM\to TM\oplus TM\oplus TM\oplus TM$ defined by $h_1(v,m_0,w,m_1,m_2)=(f_*v-m_0,m_0-g_*w,m_1,m_2)$ while the tangent space to $(V\times W){}_{f\times g}\times_{\delta} M$ is oriented as the kernel of $h_2\co TV\oplus TW\oplus TM\oplus TM\oplus TM\to TM\oplus TM\oplus TM\oplus TM$ defined by $h_2(v,w,m_0,m_1,m_2)=(f_*v-m_0,g_*w-m_1,m_0-m_2,m_1-m_2)$.  There is a commutative diagram \[ \xymatrix{ TV\oplus TM\oplus TW\oplus TM\oplus TM \ar[r]^{h_1} \ar[d]^{\phi} & TM\oplus TM\oplus TM\oplus TM \ar[d]^{\psi} \\ TV\oplus TW\oplus TM\oplus TM\oplus TM \ar[r]^{h_2} & TM\oplus TM\oplus TM\oplus TM }\] where $\phi(v,m_0,w,m_1,m_2)=(v,w,m_0+m_1,m_0+m_2,m_0)$ and $\psi(m,m',m_1,m_2)=(m-m_1,-m'-m_2,m_1,m_2)$. The sign in (\ref{diagid}) is then obtained as the product of the signs of the determinants of $\phi$ and $\psi$.

If $f_0\co V^0\to M$, $f_1\co V^1\to N$, $g_0\co W^0\to M$, and $g_1\co W^1\to N$ are smooth maps such that the fiber products $V^0{}_{f_0}\times_{g_0}W^0$ and $V^1{}_{f_1}\times_{g_1} W^1$ are both cut out transversely, then the fiber product \[(V^0\times V^1)_{f_0\times f_1}\times_{g_0\times g_1}(W^0\times W^1)\] is also cut out transversely, and the map $(v_0,v_1,w_0,w_1)\mapsto (v_0,w_0,v_1,w_1)$ is a diffeomorphism of oriented manifolds \begin{equation} \label{prodid} (V^0\times V^1)_{f_0\times f_1}\times_{g_0\times g_1}(W^0\times W^1)\cong (-1)^{(\dim N-\dim V^1)(\dim M-\dim W^0)}(V^0{}_{f_0}\times_{g_0}W^0)\times(V^1{}_{f_1}\times_{g_1}W^1).\end{equation}  The sign can easily  be obtained either directly from the definition of the fiber product orientation, or by using \cite[(83)]{BC}.

\subsubsection{Signed counts of points and intersection numbers}
If $X$ is a $0$-dimensional manifold, then an orientation of $X$ of course amounts to a choice of number $\ep(x)\in \{-1,1\}$ attached to each point $x\in X$.  Assuming $X$ to be compact (\emph{i.e.}, finite) we write $\#(X)=\sum_{x\in X}\ep(x)$ and call $\#(X)$ the ``signed number of points'' of $X$.

If, where $V,W,M$ are smooth oriented manifolds, $f\co V\to M$ and $g\co W\to M$ are smooth maps such that $\dim V+\dim W=\dim M$ and such that the fiber product $W {}_g\times_{f} V$ is cut out transversely and is compact (and so is an oriented compact zero-manifold), then the \textit{intersection number} $\iota(f,g)$ of $f$ and $g$ is by definition \[ \iota(f,g)=\#\left(W {}_g\times_{f} V\right).\]  Note the reversal of the order of $f$ and $g$; this reversal is justified by the fact (noted in \cite{BC} and easily checked) that if $f$ and $g$ are embeddings of compact submanifolds $f(V)$ and $g(W)$ then the usual intersection number between $f(V)$ and $g(W)$ (given by counting intersections $m\in f(V)\cap g(W)$ with signs according to whether $T_mf(V)\oplus T_mg(W)$ has the same orientation as $T_mM$) is equal to $\iota(f,g)$ as we have just defined it.  Evidently by (\ref{commutid}) we have \[ \iota(f,g)=(-1)^{(\dim M-\dim V)(\dim M-\dim W)}\iota(g,f).\]


\subsection{Morse functions.}\label{or:morse} Let $f\co M\to \mathbb{R}$ be a Morse function where $M$ is a smooth oriented compact $n$-dimensional manifold and $h$ a Riemannian metric on $M$ making the negative gradient flow $\phi_t\co M\to M$ of $f$ Morse--Smale.  For all critical points $p$ of $f$ we have the unstable and stable manifolds \begin{align*} W^{u}_{f}(p)&=\left\{x\in M\left|\lim_{t\to -\infty}\phi_t(x)=p\right.\right\} \\
W^{s}_{f}(p)&=\left\{x\in M\left|\lim_{t\to \infty}\phi_t(x)=p\right.\right\} \end{align*}

We choose arbitrarily orientations of the unstable manifolds $W^{u}_{f}(p)$ (recall that these are diffeomorphic to open disks of dimension equal to the index $|p|_f$ of $p$), with the provisos that if $|p|_f=n$, so that $W^{u}_{f}(p)$ is an open subset of $M$, then the orientation of $W^{u}_{f}(p)$ should coincide with the orientation of $M$; and that if $|p|_f=0$, so that $W^{u}_{f}(p)=\{p\}$, then $W^{u}_{f}(p)$ should be oriented positively.  Let $i_{u,p}\co W^{u}_{f}(p)\to M$ and $i_{s,p}\co W^{s}_{f}(p)\to M$ be the inclusions.  Having oriented the $W^{u}_{f}(p)$, we orient the $W^{s}_{f}(p)$ by noting that $W^{u}_{f}(p)$ and $W^{s}_{f}(p)$ intersect transversely in the single point $p$, and requiring that \[ \iota(i_{s,p},i_{u,p})=1\] (in other words, $W^{u}_{f}(p){}_{i_{u,p}}\times_{i_{s,p}} W^{s}_{f}(p)$ is a single positively-oriented point).  

The space of parametrized negative gradient trajectories\footnote{For the most part we will use notation that suppresses the dependence of the trajectory space on the metric $h$; when we wish to record this dependence we will use the notation $\tilde{\mathcal{M}}(p,q;f,h)$.} $\tilde{\mathcal{M}}(p,q;f)$ from $p$ to $q$ may be identified with the fiber product $W^{u}_{f}(p){}_{i_{u,p}}\times_{i_{s,q}}W^{s}_{f}(q)$; the Morse--Smale condition precisely states that this fiber product is cut out transversely, and we orient $\tilde{\mathcal{M}}(p,q;f)$ by means of this identification, using the aforementioned convention for fiber products to orient $W^{u}_{f}(p){}_{i_{u,p}}\times_{i_{s,q}}W^{s}_{f}(q)$.  For $p\neq q$, the negative gradient flow provides a free $\mathbb{R}$-action on $\tilde{\mathcal{M}}(p,q;f)$.  We denote the quotient of $\tilde{\mathcal{M}}(p,q;f)$ by this $\mathbb{R}$ action by $\mathcal{M}(p,q;f)$, and we orient $\mathcal{M}(p,q;f)$ according to \ref{or:actions}.  In the case that $|p|_f-|q|_f=1$, the Morse--Smale condition implies that $\mathcal{M}(p,q;f)$ is a compact oriented zero-manifold, and we denote by \[ m_f(p,q)=\#\left(\mathcal{M}(p,q;f)\right)\] its signed number of points.

There are tautological identifications $W^{u}_{f}(p)\cong W^{s}_{-f}(p)$ and $W^{s}_{f}(p)\cong W^{u}_{-f}(p)$.   Having already oriented $W^{u}_{f}(p)$ and $W^{s}_{f}(p)$ as in the last two paragraphs, we first orient $W^{u}_{-f}(p)$ by requiring the tautological identification $W^{s}_{f}(p)\cong W^{u}_{-f}(p)$ to be orientation-preserving.  These orientations of $W^{u}_{-f}(p)$ then yield orientations of $W^{s}_{-f}(p)$ and of the spaces $\tilde{\mathcal{M}}(q,p;-f)$ and $\mathcal{M}(q,p;-f)$ by the same prescription as before.  Routine calculation then shows that the obvious identifications provide the following diffeomorphisms of oriented manifolds, where as usual we write $n=\dim M$: \begin{align}\label{su} W_{-f}^{s}(p)&\cong (-1)^{|p|_f(n-|p|_f)}W_{f}^{u}(p) \\ \label{tmqp} \tilde{\mathcal{M}}(q,p;-f)& \cong(-1)^{(|p|_f+|q|_f)(n-|p|_f)}\tilde{\mathcal{M}}(p,q;f) \\ \label{mqp} \mathcal{M}(q,p;-f)& \cong(-1)^{1+(|p|_f+|q|_f)(n-|p|_f)}\mathcal{M}(p,q;f) 
\end{align} (the last equation takes into account that the actions of $\mathbb{R}$ on $\tilde{\mathcal{M}}(p,q;f)$ and $\tilde{\mathcal{M}}(q,p;-f)$ go in opposite directions).

In the special case that $|p|_f=|q|_f+1$ we obtain \begin{equation}\label{dualm} m_{-f}(q,p)=(-1)^{n-|q|_f}m_f(p,q)=(-1)^{|q|_{-f}}m_f(p,q).\end{equation}

As described in \cite[Section 4]{S99} (see also \cite[A.1.14]{BC} for the relevant signs in the conventions that we are using), the unstable manifolds $W^u(x)$ admit partial compactifications $\bar{W}^u(p)$, whose oriented boundaries are given by\footnote{If one prefers, one could use a partial compactification with a larger boundary, namely $\coprod_{|r|_f\leq|p|_f-1}(-1)^{|p|_f-|r|_f-1}\mathcal{M}(p,r;f)\times W^u(r)$; however since the images in $M$ of those components corresponding to $|r|_f\leq |p|_f-2$ have codimension at least two we do not include them.  Similarly, as opposed to what is done below, $W^s(q)$ could be partially compactified to have the larger boundary $\coprod_{|r|_f\geq |q|_f+1}(-1)^{n-|q|_f}W^s(r)\times\mathcal{M}(r,q;f)$, and $\tilde{\mathcal{M}}(p,q;f)$ could be given the larger boundary $\left(\coprod_{|r|_{f}\leq |p|_f-1}(-1)^{|p|_f-|r|_f-1}\mathcal{M}(p,r;f)\times\tilde{\mathcal{M}}(r,q;f)\right)\sqcup (-1)^{|p|_f+|q|_f}
\left(\coprod_{|r|_{f}\geq |q|_f+1}\tilde{\mathcal{M}}(p,r;f)\times\mathcal{M}(r,q;f)\right)$.}  \begin{equation}\label{delu} \partial\bar{W}^u(p)=\coprod_{|r|_f=|p|_f-1}\mathcal{M}(p,r;f)\times W^u(r).\end{equation}  Extending the embedding $i_{u,p}\co W^u(p)\to M$ to a map on all of $\bar{W}^u(p)$ by means of the embeddings $i_{u,r}$ of the $W^u(r)$, we obtain a smooth map $\bar{i}_{u,p}\co \bar{W}^u(p)\to M$ which is a ``pseudochain'' in the sense to be described later: essentially this means that its image may be compactified by adding sets of codimension at least two (namely the unstable manifolds of some other critical points of index at most $|p|_f-2$).

Likewise, one obtains pseudochains whose domains will be denoted $\bar{W}^s(q)$ and $\tilde{\bar{\mathcal{M}}}(p,q;f)$ which partially compactify the stable manifolds and the parametrized gradient trajectory spaces, respectively.  By using the various formulas and conventions specified above (in particular using that $W^{s}_{f}(q)=W^{u}_{-f}(q)$, so that the boundary orientation of $W^s(q)$ can be deduced from (\ref{delu})), one obtains that the oriented (codimension-one) boundaries of 
the domains of these pseudochains are given by:\small \begin{align}\label{delws}
 \partial\bar{W}^s(q)&=\coprod_{|r|_f= |q|_f+1}(-1)^{n-|q|_f}W^s(r)\times\mathcal{M}(r,q;f)
\\ \nonumber \partial \tilde{\bar{\mathcal{M}}}(p,q;f)&=\left(\coprod_{|r|_{f}=|p|_f-1}\mathcal{M}(p,r;f)\times\tilde{\mathcal{M}}(r,q;f)\right)\sqcup (-1)^{|p|_f+|q|_f} 
\left(\coprod_{|r|_{f}=|q|_f+1}\tilde{\mathcal{M}}(p,r;f)\times\mathcal{M}(r,q;f)\right),\end{align}\normalsize and these boundaries are mapped into $M$ by using the inclusions of $W^s(r)$ in the case of $\partial\bar{W}^s(q)$ and by using the inclusions of $\tilde{\mathcal{M}}(r,q;f)$ and  $\tilde{\mathcal{M}}(p,r;f)$ in the case of $\partial \tilde{\bar{\mathcal{M}}}(p,q;f)$.

\section{Linking of pseudoboundaries}\label{pseudolink}

The appropriate level of generality for the consideration of linking numbers in this paper seems to be given by some natural extensions of the formalism of pseudocycles, as described in \cite[Section 6.5]{MS}.  Given a smooth map $f\co V\to M$, where $V$ is a smooth manifold (possibly with boundary) and $M$ is a smooth manifold without boundary, recall that the $\Omega$-limit set of $f$ is by definition \[ \Omega_f=\bigcap_{A\Subset V}\overline{f(V\setminus A)},\] where the notation $\Subset$ means ``is a compact subset of.'' As can easily be checked, one has \[ \overline{f(V)}=f(V)\cup \Omega_f.\] If $S\subset M$ is any  subset, $S$ is said to have ``dimension at most $d$'' if there is a smooth map $g\co W\to M$ such that $S\subset g(W)$ where $W$ is a smooth manifold all of whose components have dimension at most $d$.

\begin{definition}\label{pseudodef} Let $V$ and $M$ be smooth oriented manifolds, where $V$ might have boundary and $\dim V=k$, and let $f\co V\to M$ be a smooth map.
\begin{itemize} \item[(i)] $f$ is called a \emph{$k$-pseudochain} if $\overline{f(V)}$ is compact and $\Omega_f$ has dimension at most $k-2$.
\item[(ii)] $f\co V\to M$ is called a \emph{$k$-pseudocycle} if $f$ is a $k$-pseudochain and $\partial V=\varnothing$.
\item[(iii)] $f\co V\to M$ is called a \emph{$k$-pseudoboundary} if $f$ is a $k$-pseudocycle and there is a $(k+1)$-pseudochain $g\co W\to M$ such that $\partial W=V$ as oriented manifolds and $g|_{\partial W}=f$.  In this case the pseudochain $g$ is called a \emph{bounding pseudochain} for $f$.
\end{itemize}
\end{definition}

In the above definition we have required $V$ to be oriented.  Deleting all references to orientation gives in the obvious way definitions of ``unoriented pseudochains, pseudocycles, and pseudoboundaries;'' in the unoriented situation one may straightforwardly modify the following discussion to obtain intersection and linking numbers which are defined modulo $2$. We remark that the restriction to the boundary of a pseudochain will not necessarily be a pseudoboundary, since the $\Omega$-limit set of the restriction might have codimension one in the boundary.

As explained in \cite{Z}, a pseudocycle naturally determines a homology class in $M$, in a way which induces an isomorphism between the group $\mathcal{H}_*(X)$ of  pseudocycles modulo  pseudoboundaries (with addition given by disjoint union) and the integral homology $H_*(M;\mathbb{Z})$.  Moreover, there is a well-defined intersection pairing on $\mathcal{H}_{*}(X)$ given by the construction of \cite[p. 161]{MS}, and under the isomorphism $\mathcal{H}_*(X)\cong H_*(X;\mathbb{Z})$ this corresponds to the standard intersection pairing.  

Essentially the same construction as was used for the intersection pairing on $\mathcal{H}_*(X)$ in \cite{MS} may be used to define linking numbers between pseudoboundaries, as we now describe.

We begin with a technical transversality result.  

\begin{lemma} \label{diffu}  $M,N,Y$ be smooth manifolds,  let $f\co M\to Y$, $g\co N\to Y$ be smooth functions, and let $S$ be a compact subset of $Y$ such that, for every pair $(m,n)\in M\times N$ such that $f(m)=g(n)$ and $(f\times g)_*\co T_m M\times T_n N\to T_{(f(m),f(m))}Y\times Y$ is not transverse to $\Delta$, it holds that $f(m)\in int(S)$ (where $int(S)$ denotes the interior of $S$).  Let  $\Diff_{S}(Y)$ denote the space of diffeomorphisms of $Y$ having support contained in $S$, equipped with the (restriction of the) Whitney $C^{\infty}$ topology.  Then \[ \mathcal{S}=\left\{\phi\in \Diff_{S}(Y)\left|\left( (\phi\circ f)\times g\right)\co M\times N\to Y\times Y \mbox{ is transverse to }\Delta       \right.\right\}
\] is a residual subset of $\Diff_{S}(Y)$.
\end{lemma}

\begin{proof}
See Section \ref{app}.
\end{proof}

The following consequence of Lemma \ref{diffu} is a small generalization of \cite[Lemma 6.5.5(i)]{MS}.

\begin{prop}\label{tvs1} Let $F_0\co X\to M$ be a pseudochain (where $X$ is a smooth manifold with boundary), and let $g\co W\to M$ be a pseudocycle such that $\overline{F_0(\partial X)}\cap \overline{g(W)}=\varnothing$.  Then if $\mathcal{U}$ is any neighborhood of $F_0$ in the Whitney $C^{\infty}$ topology there exists a pseudochain $F\co X\to M$ such that $F\in \mathcal{U}$ and \begin{itemize}  \item[(i)] $F|_{\partial X}=F_{0}|_{\partial X}$ \item[(ii)] $(F\times g)\co X\times W\to M\times M$ is transverse to the diagonal $\Delta$. \item[(iii)] $\Omega_F\cap \overline{g(W)}$ and  $\overline{F(X)}\cap \Omega_g$ both have dimension at most $\dim X+\dim W-\dim M-2$.\end{itemize}
\end{prop}

\begin{proof} Write $\dim X=k$, $\dim W=l$, and $\dim M=n$.  There are smooth maps $\alpha\co A\to M$ and $\beta\co B\to M$, where $A$ and $B$ are smooth manifolds whose components all have dimension at most $k-2$ and at most $l-2$ respectively, such that $\Omega_{F_0}\subset \alpha(A)$ and $\Omega_g\subset \beta(B)$.

 Since $\overline{F_0(\partial X)}\cap \overline{g(W)}=\varnothing$, and since $\overline{F_0(X)}\cap\overline{g(W)}$ is compact, we can find an open set $U\subset M$ containing $\overline{F_0(X)}\cap\overline{g(W)}$ and whose closure is disjoint from $\overline{F_0(\partial X)}$.  
According to repeated applications of Lemma \ref{diffu} the following subsets of the group $\Diff_{\bar{U}}(M)$ of diffeomorphisms with support in $\bar{U}$ are all residual in the $C^{\infty}$ topology: \begin{align*}
\mathcal{U}_1&=\left\{\phi\in \Diff_{\bar{U}}(M)\left| \left((\phi\circ F_0)\times g    \right)\co X\times W\to M\times M\mbox{ is transverse to $\Delta$}\right.\right\} \\
\mathcal{U}_2&=\left\{\phi\in \Diff_{\bar{U}}(M)\left| \left((\phi\circ \alpha)\times g    \right)\co A\times W\to M\times M\mbox{ is transverse to $\Delta$}\right.\right\} \\
\mathcal{U}_3&=\left\{\phi\in \Diff_{\bar{U}}(M)\left| \left((\phi\circ \alpha)\times \beta    \right)\co A\times B\to M\times M\mbox{ is transverse to $\Delta$}\right.\right\}.\\
\mathcal{U}_4&=\left\{\phi\in \Diff_{\bar{U}}(M)\left| \left(F_0 \times (\phi\circ \beta)    \right)\co X\times B\to M\times M\mbox{ is transverse to $\Delta$}\right.\right\} \end{align*}

In particular we can find a diffeomorphism $\phi$, arbitrarily $C^{\infty}$-close to the identity and supported in $\bar{U}$, such that $\phi\in \mathcal{U}_1\cap\mathcal{U}_2\cap\mathcal{U}_3$ and $\phi^{-1}\in\mathcal{U}_4$.   We claim that $F=\phi\circ F_0$ will have the desired properties.  

Since $F_0(\partial X)\cap (supp(\phi))=\varnothing$ property (i) of the proposition is clear.

The fact that $\phi\in\mathcal{U}_1$ immediately implies property (ii).

As for property (iii), since $\overline{F(X)}=F(X)\cup \Omega_F$ and $\overline{g(W)}=g(W)\cup \Omega_g$,  we need to show that $\Omega_F\cap g(W)$, $\Omega_F\cap \Omega_g$, and $F(X)\cap \Omega_g$ all have dimension at most $k+l-n-2$.  Now $\Omega_F\cap g(W)\subset (\phi\circ\alpha)(A)\cap g(W)$, and the fact that $\phi\in \mathcal{U}_2$ shows that $A{}_{\phi\circ \alpha}\times_g W$ is cut out transversely, so since all components of $A$ have dimension at most $k-2$ we see that $\Omega_F\cap g(W)$ has dimension at most $k+l-n-2$.  Similarly the fact that $\phi\in \mathcal{U}_3$ implies that $\Omega_F\cap \Omega_g$ has dimension at most $k+l-n-4$.  Finally, note that \[ F(X)\cap \Omega_g\subset F(X)\cap \beta(B)=\phi\left(F_0(X)\cap (\phi^{-1}\circ\beta)(B)\right),\] so the fact that $\phi^{-1}\in \mathcal{U}_4$ implies that $F(X)\cap \Omega_g$ has dimension at most $k+l-n-2$, completing the proof.
\end{proof}

Assume that the target manifold $M$ is oriented with $\dim M=n$, and let $f\co V\to M$ be a $k$-pseudoboundary and $g\co W\to M$ a $(n-k-1)$-pseudoboundary, such that $\overline{f(V)}\cap \overline{g(W)}=\varnothing$, and let $F_0\co X\to M$ be a bounding pseudochain for $f$.  Use Proposition \ref{tvs1} to perturb $F_0$ to $F\co X\to M$ obeying (i)-(iii) above; in particular $F$ is also a bounding pseudochain for $f$. The fiber product $X{}_{F}\times_g W$ is then a smooth oriented manifold of dimension zero, with $\Omega_F\cap \overline{g(W)}=F(X)\cap \Omega_g=\varnothing$  (since in this case $\dim X+\dim W-\dim M-2=-2$).  

Moreover $X{}_F\times_g W$ is compact: if $\{(x_n,w_n)\}$ is a sequence in $X{}_F\times_g W$ then since $\overline{F(X)}$ and $\overline{g(W)}$   are compact the sequence $\{(F(x_n),g(w_n))\}$ would have a subsequence (still denoted by \linebreak $\{(F(x_n),g(w_n))\}$) converging to a point $(m,m)\in (\overline{F(X)}\times \overline{g(W)})\cap \Delta$.  Now $\overline{F(X)}=F(X)\cup \Omega_F$, so since $\Omega_F\cap \overline{g(W)}=\varnothing$ we must have $m\in F(X)\setminus \Omega_F$.  But since $F(X)\cap \Omega_g=\varnothing$ this implies that also $m\in g(W)\setminus \Omega_g$.    Since $m$ lies in neither $\Omega_F$ nor $\Omega_g$ there are compact sets $K\Subset X$, $L\Subset W$ such that $m\notin \overline{F(X\setminus K)}$ and $m\notin \overline{g(W\setminus L)}$.  So since $F(x_n)\to m$ and $g(w_n)\to m$, infinitely many of the $x_n$ lie in $K$, and infinitely many of the $w_n$ lie in $L$.  So since $K$ and $L$ are compact a subsequence of $\{(x_n,w_n)\}$ converges to a pair $(x,w)\in K\times L\subset X\times W$ such that $F(x)=g(w)$, \emph{i.e.} such that $(x,w)\in F{}_X\times_g W$.  This confirms our assertion that $X{}_F\times_g W$ is compact provided that $F$ is as in Proposition \ref{tvs1}.

Since $X{}_F\times_g W$ is a compact oriented zero-manifold we can take the intersection number $\iota(g,F)=\#(X{}_F\times_g W)$ as described at the end of Section \ref{or}.  We would like to define the \textit{linking number} of the pseudoboundaries $g$ and $f$ to be equal to this intersection number; the justification of this definition requires the following:

\begin{prop}\label{indep}  Let $f\co V\to M$ and $g\co W\to M$ be two pseudoboundaries such that $\overline{f(V)}\cap \overline{g(W)}=\varnothing$ and $\dim V+\dim W+1=\dim M$.  Let $F_1\co X_1\to M$ and $F_2\co X_2\to M$ be two bounding pseudochains for $f$ such that, for $i=1,2$, \begin{itemize} \item[(i)] $F_i\times g\co X_i\times W\to M\times M$ is transverse to the diagonal $\Delta$.
\item[(ii)] $\Omega_{F_i}\cap \overline{g(W)}=\Omega_g\cap \overline{F_i(X_i)}=\varnothing$.
\end{itemize}

Then \[ \iota(g,F_1)=\iota(g,F_2).\] 
\end{prop}

(Of course, the argument before the proposition shows that (i) and (ii) suffice to guarantee that $X_i\,{}_{F_i}\times_g W$ is a compact oriented zero-manifold, so that $\iota(g,F_i)$ is well-defined.)

\begin{proof} We have, as oriented manifolds, $\partial X_i=V$ and $F_i|_{\partial X_i}=f$. Let $\bar{X}_1$ denote $X_1$ with its orientation reversed.  There are then neighborhoods $U_1$ of $\partial\bar{X}_1$ in $\bar{X}_1$ and $U_2$ of $\partial X_2$ in $X_2$, and orientation-preserving diffeomorphisms $\phi_1\co [0,\infty)\times V\to U_1$ and $\phi_2\co (-\infty,0]\times V\to U_2$ which restrict as the identity on $\{0\}\times V=\partial X_i$.  Gluing $\bar{X}_1$ to $X_2$ along their common boundary $V$ results in a new oriented, boundaryless manifold $X$, with an open subset $U\subset X$ which is identified via a diffeomorphism $\phi=\phi_1\cup \phi_2$ with $\mathbb{R}\times V$.  

Define a map $G_0\co X\to M$ by the requirement that $G_0|_{X_i}=F_i$.  Now $G_0$ is typically not a smooth map (its derivative in the direction normal to $\{0\}\times V$ will typically not exist), but this is easily remedied: let $\beta\co \mathbb{R}\to\mathbb{R}$ be a smooth homeomorphism such that $\beta(t)=t$ for $|t|>1$ and such that $\beta$ vanishes to infinite order at $t=0$.  Where $U\subset X$ is identified with $\mathbb{R}\times V$ as above, define $\Phi\co X\to X$ by setting $\Phi(t,v)=\left(\beta(t),v\right)$ for $(t,v)\in U$ and setting $\Phi$ equal to the identity outside $U$.  Then $\Phi$ is a smooth homeomorphism, and the function $G:=G\circ \Phi$ will now be smooth, since the normal derivatives to all orders along $\{0\}\times V$ will simply vanish.  

Now since $G=G_0\circ \Phi^{-1}$ where $\Phi^{-1}$ is a homeomorphism we have $\Omega_G=\Omega_{G_0}$.  But it is easy to check from the definitions that $\Omega_{G_0}=\Omega_{F_1}\cup \Omega_{F_2}$.  Thus $\Omega_G$, like $\Omega_{F_1}$ and $\Omega_{F_2}$, has dimension at most $\dim X-2$.  So since $\partial X=\varnothing$, $G\co X\to M$ is a pseudocycle.  Moreover we have \[ \Omega_G\cap \overline{g(W)}=(\Omega_{F_1}\cap \overline{g(W)})\cup(\Omega_{F_1}\cap \overline{g(W)})=\varnothing,\] and since $G(X)=F_1(X_1)\cup F_2(X_2)$, \[ \Omega_g\cap \overline{G(X)}=\varnothing.\]  

Furthermore, viewing the $X_i$ as  submanifolds-with-boundary of $X$ (with the orientation of $X_1$ reversed) and using that the image under $G$ of $V=X_1\cap X_2$ is disjoint from $g(W)$, we have, as oriented manifolds, \[ X{}_G\times_g W=\left(-X_1\,{}_{F_1}\times_{g}W\right)\coprod \left(X_2\,{}_{F_2}\times_{g}W\right).\]  In particular the fiber product $X{}_{G}\times_g W$ is cut out transversely, and the intersection numbers of $G,F_1,F_2$ with $g$ obey \[ \iota(g,G)=-\iota(g,F_1)+\iota(g,F_2).\]  But $G\co X\to M$ is a pseudocycle and $g\co W\to M$ is a pseudoboundary, so by \cite[Lemma 6.5.5 (iii)]{MS} one has $\iota(g,G)=0$, and so $\iota(g,F_1)=\iota(g,F_2)$.
\end{proof}

We can accordingly make the following definition:

\begin{definition}\label{lk} Let $M$ be an oriented $n$-manifold and let $f\co V\to M$ and $g\co W\to M$ be pseudoboundaries of dimension $k$ and $n-k-1$ respectively such that $\overline{f(V)}\cap \overline{g(W)}=\varnothing$.  Then the \emph{linking number} of $g$ and $f$ is \[ lk(g,f)=\#\left(X{}_{F}\times_g W\right) \] where $F\co X\to M$ is any bounding pseudochain for $f$ such that $F\times g\co X\times W\to M\times M$ is transverse to $\Delta$, and $\Omega_g\cap \overline{F(X)}=\Omega_F\cap \overline{g(W)}=\varnothing$.
\end{definition}

Of course, the existence of such an $F$ is implied by Proposition \ref{tvs1}, and the independence of $lk(g,f)$ from the choice of $F$ is given by Proposition \ref{indep}.  Moreover: 

\begin{prop}\label{linksym} For $f\co V\to M$ and $g\co W\to M$ as in Definition \ref{lk} we have \[ lk(g,f)=(-1)^{(k+1)(n-k)}lk(f,g).\]
\end{prop}

\begin{proof} Let $F_1\co X\to M$ and $G_1\co Y\to M$ be bounding pseudochains for $f$ and $g$ respectively, such that $F_1\times g\co X\times W\to M\times M$ and $G_1\times f\co Y\times V\to M\times M$ are transverse to $\Delta$, and such that \[ \Omega_g\cap\overline{F_1(X)}=
\Omega_f\cap\overline{G_1(Y)}=\Omega_{F_1}\cap\overline{g(W)}=\Omega_{G_1}\cap\overline{f(V)}=\varnothing.\]

First, using repeated applications of Lemma \ref{diffu}, one can perturb $F_1$ and $G_1$ to maps $F\co X\to M$ and $G\co Y\to M$ which (in addition to the above properties of $F_1$ and $G_1$) also have the properties that $F\times G\co X\times Y\to M\times M$ is transverse to $\Delta$, and $\Omega_F\cap \overline{G(Y)}=\Omega_G\cap \overline{F(X)}=\varnothing$  (More specifically, and ignoring issues related to the $\Omega$-limit sets which can be handled as in the proof of Proposition \ref{tvs1}, first apply Lemma \ref{diffu} with one map equal to $G_1$ and the other equal to $F_1|_{F_{1}^{-1}(U)}$ for some small neighborhood $U$ of $\overline{f(V)}$ to perturb $G_1$ to a new map $G_2$ which has no nontransverse intersections with $F_1$ or $f$ near $\overline{f(V)}$.  Then similarly perturb $F_1$ to $F_2$ which has no nontransverse intersections with $G_2$ or $g$ near $\overline{g(W)}$.  Then finally apply Lemma \ref{diffu} to $F_2$ and $G_2$ on a suitable compact subset $S$ which is disjoint from $\overline{f(V)}\cup\overline{g(W)}$.  We leave the details to the reader.)

The fiber product $X{}_{F}\times_G Y$ will then be an oriented compact one-manifold with oriented boundary given by, according to (\ref{bdryid}) and (\ref{commutid}), \begin{align*} 
\partial\left(X{}_{F}\times_G Y\right)&=\left(V{}_f\times_G Y\right)\coprod (-1)^{n-k-1}\left(X{}_F\times_g W\right)
\\&\cong (-1)^{k(n-k)}\left(Y{}_G\times_f V\right)\coprod (-1)^{n-k-1}\left(X{}_F\times_g W\right).\end{align*}  So the signed number of points of the boundary the oriented compact one-manifold $X{}_F\times_G Y$ is equal to \[ (-1)^{k(n-k)}lk(f,g)+(-1)^{n-k-1}lk(g,f).\]  But the signed number of points of the boundary of any oriented compact one-manifold is zero, and setting the above expression equal to zero yields the result.
\end{proof}

While we primarily consider pseudochains and pseudoboundaries in this paper, it is natural to ask when these can be replaced by smooth maps defined on compact smooth manifolds.  The following lemma helps to answer this question in some cases.

\begin{lemma} \label{smoothen}
Let $\phi\co V\to M$ be a $k$-pseudoboundary and let $U$ be any open neighborhood of $\overline{\phi(V)}$.  Then for some positive integer $N$, there is a compact oriented $k$-manifold $B$ and a smooth map $f\co B\to M$ which is a pseudoboundary, such that $f(B)\subset U$ and such that, for every $(n-k-1)$-pseudoboundary $g\co W\to M$ such that $g(W)\cap U=\varnothing$, we have \[ lk(g,f)=N\,lk(g,\phi).\]
\end{lemma}

\begin{proof} 
Choose an open subset $U_1\subset U$ such that $\bar{U}_1$ is a smooth compact manifold with boundary and $\overline{\phi(V)}\subset U_1\subset \bar{U}_1\subset U$ (for instance, $\bar{U}_1$ could be taken as a regular sublevel set for some smooth function supported in $U$ and equal to $-1$ on $\overline{\phi(V)}$).  Let $C_2$ be the image of $\bar{U}_1$ under the time-one flow of some vector field that points strictly into $U_1$ along $\partial \bar{U}_1$ and vanishes on $\overline{\phi(V)}$, so in particular $C_2$ is a smooth compact manifold with boundary and we have $\overline{\phi(V)}\subset C_2\subset U_1$, with the inclusion $i\co C_2\to U_1$ a homotopy equivalence.

Let $[\phi]\in H_k(U_1;\Z)$ denote the homology class of $\phi$, as given by the isomorphism $\Phi$ from \cite[Theorem 1.1]{Z} between the homology of $U_1$ and the group of equivalence classes of pseudocycles in $U_1$. Since $C_2$, like any smooth compact  manifold with boundary, is homeomorphic to a finite polyhedron, \cite[Th\'eor\`eme III.4]{T} gives a positive integer $N$, a smooth compact oriented $k$-manifold $B$ without boundary, and a continuous map $f^{0}\co B\to C_2$ such that $(f^{0})_*[B]=Ni_{*}^{-1}[\phi]$.  So if $f\co B\to U_1$ is a small perturbation of $f^0$ which is of class $C^{\infty}$, then  $f_{*}[B]=N[\phi]\in H_k(U_1;\Z)$.  

We can now think of $f$ as a pseudocycle in $U_1$; as is clear from the construction of the isomorphism $\Phi$ in \cite[Section 3.2]{Z}, the homology class determined by $f$ under $\Phi$ is just $f_{*}[B]$.  Let $NV$ denote the oriented manifold obtained by taking $N$ disjoint copies of $V$, and let $\phi^{N}\co NV\to U_1\subset M$ be the pseudocycle equal to $\phi$ on each copy of $V$.  The injectivity of Zinger's isomorphism $\Phi$ shows that $f$ and $\phi^{N}$ are equivalent as pseudocycles in $U_1$, \emph{i.e.}, there is an oriented manifold with boundary $X_1$ with $\partial X_1=B\coprod (-NV)$ and a pseudochain $F_1\co X_1\to U_1$ such that $F_1|_{B}=f$ and $F_1|_{-NV}=\phi^{N}$.  

Now $\phi\co V\to M$ was assumed to be a pseudoboundary, so taking $N$ copies of a bounding pseudochain for $\phi$ gives a bounding pseudochain $F_2\co X_2\to M$ for $\phi^{N}\co NV\to M$.  A gluing construction just like the one in the second paragraph of the proof of   
Proposition \ref{indep} then gives a bounding pseudochain $F\co X\to M$ for $f$, where $X$ is the smooth manifold resulting from gluing $X_1$ and $X_2$ along $NV$.  In particular this shows that $f$ is a pseudoboundary in $M$.  Moreover, since $\bar{U}_1\subset U$ the gluing construction can be arranged in such a way that $F^{-1}(M\setminus U)=F_{2}^{-1}(M\setminus U)$, and $F|_{F^{-1}(M\setminus U)}=F_{2}|_{F_{2}^{-1}(M\setminus U)}$.  So if $g\co W\to M$ is any pseudoboundary such that $g(W)\subset M\setminus U$, then we have \[ lk(g,f)=\#(X{}_{F}\times_g W)=\#\left(X_2\,{}_{F_2}\times_g W\right)=N\,lk(g,\phi),\] since $F_2\co X_2\to M$ was obtained by taking $N$ copies of a bounding pseudochain for $\phi$.  
\end{proof}

\section{Operations on the Morse complex}\label{ops}

Let $M$ be a compact smooth oriented $n$-manifold and let $f\co M\to\R$ be a Morse function, and fix a coefficient ring $\mathbb{K}$.  We will work with respect to a metric $h$ which belongs to the intersection of the  residual sets given by applying the forthcoming Proposition \ref{genmet} to various maps into $M$ that will be specified later; in particular, the negative gradient flow of $f$ with respect to $h$ is Morse--Smale.  Let $Crit(f)$ denote the set of critical points of $f$, and for $p\in Crit(f)$ write $|p|_f$ for the index of $p$.  As in Section \ref{or:morse}, orient the unstable manifolds $W^{u}_{f}(p)$ in such a way that when $|p|_f=n$ (so that $W^{u}_{f}(p)$ is an open subset of $M$) the orientation of $W^{u}_{f}(p)$ agrees with that of $M$, and when $|p|_f=0$, $W^{u}_{f}(p)$ is a positively oriented point.  This then induces orientations of the various $W^{s}_{f}(p)$, $W^{u}_{-f}(p)$, $W^{s}_{-f}(p)$, $\tilde{\mathcal{M}}(p,q;f)$, and $\mathcal{M}(p,q;f)$ as prescribed in Section \ref{or:morse}.  Note that these prescriptions also ensure that when $|p|_{-f}=n$ (so $|p|_f=0$) the orientation of $W^{u}_{-f}(p)$ agrees with that of $M$, and when $|p|_{-f}=0$, $W^{u}_{-f}(p)$ is a positively oriented point.

  When $|p|_f=|q|_f+1$, the Morse--Smale condition ensures that $\mathcal{M}(p,q;f)$ is a compact $0$-dimensional oriented manifold, and so has a signed number of points $\#_{\mathbb{K}}\left(\mathcal{M}(p,q;f)\right)$, evaluated in $\mathbb{K}$ (using the unique unital ring homomorphism $\Z\to\mathbb{K}$).

The Morse complex $(CM_*(f;\mathbb{K}),d_f)$ is defined as usual by letting $CM_k(f)$ be the free $\mathbb{K}$-module generated by the index-$k$ critical points of $f$, setting $CM_{*}(f;\mathbb{K})=\oplus_{k=0}^{n}CM_k(f)$, and defining $d_f=\oplus_k d_{f,k}$ where $d_{f,k}\co CM_{k}(f;\mathbb{K})\to CM_{k-1}(f;\mathbb{K})$ is defined by extending linearly from, for $p\in Crit(f)$ with $|p|_f=k$, \[ d_{f,k}(p)=\sum_{\scriptsize{\begin{array}{c} q\in Crit(f):\\|q|_f=k-1\end{array}}} \#_{\mathbb{K}}\left(\mathcal{M}(p,q;f)\right)q.\]  As is familiar (see \emph{e.g.} \cite{S93}), one has $d_f\circ d_f=0$, and the resulting homology $HM_*(f;\mathbb{K})$ is canonically isomorphic to the singular homology $H_*(M;\mathbb{K})$ of $M$ with coefficients in $\mathbb{K}$.

Moreover, given our orientation conventions, there is a canonical element $M_f\in CM_n(f;\mathbb{K})$, defined by \begin{equation} \label{mdef} M_f=\sum_{\scriptsize{\begin{array}{c} p\in Crit(f):\\|p|_f=n\end{array}}}p.\end{equation} The following shows that $M_f$ is a cycle in the Morse chain complex; in view of this, it is easy to see that $M_f$ represents the fundamental class of $M$ under the isomorphism with singular homology (using for instance the construction of this isomorphism given in \cite{S99}).

\begin{prop}\label{dmzero} $d_f M_f=0$.  
\end{prop}

\begin{proof} It suffices to show that, for any $q\in Crit(f)$ with $|q|_f=n-1$, the coefficient on $q$ in $d_fM_f$ is equal to zero.  
This coefficient is equal to \[ \sum_{\scriptsize{\begin{array}{c} p\in Crit(f):\\|p|_f=n\end{array}}}\#_{\mathbb{K}}\left(\mathcal{M}(p,q;f)\right).\]  

Now since $|q|_{f}=n-1$, the Morse--Smale condition (and the trivial fact that no critical points of $f$ have index larger than $n$) implies that the partial compactification $\bar{W}^{s}_{f}(q)$ described in Section \ref{or:morse} is in fact compact.  Consequently $\#_{\mathbb{K}}\left(\partial \bar{W}^{s}_{f}(q)\right)=0$, since the signed number of points in the boundary of a compact one-manifold is always zero.  Also, our orientation conventions ensure that, for each critical point $p$ with $|p|_f=n$, $W^{s}_{f}(p)$ is a positively-oriented point.  So consulting (\ref{delws}) we obtain \[ 0=\#_{\mathbb{K}}\left(\partial \bar{W}^{s}_{f}(q)\right)=-\sum_{\scriptsize{\begin{array}{c} p\in Crit(f):\\|p|_f=n\end{array}}}\#_{\mathbb{K}}\left(\mathcal{M}(p,q;f)\right),\] as desired.
\end{proof}

Of course, this can all be done with respect to $-f$ in place of $f$, and with the prescriptions above the $\mathbb{K}$-modules $CM_{n-k}(-f;\mathbb{K})$ and $CM_k(f;\mathbb{K})$ are defined identically.  We may then define a $\mathbb{K}$-bilinear pairing \begin{align*} \Pi\co CM_{n-*}(-f;\mathbb{K})\times CM_{*}(f;\mathbb{K})&\to \mathbb{K} \\ \left(\sum_{q\in Crit(f)}a_q q,\sum_{p\in Crit(f)}b_p p\right)&\mapsto \sum_{p\in Crit(f)}a_p b_p \end{align*}

Equation (\ref{dualm}) then translates to \begin{equation}\label{adjpi} \Pi\left(d_{-f}x,y\right)=(-1)^{n-k+1}\Pi\left(x,d_f y\right)     \quad \mbox{for } x\in CM_{n-k+1}(-f;\mathbb{K}),\,y\in CM_k(f;\mathbb{K}),\end{equation} so that $\Pi$ descends to a pairing $\underline{\Pi}\co HM_*(-f;\mathbb{K})\times HM_*(f;\mathbb{K})\to\K$. Given that the Morse--Smale condition guarantees that if $p$ and $q$ are distinct critical points of the same index then $W^{s}(q)\cap W^{u}(p)=\varnothing$,  it is easy to check that, with respect to the identifications of $HM_{*}(\pm f;\K)$ with $H_*(M;\K)$ described in \cite{S99}, this homological pairing coincides with the standard intersection pairing on $M$ (recall from Section \ref{or:morse} that for $p\in Crit(f)$ the direct sum decomposition $T_pM=T_{p}W^{u}_{-f}(p)\oplus T_{p}W^{u}_{f}(p)$ respects the orientations, in view of which $\underline{\Pi}$ has the correct sign to agree with the standard intersection pairing).

From the pairing $\Pi$ we may construct a \emph{linking pairing} between the image of $d_{-f}$ and the image of $d_f$: \begin{align} \Lambda\co \left(Im(d_{-f})\right)\times \left(Im(d_f)\right)&\to\mathbb{K} \nonumber \\ \label{mlpair}
(x,y)&\mapsto \Pi(x,z) \quad \mbox{for any $z$ such that $d_f z=y$}.\end{align}  The adjoint relation (\ref{adjpi}) and the fact that $x\in Im(d_{-f})$ readily imply that $\Pi(x,z)$ is indeed independent of the choice of $z$ such that $d_f z=y$.   Also, for $x\in CM_{n-k-1}(-f;\mathbb{K})\cap Im(d_{-f})$ and $y\in CM_k(f;\mathbb{K})\cap Im(d_f)$ the above definition is equivalent to \begin{equation}\label{altlam} \Lambda(x,y)=(-1)^{n-k}\Pi(w,y)\quad \mbox{for any $w$ such that $d_{-f}w=x$}.\end{equation}

We now turn to a transversality result for intersections of Morse trajectories with smooth maps, which, while following from fairly standard methods, will be of fundamental importance for our operations on the Morse chain complex.  Be given an exhausting Morse function $f\co M\to \R$ on an $n$-dimensional smooth manifold $m$, and let $Crit(f)$ denote the set of critical points of $f$.  If $h$ is a Riemannian metric and $p,q\in Crit(f)$ we have the inclusions of the stable and unstable manifolds $i_{s,q}\co W^{s}_{f}(q;h)\to M$, $i_{u,p}\co W^{u}_{f}(p;h)\to M$ and the trajectory space $\tilde{\mathcal{M}}(p,q;f,h)=W^{u}_{f}(p;h){}_{i_{u,p}}\times_{i_{s,p}} W^{s}_{f}(q;h)$  In \cite[Section 2.3]{S93} Schwarz constructs a Banach manifold $\mathcal{G}$ all of whose members are smooth Riemannian metrics, and shows that there is a residual subset $\mathcal{R}_0\subset\mathcal{G}$ such that for all $h\in\mathcal{R}_0$ the negative gradient flow of $f$ with respect to $h$ satisfies the Morse--Smale condition, which is to say that the fiber products $\tilde{\mathcal{M}}(p,q;f,h)=W^{u}_{f}(p;h){}_{i_{u,p}}\times_{i_{s,p}} W^{s}_{f}(q;h)$ are all cut out transversely.  

Of course,  $\tilde{\mathcal{M}}(p,q;f,h)$ may be identified with the space of smooth maps $\gamma\co \mathbb{R}\to M$ such that $\dot{\gamma}(t)+\nabla^h f(\gamma(t))=0$ for all $t$ and $\lim_{t\to -\infty}\gamma(t)=p$ and $\lim_{t\to\infty}\gamma(t)=q$.  Under this identification we have an embedding \begin{align*} e_{pq}\co  \tilde{\mathcal{M}}(p,q;f,h)&\to M \\ \gamma&\mapsto \gamma(0) \end{align*}

Where $\mathbb{R}_+$ denotes the set of positive real numbers, for any $k\in\mathbb{Z}_+$ define \begin{align*} E_k\co \tilde{\mathcal{M}}(p,q;f,h)\times \mathbb{R}_{+}^{k-1}&\to M^k
\\ (\gamma,t_1,\ldots,t_{k-1}) & \mapsto \left(\gamma(0),\gamma(t_1),\gamma(t_1+t_2),\ldots,\gamma\left(\sum_{i=0}^{k-1}t_{i}\right)\right).\end{align*}  (Thus, viewing $\mathbb{R}_{+}^{0}$ as a one-point set, $E_0$ coincides with $e_{pq}$.)

\begin{prop} \label{genmet} Let $k\in\mathbb{Z}_+$ and for $0\leq i\leq k-1$ let $V_i$ be a smooth manifold and $g_i\co V_i\to M$ be a smooth map such that $g_i(V_i)\cap Crit(f)=\varnothing$.  Then there is a residual subset $\mathcal{R}\subset \mathcal{G}$ such that for every $h\in\mathcal{R}$ the negative gradient flow of $f$ with respect to $h$ is Morse--Smale, and for all $p,q\in Crit(f)$ the fiber product \[ \mathcal{V}(p,q,f,g_0,\ldots,g_{k-1};h):=\left(V_0\times\cdots\times V_{k-1}\right){}_{g_0\times\cdots\times g_{k-1}}\times_{E_k}\left(\tilde{\mathcal{M}}(p,q;f,h)\times \mathbb{R}_{+}^{k-1}\right) \] is cut out transversely.
\end{prop}

\begin{proof}
See Section \ref{app}.
\end{proof}

\subsection{Cap products}

Continuing to fix the above Morse function $f$, let $g\co V\to M$ be any pseudochain, where $V$ is an oriented $v$-dimensional manifold with (possibly empty) boundary and $v\leq n$.  Thus $\overline{g(V)}=g(V)\cup \Omega_g$, where $\Omega_g$
is covered by the image of a smooth map $\phi\co W\to M$ and all components of $W$ have dimension at most $v-2$.  If $v<n$ we additionally assume that $\overline{g(V)}\cap Crit(f)=\varnothing$.  If $v = n$ we instead additionally assume that $Crit(f)\cap (g(\partial V)\cup \Omega_g)=\varnothing$, and that every point in $Crit(f)$ is a regular value of $g$.

Writing $\partial g=g|_{\partial V}\co \partial V\to M$, we will assume from now on that the Morse--Smale metric $h$ being used to define the gradient flow of $f$ belongs to the intersection of the residual sets obtained by applying Proposition \ref{genmet} successively (in each instance with $k=1$) to $g|_{g^{-1}(M\setminus Crit(f))}$, to $\partial g$, and to $\phi$.

This being the case, for all $p,q\in Crit(f)$ the fiber products $V{}_{g}\times_{e_{pq}}\tilde{\mathcal{M}}(p,q;f)$, $\partial V{}_{g}\times_{e_{pq}}\tilde{\mathcal{M}}(p,q;f)$, and $W{}_{\phi}\times_{e_{pq}}\tilde{\mathcal{M}}(p,q;f)$ will all be cut out transversely,\footnote{In the case that $v=n$ the transversality of fiber products of the form $V{}_{g}\times_{e_{pp}}\tilde{\mathcal{M}}(p,p;f)$ follows from the assumption that the critical points of $f$ are all regular values for $g$} where $e_{pq}\co \tilde{\mathcal{M}}(p,q;f)\to M$ is the canonical embedding (if elements of $\tilde{\mathcal{M}}(p,q;f)$ are thought of as gradient flow trajectories $\gamma$ then $e_{pq}(\gamma)=\gamma(0)$).  In particular, in the case that $|p|_{f}-|q|_{f}=n-v$, the latter two fiber products will be empty, and  $V{}_{g}\times_{e_{pq}}\tilde{\mathcal{M}}(p,q;f)$ will be an oriented zero-manifold.  Moreover this oriented zero-manifold will be compact: to see this, recall that the images under $e_{pq}$ of a divergent sequence in $\tilde{\mathcal{M}}(p,q;f)$ will, after passing to a subsequence, converge to an element of some $e_{rs}(\tilde{\mathcal{M}}(r,s;f))$ where $|r|_f-|s|_f<|p|_f-|q|_f$, and use the fact that 
$V{}_{g}\times_{e_{rs}}\tilde{\mathcal{M}}(r,s;f)$, $\partial V{}_{\partial g}\times_{e_{rs}}\tilde{\mathcal{M}}(r,s;f)$, and $W{}_{\phi}\times_{e_{rs}}\tilde{\mathcal{M}}(r,s;f)$ are all cut out transversely and hence are empty by dimension considerations.  Consequently we have a well defined $\mathbb{K}$-valued signed count of elements $\#_{\mathbb{K}}\left(V{}_{g}\times_{e_{pq}}\tilde{\mathcal{M}}(p,q;f)\right)$ whenever $|p|_f-|q|_f=n-v$.

Accordingly, given $g\co V\to M$ as above (and a suitable Morse--Smale metric) we define a map $I_g\co CM_{*}(f;\mathbb{K})\to CM_{*}(f;\mathbb{K})$ as a direct sum of maps \[ I_g\co CM_{k}(f;\mathbb{K})\to CM_{k-(n-v)}(f;\mathbb{K}) \] obtained by extending linearly from the formula \[ I_g(p)=\sum_{{\scriptsize{\begin{array}{c} q\in Crit(f):\\|q|_f=k-(n-v)\end{array}}}}\#_{\mathbb{K}}\left(V{}_{g}\times_{e_{pq}}\tilde{\mathcal{M}}(p,q;f)\right)q\]

$I_g(x)$ might be thought of as a chain-level version of the cap product of $x\in CM_*(f;\mathbb{K})$ with the pseudochain $g\co V\to M$.

Evidently we have an identically-defined map (using the same metric $h$) $I_g\co CM_{k}(-f;\mathbb{K})\to CM_{k-(n-v)}(-f;\mathbb{K})$.

\begin{prop} \label{igprop}  The maps $I_g\co CM_{*}(\pm f;\mathbb{K})\to CM_{*-(n-v)}(\pm f;\mathbb{K})$ enjoy the following properties:
\begin{itemize} \item[(i)] For $x\in CM_{2n-k-v}(-f;\mathbb{K})$ and $y\in CM_{k}(f;\mathbb{K})$, \[ \Pi\left(I_g(x),y\right)=(-1)^{(n-v)(n-k)}\Pi\left(x,I_g(y)\right).\] 
\item[(ii)]  Assuming that $\partial g\co \partial V\to M$ is also a pseudochain, so that $I_{\partial g}$ is defined, \[ I_{\partial g}-d_fI_g+(-1)^{n-v}I_gd_f=0.\]
\end{itemize}
\end{prop}

\begin{proof} Since $\Pi$ is bilinear it suffices to check the equation in (i) when $x=q$ for some critical point $q$ with $|q|_f=k-(n-v)$ (so that $|q|_{-f}=2n-k-v$) and $y=p$ for some critical point $p$ with $|p|_f=k$.  By definition we have \[ \Pi\left(q,I_g(p)\right)=\#_{\mathbb{K}}\left(V{}_{g}\times_{e_{pq}}\tilde{\mathcal{M}}(p,q;f)\right) \] and \[ \Pi\left(I_g(q),p\right)=
\#_{\mathbb{K}}\left(V{}_{g}\times_{e_{pq}}\tilde{\mathcal{M}}(q,p;-f)\right) \]

In view of (\ref{tmqp}), these differ from each other by a factor $(-1)^{(|p|_f+|q|_f)(n-|p|_f)}=(-1)^{(n-v)(n-k)}$, proving (i).

(ii) is proven by examining the boundary of the one-manifolds $V{}_{g}\times_{\bar{E}_0}\tilde{\bar{\mathcal{M}}}(p,q;f)$ where $|p|_f-|q|_f=n-v+1$.  Recall here that $\tilde{\bar{\mathcal{M}}}(p,q;f)$ is a partial compactification of $\tilde{\mathcal{M}}(p,q;f)$, with oriented boundary given by \[   \partial \tilde{\bar{\mathcal{M}}}(p,q;f)= \left(\coprod_{|r|_{f}=|p|_f-1}\mathcal{M}(p,r;f)\times\tilde{\mathcal{M}}(r,q;f)\right)\sqcup (-1)^{|p|_f+|q|_f} 
\left(\coprod_{|r|_{f}=|q|_f+1}\tilde{\mathcal{M}}(p,r;f)\times\mathcal{M}(r,q;f)\right),\] and the characteristic map $\bar{E}_{0}\co \tilde{\bar{\mathcal{M}}}(p,q;f)\to M$ is equal to $e_{pq}$ on the interior $\tilde{\mathcal{M}}(p,q;f)$ and to the canonical embeddings of $\tilde{\mathcal{M}}(r,q;f)$ and  $\tilde{\mathcal{M}}(p,r;f)$ on
$\mathcal{M}(p,r;f)\times\tilde{\mathcal{M}}(r,q;f)$ and $\tilde{\mathcal{M}}(p,r;f)\times\mathcal{M}(r,q;f)$, respectively.

Now the $\Omega$-limit set $\Omega_{\bar{E}_0}$ of $\bar{E}_0$ is contained in spaces of the form $\tilde{\mathcal{M}}(r,s;f)$ of dimension at most $|p|_f-|q|_f-2=n-v-1$, and so is disjoint from $\overline{g(V)}$ by our transversality assumptions on the metric $h$.  For similar dimensional reasons, $\Omega_g$ is disjoint from $\overline{\bar{E}_0\left(\tilde{\bar{\mathcal{M}}}(p,q;f)\right)}$, and also $g(\partial V)\cap \bar{E}_0\left(\partial\tilde{\bar{\mathcal{M}}}(p,q;f)\right)=\varnothing$.  Therefore $V{}_{g}\times_{\bar{E}_0}\tilde{\bar{\mathcal{M}}}(p,q;f)$ is a compact oriented one-manifold with boundary (and no corners, since the fiber product of the boundaries is empty); according to (\ref{bdryid}) the oriented boundary is given by \begin{equation}\label{capbdry} \left((\partial V){}_{\partial g}\times_{e_{pq}}\tilde{\mathcal{M}}(p,q;f)\right)\coprod (-1)^{n-v}\left(V{}_{g}\times_{\bar{E}_0}\partial \tilde{\bar{\mathcal{M}}}(p,q;f) \right) \end{equation}  Of course, since (\ref{capbdry}) is the boundary of a compact oriented one-manifold, its signed number of points must be zero.  The signed number of points (counted in $\mathbb{K}$) in $(\partial V){}_{\partial g}\times_{e_{pq}}\tilde{\mathcal{M}}(p,q;f)$ is $\Pi(q,I_{\partial g}p)$.  As for the other set appearing in (\ref{capbdry}), we have, freely using properties of fiber product orientations from Section \ref{or:fp},
\begin{align*} V{}_{g}\times_{\bar{E}_0}\partial \tilde{\bar{\mathcal{M}}}(p,q;f)=&\left(\coprod_{|r|_f=|p|_f-1}\left(V{}_{g}\times_{\bar{E}_0}\tilde{\mathcal{M}}(r,q;f)\right)\times \mathcal{M}(p,r;f)\right)\\& \sqcup (-1)^{|p|_f+|q|_f} \left(\coprod_{|r|_f=|q|_f+1}\left(V{}_{g}\times_{\bar{E}_0}\tilde{\mathcal{M}}(p,r;f)\right)\times \mathcal{M}(r,q;f)\right)
\end{align*}

The signed number of points in the first of the two large unions above is easily seen to be $\Pi\left(q,I_g (d_fp)\right)$, while the signed number of points in the second large union (ignoring the sign $(-1)^{|p|_f+|q|_f}$) is $\Pi\left(q,d_fI_g(p)\right)$.  So since in this case $(-1)^{|p|_f+|q|_f}=(-1)^{n-v+1}$, setting the signed number of points in  (\ref{capbdry}) equal to zero gives \begin{align*} 0&=\Pi(q,I_{\partial g}p)+(-1)^{n-v}\left(\Pi\left(q,I_g (d_fp)\right)+(-1)^{n-v+1}\Pi\left(q,d_fI_g(p)\right)\right) 
\\&=\Pi\left(q,\left(I_{\partial g}-d_f I_g+(-1)^{n-v}I_g d_f\right)p\right).\end{align*}  Since this equation holds for all $p,q\in Crit(f)$ of the appropriate indices, we have proven (ii).
\end{proof}

We also mention the following somewhat trivial proposition, which we will appeal to later.  Recall the canonical cycle $M_f=\sum_{|p|_f=n}p\in CM_n(f;\mathbb{K})$ from (\ref{mdef}); similarly we have a canonical cycle $M_{-f}=\sum_{|q|_f=0}q\in CM_{n}(-f;\mathbb{K})$.

\begin{prop}\label{piint} Let $V$ be a compact oriented zero-manifold, let $g\co V\to M$ be a map such that $g(V)\cap Crit(f)=\varnothing$, and assume that the metric $h$ belongs to the residual set of Proposition \ref{genmet} applied with $k=1$ to the map $g$, so that $I_g\co CM_{*}(f;\mathbb{K})\to CM_{*-n}(f;\mathbb{K})$ is defined.  Then the signed number of points in $V$ is given by \[ \#_{\mathbb{K}}(V)=\Pi\left(M_{-f},I_gM_f\right).\]
\end{prop}

\begin{proof} Since $\dim V=0$, by dimensional considerations our assumption on $h$ amounts to the statement that $V{}_{g}\times_{e_{pq}} \tilde{\mathcal{M}}(p,q;f)=\varnothing$ unless $|p|_f=n$ and $|q|_f=0$.  Now as $p$ varies through index-$n$ critical points of $f$ and $q$ varies through index-$0$ critical points, the $\tilde{\mathcal{M}}(p,q;f)$ are sent by their canonical embeddings $e_{pq}$ to disjoint open subsets of $M$, and our orientation prescription for the unstable manifolds of index-$n$ and index-$0$ critical points ensures that the orientation of each such $\tilde{\mathcal{M}}(p,q;f)$ coincides with its orientation as an open subset of $M$.  As a result, we obtain \begin{align*}
\#_{\mathbb{K}}(V)&=\#_{\mathbb{K}}\left(V{}_g\times_{1_M}M\right)=\sum_{|q|_f=0}\sum_{|p|_f=n}\#_{\mathbb{K}}\left(V{}_{g}\times_{e_{pq}}
\tilde{\mathcal{M}}(p,q;f)\right) \\&=\sum_{|q|_f=0}\sum_{|p|_f=n}\Pi\left(q,I_g p\right)=\Pi\left(\sum_{|q|_f=0}q,\sum_{|p|_f=n}I_gp\right)
=\Pi\left(M_{-f},I_gM_f\right),
\end{align*} as desired. \end{proof}
 
\subsection{Gradient trajectories passing through two chains}

Having defined the chain-level cap product by using Proposition \ref{genmet} with $k=1$, we now set about defining new operations on the Morse chain complex by means of the $k=2$ version of Proposition \ref{genmet}.  This will require us to understand the boundaries of (compactifications of) moduli spaces of the form $(V_0\times V_1){}_{g_0\times g_1}\times_{E_1}\left(\tilde{\mathcal{M}}(p,q;f)\times (0,\infty)\right)$, where the map $E_1\co  \tilde{\mathcal{M}}(p,q;f)\times (0,\infty)\to M\times M$ is defined by \[ E_1(\gamma,t)=(\gamma(0),\gamma(t)).\]  

\begin{lemma}\label{mark2} Assume that the negative gradient flow of the exhausting Morse function $f\co M\to\R$ on a smooth oriented manifold $M$ is Morse--Smale and that $p,q\in Crit(f)$ are distinct. There is a pseudochain $\bar{E}_1\co \overline{\tilde{\mathcal{M}}(p,q;f)\times(0,\infty)}\to M\times M$, where $\overline{\tilde{\mathcal{M}}(p,q;f)\times(0,\infty)}$ is an oriented manifold with boundary whose interior is $\tilde{\mathcal{M}}(p,q;f)\times(0,\infty)$ and whose oriented boundary is given by \[ (-1)^{|p|_f-|q|_f}\partial \overline{\tilde{\mathcal{M}}(p,q;f)\times(0,\infty)}=\left((-1)^{|p|_f-|q|_f}C_1\right)\sqcup  C_2 \sqcup  C_3 \sqcup (-C_4)\sqcup C_5\sqcup C_6  \] 
where \begin{flalign*} \quad C_1&=\coprod_{|r|_f=|p|_f-1} \mathcal{M}(p,r;f)\times\tilde{\mathcal{M}}(r,q;f)\times (0,\infty)& &\bar{E}_1|_{C_1}([\gamma_1],\gamma_2,t)=(\gamma_2(0),\gamma_2(t)), \quad\\ \quad
C_2&=\coprod_{|r|_f=|q|_f+1}\tilde{\mathcal{M}}(p,r;f)\times\mathcal{M}(r,q;f)\times(0,\infty)& &\bar{E}_1|_{C_2}(\gamma_1,[\gamma_2],t)=(\gamma_1(0),\gamma_1(t)),\quad \\ \quad
C_3&=\coprod_{|q|_f<|r|_f<|p|_f}\tilde{\mathcal{M}}(p,r;f)\times\tilde{\mathcal{M}}(r,q;f)& &\bar{E}_1|_{C_3}(\gamma_1,\gamma_2)=(\gamma_1(0),\gamma_2(0)),  \quad \\
\quad C_4&=\tilde{\mathcal{M}}(p,q;f)& &\bar{E}_1|_{C_4}(\gamma)=(\gamma(0),\gamma(0)), \quad\\ \quad
C_5&=\tilde{\mathcal{M}}(p,q;f)& &\bar{E}_1|_{C_5}(\gamma)=(p,\gamma(0)), \quad \\ \quad
C_6&=\tilde{\mathcal{M}}(p,q;f)& &\bar{E}_1|_{C_6}(\gamma)=(\gamma(0),q).\quad\end{flalign*}  Moreover the $\Omega$-limit set $\Omega_{\bar{E}_1}$  is contained in the union of sets of the following form: \begin{itemize} \item[(i)]  Images of maps $E_1\co  \tilde{\mathcal{M}}(a,b;f)\times (0,\infty)\to M\times M$ \mbox{ with } $|a|_f-|b|_f\leq |p|_f-|q|_f-2$
\item[(ii)]  $e_{a,b}\left(\tilde{\mathcal{M}}(a,b;f)\right)\times e_{c,d}\left(\tilde{\mathcal{M}}(c,d;f)\right)$\mbox{ with }$(|a|_f-|b|_f)+(|c|_f-|d|_f)\leq |p|_f-|q|_f-1$ 
\item[(iii)] $\delta(e_0(\tilde{\mathcal{M}}(a,b;f)))$ where $\delta\co M\to M\times M$ is the diagonal embedding and $|a|_f-|b|_f\leq |p|_f-|q|_f-1$.
\item[(iv)] $\left(e_0(\tilde{\mathcal{M}}(a,b;f))\times Crit(f)\right)\cup\left(Crit(f)\times e_0(\tilde{\mathcal{M}}(a,b;f))\right)$ with
$|a|_f-|b|_f\leq |p|_f-|q|_f-1$.
\end{itemize}
\end{lemma}

\begin{proof}
See Section \ref{app}.
\end{proof}

We assume again that $M$ is compact and fix a Morse function $f\co M\to\R$.

Now suppose that we have two pseudochains $g_0\co V_0\to M$ and $g_1\co V_1\to M$; for $i=1,2$ write  $v_i=\dim V_i$, $\partial g_i=g_i|_{\partial V_i}$, and let $\phi_i\co W_i\to M$ be a smooth map from a manifold whose components all have dimension at most $v_i-2$ such that $\Omega_{g_i}\subset \phi_i(W_i)$.  Furthermore we assume the following:
\begin{itemize} \item[(A)] $\overline{g_0(V_0)}\cap \Omega_{g_1}=\overline{g_1(V_1)}\cap \Omega_{g_0}=\varnothing$.
\item[(B)] The fiber products $V_{0}\,{}_{g_0}\times_{g_1}V_1$, $(\partial V_{0}){}_{\partial g_0}\times_{g_1}V_1$, $V_{0}\,{}_{g_0}\times_{\partial g_1}(\partial V_1)$, and $(\partial V_{0}){}_{\partial g_0}\times_{\partial g_1}(\partial V_1)$ are all cut out transversely.
\item[(C)] $\overline{g_0(V_0)}\cap Crit(f)=\overline{g_1(V_1)}\cap Crit(f)=\varnothing$.
\end{itemize}

Also, as a general point of notation, if $\alpha\co A\to M$ and $\beta\co B\to M$ are smooth maps such that $A{}_{\alpha}\times_{\beta} B$ is cut out transversely, we will write $\alpha\times_M\beta$ for the map from $A{}_{\alpha}\times_{\beta} B$ to $M$ defined by $(\alpha\times_M\beta)(a,b)=\alpha(a)=\beta(b)$.

\begin{definition}\label{genwrt} Where $f,g_0,g_1,\phi_0,\phi_1$ are as above, a Riemannian metric $h$ on $M$ will be said to be \emph{generic with respect to }$f,g_0,g_1$ provided that it belongs to the residual sets given by Proposition \ref{genmet} applied with: 
\begin{itemize} \item $k=1$, to each of the functions \[ g_0,\,g_1,\,\partial g_0,\,\partial g_1,\,\phi_0,\,\phi_1,\, g_0\times_M g_1,\,    g_0\times_M\partial g_1,\,\partial g_0\times_M g_1,\partial g_0\times_M\partial g_1\]
\item $k=2$, to each of the pairs of functions \[ (g_0,g_1),\,(g_0,\partial g_1),\,(g_0,\phi_1),\,(\partial g_0,g_1),\,(\partial g_0,\partial g_1),\,(\partial g_0,\phi_1),\,(\phi_0,g_1),\,(\phi_0,\partial g_1),\,(\phi_0,\phi_1).\]
\end{itemize}
\end{definition}

Choose a Riemannian metric which is generic with respect to $f,g_0,g_1$ and let $p,q\in Crit(f)$ with $v_0+v_1+|p|_f-|q|_f+1=2n$.  Then the fiber product \[ (V_0\times V_1){}_{g_0\times g_1}\times_{\bar{E}_1}\overline{\tilde{\mathcal{M}}(p,q;f)\times (0,\infty)} \] is cut out transversely and by dimension considerations is an oriented zero-manifold.  Moreover the characterization of $\Omega_{\bar{E}_1}$ in Lemma \ref{mark2}, the assumption on the indices of $p$ and $q$, and the genericity assumption on $h$ ensure that any hypothetical divergent sequence in $(V_0\times V_1){}_{g_0\times g_1}\times_{\bar{E}_1}\overline{\tilde{\mathcal{M}}(p,q;f)\times(0,\infty)} $ would have a subsequence whose image under $(g_0\times g_1)\times_M\bar{E}_1$ converging to a point in a transversely-cut-out fiber product which has negative dimension and so is empty.  Thus $(V_0\times V_1){}_{g_0\times g_1}\times_{\bar{E}_1}\overline{\tilde{\mathcal{M}}(p,q;f)\times(0,\infty)}$ is compact and we may define \[ I_{g_0,g_1}(p)=\sum_{{\scriptsize{\begin{array}{c} q\in Crit(f):\\|q|_f=|p|_f+1-(2n-v_1-v_2)\end{array}}}}\#_{\mathbb{K}}\left((V_0\times V_1){}_{g_0\times g_1}\times_{\bar{E}_1}\overline{\tilde{\mathcal{M}}(p,q;f)\times(0,\infty)} \right)q \]  Extending this linearly gives us a map \[ I_{g_0,g_1}\co CM_{k}(f;\mathbb{K})\to CM_{k+1-(2n-v_1-v_2)}(f;\mathbb{K}).\]

\begin{prop} \label{fundid} The maps $I_{g_0},I_{g_1}$, and $I_{g_0,g_1}$ obey the following identity: \[ I_{\partial g_0,g_1}+(-1)^{v_0}I_{g_0,\partial g_1}+(-1)^{v_0+v_1}I_{g_0,g_1}d_f+d_f I_{g_0,g_1}+(-1)^{v_0(n-v_1)}I_{g_1}I_{g_0}+(-1)^{1+n(n-v_1)}I_{g_0\times_M g_1}=0\]
\end{prop}

\begin{proof} Let $p,q\in Crit(f)$ be critical points whose indices obey $|p|_f-|q|_f+v_0+v_1=2n$.  Then the transversely-cut-out fiber product $(V_0\times V_1){}_{g_0\times g_1}\times_{\bar{E}_1}\overline{\tilde{\mathcal{M}}(p,q;f)\times (0,\infty)}$ is one-dimensional, and the genericity assumption on $h$ together with dimensional considerations\footnote{\emph{i.e.}, any hypothetical divergent sequence would converge to a transversely-cut-out fiber product of negative dimension} imply that this fiber product is compact after adding its oriented boundary, which is given by \small \begin{align*} \partial\left((V_0\times V_1){}_{g_0\times g_1}\times_{\bar{E}_1}\overline{\tilde{\mathcal{M}}(p,q;f)\times (0,\infty)}\right)=&\left(\left(\partial\left(V_0\times V_1\right)\right){}_{g_0\times g_1}\times_{\bar{E}_1}\overline{\tilde{\mathcal{M}}(p,q;f)\times (0,\infty)}\right) \\& \quad \coprod \left(\left( V_0\times V_1\right){}_{g_0\times g_1}\times_{\bar{E}_1}(-1)^{2n-v_0-v_1}\partial\overline{\tilde{\mathcal{M}}(p,q;f)\times (0,\infty)}\right)\end{align*}  \normalsize

Now \begin{align*} 
\left(\partial  \left(V_0\times V_1\right)\right)&{}_{g_0\times g_1}\times_{\bar{E}_1}\overline{\tilde{\mathcal{M}}(p,q;f)\times (0,\infty)}\\= &\left(\left((\partial V_0)\times V_1\right){}_{(\partial g_0)\times g_1}\times_{\bar{E}_1}\overline{\tilde{\mathcal{M}}(p,q;f)\times (0,\infty)}\right)\\ & \qquad\qquad\sqcup (-1)^{v_0}\left(\left( V_0\times (\partial V_1)\right){}_{g_0\times \partial g_1}\times_{\bar{E}_1}\overline{\tilde{\mathcal{M}}(p,q;f)\times (0,\infty)}\right) \end{align*} 
has signed number of points (evaluated in $\mathbb{K}$)  equal to \[ \Pi\left(q,I_{\partial g_0,g_1}p\right)+(-1)^{v_0}\Pi\left(q,I_{g_0,\partial g_1}p\right).\]

Meanwhile since $2n-v_0-v_1=|p|_f-|q|_f$ we see that, with notation as in Lemma \ref{mark2}, \begin{align*} &\left(\left( V_0\times V_1\right){}_{g_0\times g_1}\times_{\bar{E}_1}(-1)^{2n-v_0-v_1}\partial\overline{\tilde{\mathcal{M}}(p,q;f)\times (0,\infty)}\right)\\ &\qquad = 
\left(V_0\times V_1\right){}_{g_0\times g_1}\times_{\bar{E}_1}\left(((-1)^{v_0+v_1}C_1)\sqcup C_2\sqcup C_3\sqcup(-C_4)\right)
\end{align*}

(The fiber products with $C_5$ and $C_6$ are empty since $\overline{g_0(V_0)}$ and $\overline{g_1(V_1)}$ are disjoint from $Crit(f)$.)

Now the signed number of points in $\left(V_0\times V_1\right){}_{g_0\times g_1}\times_{\bar{E}_1}((-1)^{v_0+v_1}C_1)$
is easily seen to be \[ (-1)^{v_0+v_1}\Pi\left(q,I_{g_0,g_1}d_fp\right),\] while that in $\left(V_0\times V_1\right){}_{g_0\times g_1}\times_{\bar{E}_1}C_2$ is \[ \Pi\left(q,d_fI_{g_0,g_1}p\right).\]  Meanwhile for any critical point $r$ with $|q|_f<|r|_f<|p|_f$ we have, using (\ref{prodid}), \begin{align*} \left(V_0\times V_1\right)&{}_{g_0\times g_1}\times_{e_{pr}\times e_{rq}}\left(\tilde{\mathcal{M}}(p,r;f)\times \tilde{\mathcal{M}}(r,q;f)\right)\\&=(-1)^{(n-v_1)(n-|p|_f+|r|_f)}\left(V_0\,{}_{g_0}\times_{e_{pr}}  \tilde{\mathcal{M}}(p,r;f) \right)\times\left( V_1\,{}_{g_1}\times_{e_{rq}}  \tilde{\mathcal{M}}(r,q;f) \right).\end{align*}  In order for neither $V_0\,{}_{g_0}\times_{e_{pr}}  \tilde{\mathcal{M}}(p,r;f)$ nor $V_1\,{}_{g_1}\times_{e_{rq}}  \tilde{\mathcal{M}}(r,q;f) $ to be nonempty it is necessary that $v_0+|p|_f-|r|_f=n$, in view of which it follows that the signed number of points in $\left(V_0\times V_1\right){}_{g_0\times g_1}\times_{\bar{E}_1}C_3$ is \[ (-1)^{v_0(n-v_1)}\Pi\left(q,I_{g_1}I_{g_0}p\right).\]  

Finally, where $\delta\co M\to M\times M$ for the diagonal embedding, using (\ref{associd}) and (\ref{diagid}) we have \begin{align*}
\left(V_0\times V_1\right){}_{g_0\times g_1}\times_{\bar{E}_1}(-C_4)&=-\left(V_0\times V_1\right){}_{g_0\times g_1}\times_{\delta}\left(M{}_{1_M}\times_{e_{pq}}\tilde{\mathcal{M}}(p,q;f)\right)
\\&=-\left(\left(V_0\times V_1\right){}_{g_0\times g_1}\times_{\delta}M\right){}_{1_M}\times_{e_{pq}}\tilde{\mathcal{M}}(p,q;f)
\\&=(-1)^{1+n(n-v_1)}\left(V_0\,{}_{g_0}\times_{g_1}V_1\right){}_{g_0\times_M g_1}\times_{e_{pq}}\tilde{\mathcal{M}}(p,q;f)
\end{align*}
Thus the signed number of points in $\left(V_0\times V_1\right){}_{g_0\times_M g_1}\times_{\bar{E}_1}(-C_4)$ is \[ (-1)^{1+n(n-v_1)}\Pi\left(q,I_{g_0\times_M g_1}p\right).\]  

We have now computed the signed number of points in all of the components of the boundary of the compact oriented one-manifold $(V_0\times V_1){}_{g_0\times g_1}\times_{\bar{E}_1}\overline{\tilde{\mathcal{M}}(p,q;f)\times (0,\infty)}$.  Of course, the total signed number of boundary points of this manifold is necessarily zero, and so we obtain  \begin{align*} 0=\Pi & \left(q,I_{\partial g_0,g_1}p +(-1)^{v_0}I_{g_0,\partial g_1}p+(-1)^{v_0+v_1}I_{g_0,g_1}d_fp+d_f I_{g_0,g_1}p\right.\\&\qquad\qquad\qquad\qquad \left. +(-1)^{v_0(n-v_1)}I_{g_1}I_{g_0}p+(-1)^{1+n(n-v_1)}I_{g_0\times_M g_1}p\right).\end{align*}

Since this holds for all critical points $p$ and $q$ of the appropriate indices the result follows.
\end{proof}

\begin{remark}\label{ainfty} Of course, one may continue in this fashion and define, for any positive integer $k$ and suitably transverse pseudochains $g_i\co V_i\to M$ for $i=0,\ldots,k-1$ of dimension $v_i$, operations \[ I_{g_0,\ldots,g_{k-1}}\co CM_{*}(f;\K)\to CM_{*-1-\sum_{i=0}^{k-1}(n-v_i-1)}(f;\K)\] by counting elements of fiber products $\left(V_0\times\cdots\times V_{k-1}\right){}_{g_0\times\cdots\times g_{k-1}}\times_{E_k}\left(\tilde{\mathcal{M}}(p,q;f,h)\times \mathbb{R}_{+}^{k-1}\right)$. One can see that these operations satisfy \small \begin{align*} (-1)^k & d_f  I_{g_0,\ldots,g_{k-1}}+(-1)^{\sum_{i=0}^{k-1}(n-v_i)}I_{g_0,\ldots,g_{k-1}}d_f +\sum_{l=0}^{k-1}(-1)^{\sum_{i=0}^{l-1}v_i}I_{\ldots,g_{l-1},\partial g_l,g_{l+1},\ldots}\\&+\sum_{l=1}^{k-1}\left((-1)^{kl+\left(1+\sum_{i=0}^{l-1}(v_i-1)\right)\left(\sum_{j=l}^{k-1}(n-v_j)\right)}I_{g_{l},\ldots,g_{k-1}}I_{g_0,\ldots,g_{l-1}}+(-1)^{k+l+n\sum_{j=l}^{k-1}(n-v_j)}I_{\ldots,g_{l-2},g_{l-1}\times_Mg_{l},g_{l+1},\ldots}\right)
=0\end{align*}\normalsize

The proof of this identity for the most part follows straightforwardly by the same arguments as were used in the proofs of Lemma \ref{mark2} and Proposition \ref{fundid}; a little additional effort is required to obtain the sign on $I_{\ldots,g_{l-2},g_{l-1}\times_Mg_{l},g_{l+1},\ldots}$, which entails comparing the orientations of \begin{equation}\label{fp1} \left(\cdots\times V_{l-2}\times(V_{l-1}\,{}_{g_{l-1}}\times_{g_l}V_l)\times V_{l+1}\times\cdots\right){}_{\ldots\times g_{l-2}\times g_{l-1}\times_Mg_{l}\times g_{l+1},\ldots}\times_{\bar{E}_{k-1}}(\tilde{\mathcal{M}}(p,q;f)\times\R_{+}^{k-2}) \end{equation} and \begin{equation}\label{fp2} (V_0\times\cdots\times V_{k-1}){}_{g_0\times\cdots\times g_{l-1}}\times_{\delta_l\circ\bar{E}_{k-1}} (\tilde{\mathcal{M}}(p,q;f)\times\R_{+}^{k-2})\end{equation} where $\delta_l\co M^{k-1}\to M^k$ is defined by $\delta_l(m_0,\ldots,m_{k-2})=(m_0,\ldots,m_{l-1},m_{l-1},\ldots,m_{k-2})$.  To do this, note that we can rewrite (\ref{fp2}) as \[ \left((V_0\times\cdots\times V_{k-1}){}_{g_0\times\cdots\times g_{l-1}}\times_{\delta_l}M^{k-1}\right){}_{1_{M^{k-1}}}\times_{\bar{E}_{k-1}} (\tilde{\mathcal{M}}(p,q;f)\times\R_{+}^{k-2}),\] so that the problem reduces to comparing the orientation of $(V_0\times\cdots\times V_{k-1}){}_{g_0\times\cdots\times g_{l-1}}\times_{\delta_l}M^{k-1}$ to that of $V_0\times\cdots\times V_{l-2}\times(V_{l-1}\,{}_{g_{l-1}}\times_{g_l}V_l)\times V_{l+1}\times\cdots\times V_{k-1}$.  In turn, this can be done by repeated use of (\ref{prodid}) and (\ref{diagid}).  We will not use this construction for $k>2$, so further details are left to the reader.
\end{remark}

\section{From linked pseudoboundaries to critical points}\label{linkcrit}

We are now prepared to demonstrate a relationship between linking numbers of pseudoboundaries and the Morse-theoretic linking pairing (\ref{mlpair}).  We continue to fix a Morse function $f\co M\to\R$ where $M$ is a compact oriented smooth $n$-dimensional manifold.


\begin{definition}For any integer $k$ with $0\leq k\leq n-1$, \begin{itemize}\item $\mathcal{B}_k(M)$ denotes the set of $k$-pseudoboundaries in $M$. \item  $\mathcal{T}_k(M,f)$ denotes the collection of pairs $(b_0,b_1)\in \mathcal{B}_{k}(M)\times \mathcal{B}_{n-k-1}(M)$ such that \[ \overline{Im(b_{0})}\cap Crit(f)=\overline{Im(b_{1})}\cap Crit(f)=\overline{Im(b_{0})}\cap \overline{Im(b_{1})}=\varnothing.\]\end{itemize}
\end{definition}

Thus if $(b_0,b_1)\in \mathcal{T}_k(M,f)$ then we obtain a well-defined linking number $lk(b_0,b_1)$, and by Proposition \ref{genmet}, all metrics $h$ in some residual subset will be generic with respect to $f,b_{0},b_{1}$ in the sense of Definition \ref{genwrt}.
 For any such metric $h$ we may then define the maps $I_{b_0},I_{b_1}\co CM_{*}(\pm f;\K)\to CM_{*}(\pm f;\K)$ and $I_{b_0,b_1}\co CM_{*}(f;\K)\to CM_{*}(f;\K)$.

Recall that the Morse complex $CM_*(f;\K)$ on an $n$-dimensional manifold $M$ carries a distinguished element $M_f\in CM_n(f;\K)$ defined in (\ref{mdef}), which is a cycle by Proposition \ref{dmzero}.

\begin{prop}\label{linklinkprop} Let $f\co M\to\R$ be a Morse function on a compact smooth oriented $n$-dimensional manifold $M$, suppose that $(b_0,b_1)\in \mathcal{T}_k(M,f)$ where $0\leq k\leq n-1$, let $\K$ be any ring, and choose a Riemannian metric $h$ which is generic with respect to $f,b_0,b_1$.   Then:
\begin{itemize}
\item[(i)] The elements $I_{b_0} M_f$ and $I_{b_1} M_{-f}$ belong to the images of the maps $d_f\co CM_{k+1}(f;\K)\to CM_k(f;\K)$ and $d_{-f}\co CM_{n-k}(-f;\K)\to CM_{n-k-1}(-f;\K)$, respectively.
\item[(ii)] Where $lk_{\K}(b_1,b_0)$ is the image of $lk(b_1,b_0)$ under the unique unital ring morphism $\Z\to \K$, \begin{equation}\label{linklink} \Lambda(I_{b_1}M_{-f},I_{b_0}M_f)=lk_{\K}(b_1,b_0)-(-1)^{(k+1)(n-k)}\Pi(M_{-f},I_{b_0,b_1}M_f).\end{equation}  
\end{itemize}
\end{prop}

\begin{remark}\label{nomassey} Observe that the last term in (\ref{linklink}) counts integral curves $\gamma\co[0,T]\to M$ of $-\nabla f$ (with $T>0$ arbitrary) such that $\gamma(0)\in b_0(B_0)$ and $\gamma(T)\in b_1(B_1)$.  In particular the last term of (\ref{linklink}) automatically vanishes if $\inf(f|_{b_1(B_1)})\geq \sup(f|_{b_0(B_0)})$, by virtue of the fact that $f$ decreases along its negative gradient flowlines.
\end{remark}

\begin{proof} For notational convenience we will first give the proof assuming that $k<n-1$; at the end of the proof we will then indicate how modify the proof if instead $k=n-1$.  

Since $b_0$ and $b_1$ are assumed to be pseudoboundaries, there are pseudochains $c_0\co C_0\to M$ and $c_1\co C_1\to M$, of dimensions $k+1$ and $n-k$ respectively, such that $\partial C_0=B_0$, $\partial C_1=B_1$, $c_0|_{B_0}=b_0$, and $c_1|_{B_1}=b_1$.  By a suitable perturbation  we may assume that the conclusion of Proposition \ref{tvs1} holds with $F=c_0$ and $g=b_1$, and moreover that $\Omega_{c_1}\cap Crit(f)=\overline{c_0(C_0)}\cap Crit(f)=\varnothing$ (for the latter we use the assumption that $k\neq n-1$) and that each point of $Crit(f)$ is a regular value for $c_1$.  We will always assume below that the  Riemannian metric is chosen from the intersection of an appropriate collection of the residual subsets given by Proposition \ref{genmet}.

Statement (i) then follows from Propositions \ref{dmzero} and \ref{igprop}(ii), as we have \[ d_f(I_{c_0}M_f)=I_{b_0}M_f+(-1)^{n-k-1}I_{c_0}d_fM_f=I_{b_0}M_f \] and likewise $d_{-f}(I_{c_1}M_{-f})=I_{b_1}M_{-f}$.

Moreover, by Definition \ref{lk} and Proposition \ref{piint} we have \[ lk_{\K}(b_1,b_0)=\Pi\left(M_{-f},I_{c_0\times_M b_1}M_f\right).\]

Now since $\partial B_1=\varnothing$ and since $c_0|_{\partial C_0}=b_0$, Proposition \ref{fundid} applied with $g_0=c_0$ and $g_1=b_1$ gives (bearing in mind that $(-1)^{k(k+1)}=1$) \begin{equation}\label{pq} I_{c_0\times_M b_1}-(-1)^{(n-k)(k+1)}I_{b_0,b_1}=(-1)^nI_{c_0,b_1}d_f+(-1)^{n(k+1)}d_fI_{c_0,b_1}+(-1)^{(n+1)(k+1)}I_{b_1}I_{c_0}.\end{equation}  Now since $d_fM_f=0$ and $d_{-f}M_{-f}=0$ we have \[ \Pi(M_{-f},I_{c_0,b_1}d_f M_f)=0\qquad \mbox{and}\qquad \Pi(M_{-f},d_fI_{c_0,b_1}M_f)=(-1)^n\Pi(d_{-f}M_{-f},I_{c_0,b_1}M_f)=0.
\]

So by (\ref{pq}) and Proposition \ref{igprop}(i) we obtain \begin{align*} lk_{\K}(b_1,b_0)-&(-1)^{(n-k)(k+1)}\Pi(M_{-f},I_{b_0,b_1}M_f)=(-1)^{(n+1)(k+1)}\Pi(M_{-f},I_{b_1}I_{c_0}M_f)\\&=(-1)^{(n+1)(k+1)}(-1)^{(k+1)(n-k-1)}\Pi(I_{b_1}M_{-f},I_{c_0}M_f)=\Pi(I_{b_1}M_{-f},I_{c_0}M_f)\end{align*}

Since $d_f\left(I_{c_0}M_f\right)=I_{b_0}M_f$, we have by definition $\Pi(I_{b_1}M_{-f},I_{c_0}M_f)=\Lambda(I_{b_1}M_{-f},I_{b_0}M_f)$, proving 
(\ref{linklink}).  

This completes the proof if $k<n-1$.  Now suppose that $k=n-1\geq 1$.  Then $n-k-1< n-1$, so in the first paragraph of the proof we may instead arrange for $\overline{c_1(C_1)}\cap Crit(f)=\Omega_{c_0}\cap Crit(f)=\varnothing$ and for every point of $Crit(f)$ to be a regular value for $c_0$.  Just as in the $k<n-1$ case we have $I_{b_0} M_f=d(I_{c_0}M_f)$ and $I_{b_1} M_{-f}=d(I_{c_1}M_{-f})$.  If the image of $c_0$ intersects $Crit(f)$ then the operator $I_{c_0,b_1}$ is no longer defined; however now $I_{b_0,c_1}$ is defined, and using Proposition \ref{linksym} we have $lk(b_1,b_0)=-\Pi(M_{-f},I_{b_0\times_M c_1}M_f)$.  Then using Proposition \ref{fundid} with $g_0=b_0$ and $g_1=c_1$ together with (\ref{altlam}),  an identical argument to the one given above yields (\ref{linklink}).

The only remaining case is where $n=1$ and $k=0$, \emph{i.e.} where $M$ is a disjoint union of circles and the pseudoboundaries $b_0$ and $b_1$ are homologically trivial linear combinations of points on these circles.  In this case the proposition is an exercise in the combinatorics of points on one-manifolds equipped with a Morse function, for which we give the following outline, leaving details to the reader.  The linking number $lk(b_1,b_0)$ is computed by pairwise connecting the points of $b_0$ by a collection $\mathcal{I}_0$ of intervals and then counting the intersections of these intervals with the points of $b_1$.  To compute $\Lambda(I_{b_1} M_{-f},I_{b_0} M_f)$, one modifies the intervals of $\mathcal{I}_0$ by, for each point $p$ of $b_0$, adding or deleting the segment from $p$ to the local minimum adjacent to $p$, and then counts the intersections of these modified intervals with the points of $b_1$.  The difference $\Lambda(I_{b_1} M_{-f},I_{b_0} M_f)-lk_{\K}(b_1,b_0)$ then counts the points of $b_1$ which lie between a point of $b_0$ and its adjacent minimum, \emph{i.e.}, the points of $b_1$ which lie below a point of $b_0$ on a gradient flowline of $f$.  Such points  are precisely counted by $\Pi(M_{-f},I_{b_0,b_1}M_f)$, proving (\ref{linklink}).
\end{proof}


\begin{cor}\label{linkcor1} Let $\K$ be a field, let $f\co M\to\R$ be a Morse function on a compact smooth oriented $n$-manifold $M$, and suppose that, for $1\leq i\leq r$, $1\leq j\leq s$, we have $b_{i,+}\in \mathcal{B}_k(M)$ and $b_{j,-}\in \mathcal{B}_{n-k-1}(M)$ such that, for all $i,j$, $(b_{j,-},b_{i,+})\in \mathcal{T}_k(M,f)$.  Choose a Riemannian metric which is generic with respect to $f,b_{i,+},b_{j,-}$ for all $i$ and $j$  and consider the $r\times s$ matrix $L$ with entries \[ L_{ij}=lk_{\K}(b_{j,-},b_{i,+})-(-1)^{(n-k)(k+1)}\Pi(M_{-f},I_{b_{i,+},b_{j,-}}M_f) \] 
Then the rank of the operator $d_{f,k+1}\co CM_{k+1}(f;\K)\to CM_k(f;\K)$ is at least equal to the rank of the matrix $L$.  Thus where $c_j(f)$ denotes the number of critical points of $f$ with index $j$, and where $\mathfrak{b}_j(M;\K)$ is the rank of the $j$th singular homology of $M$ with coefficients in $\K$, we have \[ c_{k}(f)\geq \mathfrak{b}_k(M;\K)+\rk(L)\qquad \mbox{ and }\qquad c_{k+1}(f)\geq \mathfrak{b}_{k+1}(M;\K)+\rk(L).\]
\end{cor}

\begin{proof}
Denote \[ B_{k}^{f}=Im\left(d_f\co  CM_{k+1}(f;\K)\to CM_k(f;\K)\right) \quad B_{n-k-1}^{-f}=Im\left(d_{-f}\co  CM_{n-k}(-f;\K)\to CM_{n-k-1}(-f;\K)\right).\]  The Morse-theoretic linking form $\Lambda$ gives a linear map $\Lambda^{\diamond}\co B_{k}^{f}\to Hom_{\K}(B_{n-k-1}^{-f};\K)$ defined by $(\Lambda^{\diamond}x)(y)=\Lambda(y,x)$.  Define $A_f\co \K^{r}\to B_{k}^{f}$ by $A_f(x_1,\ldots,x_r)=\sum_i x_i I_{b_{i,+}}M_f$, and $A_{-f}\co \K^s\to B_{n-k-1}^{f}$ by $A_{-f}(y_1,\ldots,y_s)=\sum_j y_jI_{b_{j,-}}M_{-f}$.  Then where $A_{-f}^{*}\co Hom_{\K}(B_{n-k-1}^{f},\K)\to Hom_{\K}(\K^{s},\K)$ denotes the adjoint of $A_{-f}$, Proposition \ref{linklinkprop} shows that we have a commutative diagram \[ \xymatrix{ 
\K^r\ar[r]^{A_f} \ar[d]_{L^{\diamond}} & B_{k}^{f} \ar[d]^{\Lambda^{\diamond}} \\ Hom_{\K}(\K^{s},\K) & \ar[l]^{A_{-f}^{*}} Hom_{\K}(B_{n-k-1}^{f},\K)
}\] where $L^{\diamond}$ is defined by $(L^{\diamond}\vec{x})(\vec{y})=\sum_{i,j}L_{ij}x_iy_j$.  The rank of the linear map $L^{\diamond}$ is equal to the rank of the matrix $L$, so since $L^{\diamond}$ factors through $B_{k}^{f}$ it follows that $B_{k}^{f}$ has dimension at least equal to the rank of $L$.

The last sentence of the corollary then follows immediately, since $CM_{k}(f;\K)$ and $CM_{k+1}(f;\K)$ are freely generated over $\K$ by the critical points of $f$ with index, respectively, $k$ and $k+1$, and since the singular homology of $M$ is equal to the homology of the complex $(CM_*(f;\K),d_f)$ (so that $c_k(f)$ and $c_{k+1}(f)$ are each  equal to at least the rank of $d_f\co  CM_{k+1}(f;\K)\to CM_k(f;\K)$ plus, respectively, $\mathfrak{b}_k(M;\K)$ and $\mathfrak{b}_{k+1}(M;\K)$). 
\end{proof}



We would now like to connect some of these results to the filtration structure on the Morse complex $CM_*(f;\K)$ of $f$.  Define a function $\ell_f\co CM_*(f;\K)\to\R\cup\{-\infty\}$ by \[ \ell_f\left(\sum_{p\in Crit(f)}a_p p\right)=\max\{f(p)|a_p\neq 0\},\] where the maximum of the empty set is defined to be $-\infty$.  Then for any $\lambda\in \R$ and $k\in\N$, \[ CM^{\lambda}_{*}(f;\K)=\{y\in CM_*(f;\K)|\ell_f(y)\leq \lambda\} \] is a subcomplex of $CM_*(f;\K)$ (with respect to the Morse boundary operator associated to any Morse--Smale metric), owing to the fact that the function $f$ decreases along its negative gradient flowlines.  Of course we have corresponding notions with $f$ replaced by $-f$.

One useful fact is that the \emph{filtered isomorphism type} of the Morse complex $CM_{*}(f;\K)$ is independent of the choice of the Morse--Smale metric $h$ used to define it.  This was essentially observed in \cite[Theorem 1.19, Remark 1.23(b)]{CR}; see also \cite[Lemma 3.8]{U11} for a proof of the analogous statement in the more complicated setting of Hamiltonian Floer theory.

\begin{definition} Let $f\co M\to \R$ be a Morse function on a compact $n$-dimensional manifold and fix a coefficient ring $\K$, a metric $h$ with respect to which the negative gradient flow of $f$ is Morse--Smale, and a number $k\in \{0,\ldots,n-1\}$.  The \emph{algebraic link separation} of $f$ is the quantity \[ \beta^{alg}_{k}(f;\K)=\sup\left(\{0\}\cup\left\{-\ell_{-f}(x)-\ell_f(y)|x\in Im(d_{-f,n-k}),y\in Im(d_{f,k+1}),\,\Lambda(x,y)\neq 0\right\}\right).\]
\end{definition}

As the notation suggests, this quantity depends on $\K$ but not on the metric $h$.  This can be proven in a variety of different ways; for instance, for a given metric $h$, the complex $CM_{*}(-f;\K)$ is given as the dual of the complex $CM_*(f;\K)$ by means of the pairing $\Pi$ according to (\ref{adjpi}).  Consequently  $\beta^{alg}_{k}(f;\K)$ is determined by the filtered isomorphism type of $CM_*(f;\K)$, which as mentioned earlier is independent of $h$.


The second sentence of the following is an easy special case of \cite[Corollary 1.6]{U10}; we include a self-contained proof to save the reader the trouble of wading through the technicalities required for the more general version proven there.

\begin{prop}\label{alg-betaprop}  For any nontrivial coefficient ring $\K$ and any grading $k$ we have $\beta^{alg}_{k}(f;\K)=0$ if and only if $d_{f,k+1}=0$.  Furthermore, if $\K$ is a field, then \begin{equation} \label{alg-beta}
\beta^{alg}_{k}(f;\K)=\inf\left\{\beta\geq 0\left|(\forall \lambda\in\R)\left(Im(d_{f,k+1})\cap CM^{\lambda}_{k}(f;\K)\subset d_{f,k+1}(CM_{k+1}^{\lambda+\beta}(f;\K))\right) \right.\right\}. \end{equation}  
\end{prop}
\begin{proof}  Denote the right-hand side of (\ref{alg-beta}) by $\beta_k(f;\K)$.  Note first that if $d_{f,k+1}=0$ then (for any ring $\K$, not necessarily a field) it follows immediately from the definitions that $\beta^{alg}_{k}(f;\K)=\beta_{k}(f;\K)=0$.  So for the rest of the proof we assume that $d_{f,k+1}\neq 0$; we now show that this implies that $\beta^{alg}_{k}(f;\K)>0$.

Since $d_{f,k+1}\neq 0$ let us choose an element $y=\sum_{i=1}^{l}y_ip_i\in Im(d_{f,k+1})\setminus\{0\}$ (where the $p_i$ are are all distinct).  Reordering the indices if necessary we may assume that $y_1\neq 0$ and $f(p_1)=\ell_{f}(y)$.  Now view $p_1$ as an element of $CM_{n-k}(-f;\K)$ and let $x=d_{-f,n-k} p_1$.    By (\ref{altlam}) we see that \[ \Lambda(x,y)=(-1)^{n-k}\Pi(p_1,y)=(-1)^{n-k}y_1\neq 0.\]  Moreover where $\mu$ is the smallest critical value of $f$ which is strictly larger than $f(p_1)$, one has $\ell_{-f}(x)\leq -\mu$.  Thus \[  -\ell_{-f}(x)-\ell_f(y)\geq \mu-f(p_1)>0.\]  By the definition of $\beta^{alg}_{k}(f;\K)$ this completes the proof of the first sentence of the proposition.

Let us now prove the second sentence of the proposition; in fact our argument will show that $\beta^{alg}_{k}(f;\K)\leq \beta_{k}(f;\K)$ for any ring $\K$, with equality if $\K$ is a field.

Consider any $y\in CM_{k}(f;\K)$ with $0\neq y\in Im (d_{f,k+1})$.   Suppose that $x\in CM_{n-k-1}(-f;\K)$ obeys $\Lambda(x,y)\neq 0$.  Then for any $z\in CM_{k+1}(f;\K)$ such that $d_f z= y$, we have $\Pi(x,z)\neq 0$.  But it is easy to see that the fact that $\Pi(x,z)\neq 0$ implies that $\ell_{-f}(x)+\ell_f(z)\geq 0$, \emph{i.e.}, $-\ell_{-f}(x)\leq \ell_f(z)$.  Thus, for all $y\in Im (d_{f,k+1})\setminus \{0\}$, we have \begin{align}\label{ellell} \sup\{-\ell_{-f}(x)-\ell_{f}(y)&|x\in CM_{n-k-1}(-f;\K), \Lambda(x,y)\neq 0\}\\& \qquad \qquad \leq \inf\{\ell_f(z)-\ell_f(y)|z\in CM_{k+1}(f;\K),\,d_{f,k+1}z=y\}.\nonumber \end{align}

Now since we have already shown that $\beta^{alg}_{k}(f;\K)>0$,  $\beta^{alg}_{k}(f;\K)$ is equal to the supremum of the left-hand side of (\ref{ellell}) over all $y\in Im (d_{f,k+1})\setminus \{0\}$.  On the other hand, given that $d_{f,k+1}\neq 0$, it is easy to see that $\beta_k(f;\K)$ is equal to the supremum of the right-hand side of (\ref{ellell}) over all $y\in Im (d_{f,k+1})\setminus \{0\}$.  Thus taking the suprema of the two sides of (\ref{ellell}) over $y$ establishes that \[ \beta^{alg}_{k}(f;\K)\leq \beta_{k}(f;\K).\]

It remains to prove the reverse inequality, for which we restrict to the case that $\K$ is a field (it is not difficult to construct counterexamples to this inequality when $\K$ is not a field).  Let $\alpha<\beta_k(f;\K)$; we will show that $\beta^{alg}_{k}(f;\K)\geq \alpha$.  For notational convenience we may assume that $\alpha$ is not equal to the difference between any two critical values of  $f$.  By definition there is then some $\lambda\in \R$ and some element $y\in (Im (d_{f,k+1}))\cap CM_{k}^{\lambda}(f;\K)$ such that $y\notin d_{f,k+1}(CM_{k}^{\lambda+\alpha}(f;\K))$; decreasing $\lambda$ if necessary we may assume that $\lambda=\ell(y)$, so that $\lambda$ is a critical value of $f$, and therefore $\lambda+\alpha$ is not a critical value of $f$ by our choice of $\alpha$. Since $y\in Im(d_{f,k+1})$, $y$ is a cycle, but since $y\notin d_{f,k+1}(CM_{k}^{\lambda+\alpha}(f;\K))$, $y$ represents a nontrivial element $[y]$ in the filtered homology $H_k(CM_{*}^{\lambda+\alpha}(f;\K))$.  

  Consider the quotient complex $D^{-\lambda-\alpha}_{*}:=\frac{CM_{*}(-f;\K)}{CM^{-\lambda-\alpha}(-f;\K)}$.  Since $\lambda+\alpha$ is not a critical value of $f$ the Poincar\'e pairing $\Pi$ vanishes on $CM^{-\lambda-\alpha}_{*}(-f;\K)\times CM^{\lambda+\alpha}_{*}(f;\K)$, and descends to a perfect pairing $\underline{\Pi}\co D^{-\lambda-\alpha}_{*}\times CM^{\lambda+\alpha}_{*}(f;\K)\to \K$. Moreover by (\ref{adjpi}) the differential on the quotient complex $D^{-\lambda-\alpha}_{*}$ induced by $d_{-f}$ is (up to a grading-dependent sign) dual via $\underline{\Pi}$ to the differential $d_f$ on $CM^{\lambda+\alpha}_{*}(f;\K)$.  Therefore by the field-coefficient case of the universal coefficient theorem the pairing $\underline{\Pi}$ induces a nondegenerate pairing between the homologies of  $D^{-\lambda-\alpha}_{*}$ and $CM^{\lambda+\alpha}_{*}(f;\K)$.  In particular since our element $y$ is homologically nontrivial in $CM^{\lambda+\alpha}_{*}(f;\K)$ there is a degree-$(n-k)$ \emph{cycle} $\bar{w}\in D^{-\lambda-\alpha}_{*}=\frac{CM_{*}(-f;\K)}{CM^{-\lambda-\alpha}_{*}(-f;\K)}$ which pairs nontrivially with $y$; thus where $w\in CM_{n-k}(-f;\K)$ is a representative of $\bar{w}$ we have $\Pi(w,y)\neq 0$.  Now the fact that $\bar{w}$ is a degree-$(n-k)$ cycle in $D^{-\lambda-\alpha}_{*}$ implies that $x:=d_{-f}w\in CM^{-\lambda-\alpha}_{n-k-1}(-f;\K)$.  By (\ref{altlam}) we have $\Lambda(x,y)=(-1)^{n-k}\Pi(w,y)\neq 0$.    Moreover $\ell_{-f}(x)+\ell_f(y)\leq -\lambda-\alpha+\lambda=-\alpha$.  Thus we have found $x\in Im(d_{-f,n-k})$ and $y\in Im(d_{f,k+1})$ such that $\Lambda(x,y)\neq 0$ and $-\ell_{-f}(x)-\ell_f(y)\geq \alpha$, proving that $\beta^{alg}_{k}(f;\K)\geq\alpha$.  Since $\alpha$ was an arbitrary number smaller than $\beta_k(f;\K)$ (and not equal to the difference between any two critical values of $f$), this implies that \[ \beta^{alg}_{k}(f;\K)\geq \beta_k(f;\K),\] completing the proof.
\end{proof}

\begin{definition} If $f\co M\to\R$ is a Morse function on a compact $n$-dimensional manifold, $\K$ is a ring, and $k\in\{0,\ldots,n-1\}$, the \emph{geometric link separation} of $f$ is \[ \beta_{k}^{geom}(f;\K)=\sup\left\{\min(f|_{\overline{Im(b_-)}})-\max(f|_{\overline{Im(b_+)}}) \left|\begin{array}{c} b_-\co B_-\to M\mbox{ is an $(n-k-1)$-pseudoboundary,}\\ b_+\co B_+\to M  \mbox{ is a $k$-pseudoboundary,}\\ \overline{b_-(B_-)}\cap \overline{b_+(B_+)}=\varnothing,\,lk_{\K}(b_-,b_+)\neq 0\end{array}\right.\right\}.\]
\end{definition}

\begin{remark} \label{betasmooth} If the ring $\K$ has characteristic zero (\emph{i.e.}, if for every nonzero integer $n$ one has $n1\neq 0$ where $1$ is the multiplicative identity in $\K$ and we view $\K$ as a $\mathbb{Z}$-module), then one could restrict the pseudoboundaries $b_{\pm}$ in the definition of $\beta_{k}^{geom}(f;\K)$ to have domains which are compact smooth oriented manifolds.  Indeed this follows easily from two instances of Lemma \ref{smoothen}, applied using appropriately small open sets around $\overline{b_{\pm}(B_{\pm})}$.  In this regard note also that if $B$ is a compact smooth oriented manifold without boundary and $b\co B\to M$ is a smooth map, then it follows from results of \cite{Z} that $b$ is a pseudoboundary if and only if $b_*[B]=0\in H_{*}(M;\Z)$.
\end{remark}

The following is one of our main results.

\begin{theorem}\label{alggeom} For any Morse function $f\co M\to\R$ on a compact $n$-dimensional manifold, any nontrivial ring $\K$, and any $k\in \{0,\ldots,n-1\}$, we have \[ \beta^{alg}_{k}(f;\K)=\beta^{geom}_{k}(f;\K).\]
\end{theorem}

We will prove the inequality ``$\geq$'' in Theorem \ref{alggeom} now, and the reverse inequality in the following section.

\begin{proof}[Proof that $\beta^{alg}_{k}(f;\K)\geq \beta^{geom}_{k}(f;\K)$]  Suppose that $\alpha<\beta^{geom}_{k}(f;\K)$.  There are then an $(n-k-1)$-pseudoboundary $b_-\co B_-\to M$ and a $k$-pseudoboundary $b_+\co B_+\to M$  such that $\overline{b_-(B_-)}\cap \overline{b_+(B_+)}=\varnothing$, $lk_{\K}(b_-,b_+)\neq 0$, and $\min (f|_{\overline{b_-(B_-)}}) - \max (f|_{\overline{b_+(B_+)}}) >\alpha$.  By replacing $b_-$ and $b_+$ by $\phi\circ b_-$ and $\phi\circ b_+$ where $\phi$ is an appropriately-chosen diffeomorphism which is close to the identity, we may arrange that the above properties still hold and additionally $\overline{b_-(B_-)}\cap Crit(f)=\overline{b_+(B_+)}\cap Crit(f)=\varnothing$.  

We intend to show that $\beta^{alg}_{k}(f;\K)> \alpha$.  If $\alpha< 0$ this is obvious, since by definition $\beta^{alg}_{k}(f;\K)\geq 0$, so assume $\alpha\geq  0$.  So $\min f|_{\overline{b_-(B_-)}}>\max f|_{\overline{b_+(B_+)}}$, and (with respect to a suitably generic metric in order to define the relelvant operations on the Morse complex $CM_*(f;\K)$) we may apply Proposition \ref{linklinkprop}.  This gives elements $I_{b_-} M_{-f}\in Im(d_{-f,n-k})$ and $I_{b_+} M_f\in Im(d_{f,k+1})$ such that $\Lambda(I_{b_-} M_{-f},I_{b_+}M_f)=lk_{\K}(b_-,b_+)\neq 0$ (the other term in (\ref{linklink}) vanishes by Remark \ref{nomassey}).  Now the fact that $f$ decreases along its negative gradient flowlines is easily seen to imply that \[ \ell_f(I_{b_+} M_f)< \max f|_{\overline{b_+(B_+)}},\] since the critical points contributing to the Morse chain $I_{b_+} M_f$ are the limits in positive time of negative gradient flowlines of $f$ that pass through the image of $b_+$.  Similarly we have \[ \ell_{-f}(I_{b_-} M_{-f})< \max \left(-f|_{\overline{b_-(B_-)}}\right)=-\min f|_{\overline{b_-(B_-)}}.\] Thus \[ -\ell_{-f}(I_{b_-} M_{-f})-\ell_f(I_{b_+} M_f)>\min f|_{\overline{b_-(B_-)}}-\max f|_{\overline{b_+(B_+)}}>\alpha.\]  Since $I_{b_-} M_{-f}$ and $I_{b_+} M_f$ have nontrivial linking pairing over $\K$ this shows that $\beta^{alg}(f;\K)>\alpha$.  So since $\alpha$ was an arbitrary nonnegative number smaller than $\beta^{geom}_{k}(f;\K)$ this proves that $\beta^{alg}_{k}(f;\K)\geq \beta^{geom}_{k}(f;\K)$.
 \end{proof}

\section{From critical points to linked pseudoboundaries}\label{crittolink}

We now turn attention to the proof of the inequality $\beta^{alg}\leq \beta^{geom}$ in Theorem \ref{alggeom}, and to the implications ``(i)$\Rightarrow$(ii)'' in Theorems \ref{main1} and \ref{main2}. Throughout this section we fix a Morse function $f\co M\to\R$ where $M$ is a compact $n$-dimensional manifold without boundary, and we fix a Riemannian metric $h$ such that the gradient flow of $f$ with respect to $h$ is Morse--Smale; we moreover assume that the pair $(f,h)$ is \emph{locally trivial} in the sense that around each critical point $p$ there are coordinates $(x_1,\ldots,x_n)$ such that $f(x_1,\ldots,x_n)=f(p)-\sum_{i=1}^{k}x_{i}^{2}+\sum_{i=k+1}^{n}x_{i}^{2}$ and such that $h$ is given by the standard Euclidean metric in some coordinate ball around the origin.  Metrics which simultaneously have this local triviality property and make the gradient flow of $f$ Morse--Smale exist in abundance by \cite[Proposition 2]{BH} (note that, in constrast to our usage, the definition of ``Morse--Smale'' that is used in \cite{BH} already incorporates the local triviality property).  Our purpose in assuming  local triviality is that, by \cite[Theorem 1(2)]{BH}, it guaranteees that the standard broken-flowline compactification of the unstable manifolds is a smooth manifold with corners (indeed, with faces), with the evaluation map extending smoothly up to the corners.

\subsection{Manifolds with corners}
Let us briefly recall some facts about manifolds with corners; see \cite{Do},\cite[Section 1.1]{J} for more details.  An $n$-dimensional smooth manifold with corners is by definition a second-countable Hausdorff space $X$ locally modeled on open subsets of $[0,\infty)^n$, with smooth transition functions.   For $x\in X$ and a coordinate patch $\phi\co U\to [0,\infty)^n$ with $x\in U$, the number of coordinates of $\phi(x)$ which are equal to $0$ is independent of the choice of coordinate patch $\phi$, and will be denoted by $c(x)$.  For $k\in\{0,\ldots,n\}$, the subset $\partial^{{}^{\circ}k} X =\{x\in X|c(x)=k\}$ is an $(n-k)$-dimensional smooth manifold.  One has $\overline{\partial^{{}^{\circ}k} X}=\cup_{l=k}^{n}\partial^{{}^{\circ}l}X$, and $\partial^{{}^{\circ}k} X$ is open as a subset of $\overline{\partial^{{}^{\circ}k} X}$.  Of course, $X\setminus \cup_{k\geq 2}\partial^{{}^{\circ}k} X$ is naturally a manifold with boundary.

We intend to build pseudochains and pseudoboundaries out of maps defined on manifolds with corners; since both of the former have domains which do not have corners the following will be useful.

\begin{lemma}\label{resolve}  Let $X$ be an $n$-dimensional manifold with corners such that $\partial^{{}^{\circ}k} X=\varnothing$ for all $k\geq 3$.  Then there is a smooth manifold with boundary $X'$ and a smooth homeomorphism $\pi\co X'\to X$ which restricts to $\pi^{-1}(X\setminus \partial^{{}^{\circ}2} X)$ as a diffeomorphism between $\pi^{-1}(X\setminus \partial^{{}^{\circ}2} X)$ and $X\setminus \partial^{{}^{\circ}2} X$.
\end{lemma}
 (Of course, if $\partial^{{}^{\circ}2} X\neq \varnothing$, the inverse $\pi^{-1}$ must not be smooth.)

\begin{proof}

The manifold $X'$ will be formed by removing $\partial^{{}^{\circ}2} X$ and then gluing in a smooth manifold with boundary which is homeomorphic to a tubular neighborhood of $\partial^{{}^{\circ}2}X$.

In this direction, note that the structure group of the normal bundle $E$ to $\partial^{{}^{\circ}2} X$ reduces to that subgroup $G$ of $O(2)$ which preserves the quadrant $\{(x,y)\in \R^2|x\geq 0,y\geq 0\}$.   Of course $G$ is just given by $G=\left\{\left(\begin{array}{cc}1 & 0\\ 0& 1\end{array}\right),\left(\begin{array}{cc}0 & 1\\ 1& 0\end{array}\right)\right\}$. In other words, there is a principal $G$-bundle $P\to \partial^{{}^{\circ}2} X$ with $E$ given as the associated bundle \[ E=P\times_G \R^2=\frac{P\times \R^2}{(pg,v)\sim (p,gv)} \] (Geometrically, given a Riemannian metric on $X$, the fiber of $P$ over a point $x\in\partial^{{}^{\circ}2} X$ can be identified with the pair of unit vectors which are normal to $\partial^{{}^{\circ}2}X$ and tangent to $\overline{\partial^{{}^{\circ}1} X}$.)  

Write $Q=\{(x,y)\in\R^2|x\geq 0,y\geq 0\}$ and $\mathbb{H}=\{(x,y)\in \R^2|x+y\geq 0\}$, so the standard action of $G$ on $\R^2$ restricts to actions on both $Q$ and $\mathbb{H}$. Moreover, there exists a $G$-equivariant smooth homeomorphism $\phi\co \mathbb{H}\to Q$ with $\phi(0,0)=(0,0)$ such that $\phi|_{\mathbb{H}\setminus\{(0,0)\}}$ is a diffeomorphism; for instance, identifying $\R^2$ with $\mathbb{C}$, one can use the map \[ \phi(re^{i\theta})=\beta(r)e^{\frac{i}{2}\left(\theta+\frac{\pi}{4}\right)} \quad\left(\mbox{for}-\frac{\pi}{4}\leq\theta\leq \frac{3\pi}{4}\right),\] where $\beta\co \R\to\R$ is a smooth surjective map with $\beta'(r)>0$ for all $r\neq 0$ such that $\beta$ vanishes to infinite order at $r=0$.

Now the normal cone to $\partial^{{}^{\circ}2} X$ in $X$ (\emph{i.e.}, the subset of the normal bundle $E$ consisting of tangent vectors $\gamma'(0)$ to smooth curves $\gamma\co [0,1)\to X$ with $\gamma(0)\in\partial^{{}^{\circ}2} X$) is naturally identified with the associated bundle $B=P\times_G Q$ over $\partial^{{}^{\circ}2} X$, with fiber the quadrant $Q$.  By a special case of \cite[Th\'eor\`eme 1]{Do}, there is a neighborhood $N\subset X$ of $\partial^{{}^{\circ}2} X$ and a diffeomorphism $\Psi\co N\to B$, which restricts to $\partial^{{}^{\circ}2} X$ as the standard embedding of the zero-section.

Now form the associated bundle $C=P\times_G \mathbb{H}$; this has an obvious manifold-with-boundary structure, with $\partial C=\{[p,(x,y)]\in C|x+y=0\}$.  Where $\phi\co\mathbb{H}\to Q$ is as above, the $G$-equivariance of $\phi$ implies that we have a well-defined map $\tilde{\phi}\co C\to B$ defined by $\tilde{\phi}[p,h]=[p,\phi(h)]$; evidently $\tilde{\phi}$ is a smooth homeomorphism which restricts to the complement of $\{[p,(0,0)]\}\subset C$ as a diffeomorphism to the complement of $\{[p,(0,0)]\}\subset B$.

The assumption that $\partial^{{}^{\circ}k} X=\varnothing$ for all $k\geq 3$ implies that $\partial^{{}^{\circ}2} X$ is a closed subset of $X$.  We now define \[ X'=\frac{C\coprod \left(X\setminus \partial^{{}^{\circ}2} X\right)}{c\sim \Psi^{-1}(\tilde{\phi}(c))\mbox{ if }\tilde{\phi}(c)\in \Psi(N\setminus \partial^{{}^{\circ}2} X)}.\]  Since $\Psi^{-1}\circ \tilde{\phi}$ restricts to the open set $\tilde{\phi}^{-1}(\Psi(N\setminus\partial^{{}^{\circ}2} X))\subset C$ as a diffeomorphism to its image, which is open in $X$, and since $C$ and $X\setminus \partial^{{}^{\circ}2} X$ are both manifolds with boundary (and without corners), $X'$ inherits the structure of a manifold with boundary from $C$ and $X\setminus \partial^{{}^{\circ}2}X$.  The desired map $\pi\co X'\to X$ is then obtained by setting $\pi$ equal to $\Psi^{-1}\circ\tilde{\phi}$ on $C$ and equal to the inclusion on $X\setminus\partial^{{}^{\circ}2}X$.  
\end{proof} 

If $X$ is a manifold with corners, following \cite{J}, a \emph{connected face} of $X$ is by definition the closure of a connected component of $\partial^{{}^{\circ}1}X$.  $X$ is then said to be a \emph{manifold with faces} if every point $x\in X$ belongs to $c(x)$ distinct connected faces (said differently, if $U$ is a small connected coordinate neighborhood of $x$ then the inclusion-induced map $\pi_0(U\cap\partial^{{}^{\circ}1}X)\to \pi_0(\partial^{{}^{\circ}1}X)$ should be injective). A \emph{face} of a manifold with faces is a (possibly empty) union of pairwise disjoint faces. If $X$ is a manifold with faces and if $F\subset X$ is a face then $F$ inherits the structure of a manifold with corners,  with $\partial^{{}^{\circ}k}F=F\cap \partial^{{}^{\circ}k+1}X$.  

\begin{lemma}\label{fuse} Let $X$ be a manifold with faces, let $F_-,F_+\subset X$ be two disjoint faces of $X$, and let $\phi\co F_-\to F_+$ be a diffeomorphism.  Then the topological space \[ X^{\phi}=\frac{X}{x\sim\phi(x)\mbox{ if }x\in F_-} \] may be endowed with the structure of a smooth manifold with corners in such a way that, where $\pi\co X\to X^{\phi}$ is the quotient projection, for any other smooth manifold $Y$ and any smooth map $g\co X\to Y$ such that $g(x)=g(\phi(x))$ for all $x\in F_-$, the unique map $\bar{g}\co X^{\phi}\to Y$ obeying  $g=\bar{g}\circ \pi$ is smooth.  The corner strata of $X^{\phi}$ are determined by \[ \overline{\partial^{{}^{\circ}k}X^{\phi}}=\overline{\pi(\partial^{{}^{\circ}k}X\setminus(F_-\cup F_+))}.\]

Moreover, if $X$ is oriented and if $\phi\co F_-\to F_+$ is orientation-reversing with respect to the induced boundary orientations on $F_{\pm}$, then $X^{\phi}$ carries an orientation such that $\pi|_{ X\setminus(F_-\cup F_+)}$ is an orientation-preserving diffeomorphism onto its image.
\end{lemma}

\begin{proof}
The faces $F_{\pm}$ are, in the sense of \cite{Do}, submanifolds without relative boundary of $X$ having coindex and codimension both equal to $1$; consequently the tubular neighborhood theorem \cite[Th\'eor\`eme 1]{Do} applies to give diffeomorphisms $\Phi_{\pm}\co (-1,0]\times F_{\pm}\to U_{\pm}$ where $U_{\pm}$ is a neighborhood of $F_{\pm}$ with $U_+\cap U_-=\varnothing$, $\Phi_{\pm}|_{\{0\}\times F_{\pm}}$ restricts as the identity map to $F_{\pm}$, and $(-1,0]\times F_{\pm}$ is endowed with its obvious product manifold-with-corners structure.  If $X$ is oriented then $\Phi_{\pm}$ will necessarily be orientation preserving with respect to the standard product orientation on $(-1,0]\times F_{\pm}$.  

Given these tubular neighborhoods, the lemma is a straightforward generalization of a standard gluing construction from the theory of manifolds without corners; we briefly indicate the argument, leaving details to the reader.  Let $V=(-1,1)\times F_-$ and $F=\{0\}\times F_-\subset V$.  Let $\beta\co (-1,1)\to (-1,1)$  be a smooth homeomorphism such that $\beta(t)=t$ for $|t|>1/2$, $\beta'(t)>0$ for all $t\neq 0$, and $\beta$ vanishes to infinite order at $t=0$.  We can then define a diffeomorphism $\Psi\co V\setminus F\to (U_-\setminus F_-)\cup (U_+\setminus F_+)$ by $\Psi(t,x)=\Phi_-(\beta(t),x)$ for $t<0$ and $\Psi(t,x)=\Phi_+(-\beta(t),\phi(x))$ for $t>0$.  Then \[ \frac{(X\setminus (F_-\cup F_+))\coprod V}{v\sim \Psi(v)\mbox{ for }v\in V\setminus F} \] inherits the structure of a smooth manifold with corners, and is clearly homeomorphic to $X^{\phi}$.  The various required properties are easy to check; we just note that, if $g\co X\to Y$ is a smooth map with $g|_{F_-}=g\circ \phi$, then the induced map $\bar{g}\co X^{\phi}\to Y$ restricts to $V$ as the map \[ (t,x)\mapsto\left\{\begin{array}{ll} (g\circ \Phi_-)(\beta(t),x) & \mbox{ if }t\leq 0, \\ (g\circ\Phi_+)(-\beta(t),\phi(x)) & \mbox{ if } t\geq 0.\end{array}\right.\]  This map is smooth along $F$ by virtue of the facts that $g|_{F_{\pm}}$ is smooth and that $\beta$ vanishes to infinite order at $t=0$, so that the derivatives of all orders of $\bar{g}$ in directions normal to $F$  vanish as well.
\end{proof}

\subsection{Constructing pseudochains from Morse chains} 
Our Morse--Smale pair $(f,h)$ where $h$ is locally trivial 
determines Morse complexes $CM_{*}(\pm f;\K)$ and stable and unstable manifolds $W^{s}_{f}(p)=W^{u}_{-f}(p)$ and $W^{u}_{f}(p)=(-1)^{|p|_f(n-|p|_f)}W^{s}_{-f}(p)$, oriented as in Section \ref{or:morse}.  We intend to construct, for any given pair $b_-\in d_{f,n-k}(CM_{n-k}(-f;\K))$, $b_+\in d_{f,k+1}(CM_{k+1}(f;\K))$ with $\Lambda(b_-,b_+)\neq 0$, a corresponding pair of pseudoboundaries $\beta_-\co B_-\to M$, $\beta_+\co B_+\to M$ such that $lk_{\K}(\beta_-,\beta_+)=\Lambda(b_-,b_+)$ and $\min(f|_{\overline{\beta_-(B_-)}})-\max(f|_{\overline{\beta_+(B_+)}})=-\ell_{-f}(b_-)-\ell_f(b_+)$.  This construction generalizes one found in \cite[Section 4]{S99}, in which Schwarz associates a pseudocycle to any Morse cycle.  Before formulating the key lemma we introduce a definition:

\begin{definition} Let $X,Y,Z$ be smooth oriented manifolds, possibly with boundary, let $f\co X\to Z$ and $g\co Y\to Z$ be smooth maps, and $z\in Z$.  We say that \emph{$f$ is coincident to $g$ near $z$} if there is a neighborhood $U$ of $z$ and an orientation-preserving diffeomorphism $\phi\co f^{-1}(U)\to g^{-1}(U)$ such that $f|_{f^{-1}(U)}=g\circ\phi$.
\end{definition}

Also, as a point of notation, if $X$ is an oriented manifold and $m\in\mathbb{Z}$ we denote by $mX$ the oriented manifold obtained by taking $|m|$ disjoint copies of $X$, all oriented in the same way as $X$ if $m>0$ and oriented oppositely to $X$ if $m<0$.

For any $j\in \N$ let $Crit_j(f)$ denote the collection of index-$j$ critical points of $f$.

\begin{lemma}\label{chainconstruct}  Let $a=\sum_{i=1}^{l}a_i p_i\in CM_{k+1}(f;\Z)$, with $d_{f,k+1}a=\sum_{j=1}^{m} z_j q_j$, where we assume all $a_i$ and $z_j$ are nonzero and the $p_i$ and $q_j$ are all distinct.  Then there is a smooth map $\alpha_a\co Y_a\to M$, where $Y_a$ is a smooth oriented $(k+1)$-manifold with boundary, having the following properties:
\begin{itemize} \item[(i)] $\alpha_a$ is a $(k+1)$-pseudochain, and $\alpha_a|_{\partial Y_a}$ is a $k$-pseudoboundary.
\item[(ii)] \[ \overline{\alpha_a(Y_a)}\subset \bigcup_{p\in Crit(f),|p|_f\leq k+1} W^{u}_{f}(p) \quad\mbox{and}\quad \overline{\alpha_a(\partial Y_a)}\subset \bigcup_{q\in Crit(f),|q|_f\leq k} W^{u}_{f}(q) \]
\item[(iii)] For each $i$, $\alpha_a$ is coincident near $p_i$ to the map $\coprod a_i W^{u}_{f}(p_i)\to M$ which is equal to the inclusion on each component of the domain.  Similarly, for each $j$, $\alpha_{a}|_{\partial Y_a}$ is coincident near $q_j$ to the map $\coprod z_j W^{u}_{f}(q_j)\to M$ which is equal to the inclusion on each component of the domain.
\item[(iv)] If $p\in Crit_{k+1}(f)\setminus \{p_1,\ldots,p_l\}$ then $p\notin \overline{\alpha_a(Y_a)}$. Similarly, if $q\in Crit_{k}(f)\setminus \{q_1,\ldots,q_m\}$ then $q\notin  \overline{\alpha_a(\partial Y_a)}$.
\item[(v)] \[ \max(f|_{\overline{\alpha_a(\partial Y_a)}})=\max\{f(q_j)|j=1,\ldots,m\}.\]
\end{itemize}
\end{lemma}

\begin{proof} Following \cite{S99}, let $\Delta a$ denote the compact oriented zero-manifold obtained as a disjoint union of $a_i$-many copies of each of the oriented zero-manifolds $\mathcal{M}(p_i,q)$, as $i$ varies from $1$ to $l$ and as $q$ varies through $Crit_k(f)$.  For $q_0\in Crit_k(f)$ write $\Delta a(q_0)$ for the oriented zero-submanifold of $\Delta a$ consisting of the copies of those $\mathcal{M}(p_i,q;f)$ with $q=q_0$.  Thus we have \[ d_{f,k+1}a=\sum_{q\in Crit_j(f)}\#\left(\Delta a(q)\right) q,\] and so \[ \#\left(\Delta a(q)\right)=\left\{\begin{array}{ll} z_j & \mbox{if }q=q_j, \\ 0 &\mbox{otherwise}.\end{array}\right.\]

Now it is a general combinatorial fact that, if $S$ is a compact oriented zero-manifold, then an equivalence relation may be constructed on $S$ so that $|\#(S)|$-many of the equivalence classes are singletons (oriented consistently with $sign(\#(S))$) and the rest of the equivalence classes are two-element sets $\{s_-,s_+\}$ where $s_-$ is negatively oriented and $s_+$ is positively oriented.  Choose such an equivalence relation on each of the oriented zero-manifolds $\Delta a(q)$, and let $\sim_{\Delta}$ denote the union of these equivalence relations, so that $\sim_{\Delta}$ is an equivalence relation on $\Delta a$.  For $i=1,2$ let $\Delta^ia(q)$ denote the set of elements of $\Delta a(q)$ whose equivalence class has cardinality $i$, and let $\Delta^i a=\cup_{q}\Delta^ia(q)$, so $\Delta a=\Delta^1 a\cup \Delta^2a$.

The disjoint union $\coprod_{i=1}^{l}a_iW^{u}_{f}(p_i)$ has a broken-trajectory compactification $\widehat{Y}$ as  in \cite[Theorem 1(2)]{BH} which is a smooth compact manifold with faces and a smooth evaluation map;  a general codimension-$c$ connected stratum of this compactification is given by a connected component of a product $\mathcal{M}(p_i,r_1;f)\times\mathcal{M}(r_1,r_2;f)\times \cdots \mathcal{M}(r_{c-1},r_c;f)\times W^{u}_{f}(r_c)$ where $|r_c|_f<\cdots<|r_1|_f<|p_i|_f$, with the evaluation map restricting to the stratum as the natural embedding of $W^{u}_{f}(r_c)$  (and, of course, we take $a_i$ copies of each of these strata).  Here and below a ``codimension-$c$ connected stratum'' of a manifold with corners $X$ refers to a connected component of $\partial^{{}^{\circ}c} X$, and a ``codimension-$c$ stratum'' is a disjoint union of codimension-$c$ connected strata.   We will first form a manifold with faces $Y_0$, defined to be the open subset of $\widehat{Y}$ given as the union of the following types of strata:
\begin{itemize} \item[(0)] All of the codimension-zero strata (\emph{i.e.}, $a_i$ copies of $W^{u}_{f}(p_i)$ for each $i$);
\item[(1A)] Those codimension-one strata of the form $\mathcal{M}(p_i,q;f)\times W^{u}_{f}(q)$ where $|q|_f=k$;
\item[(1B)] Those codimension-one strata of the form $\mathcal{M}(p_i,r;f)\times W^{u}_{f}(r)$ where $|r|_f=k-1$ and where, for some $j$, we have $\mathcal{M}(q_j,r;f)\neq\varnothing$.
\item[(2)] Those codimension-two strata of the form $\mathcal{M}(p_i,q;f)\times \mathcal{M}(q,r;f)\times W^{u}_{f}(r)$ where $|q|_f=k$ and $|r|_f=k-1$ is such that, for some $j$, we have $\mathcal{M}(q_j,r;f)\neq\varnothing$.  
\end{itemize}

(The fact that this is indeed open in $\widehat{Y}$ follows from the fact that the connected faces which contain any of the strata in (2) are closures of connected components of strata appearing in (1A) or (1B).)  

Among the connected faces of the manifold with corners $Y_0$ are the closures $\overline{\{\gamma\}\times W^{u}_{f}(q)}$ where $\gamma\in \Delta a$; an element of such a closure is represented by a broken trajectory whose first component is $\gamma$, and so all of these faces are disjoint as $\gamma$ varies through $\Delta a$.  Let $F_-$ be the union of the connected faces $\overline{\{\gamma\}\times W^{u}_{f}(q)}$ as $\gamma$ varies through those elements of $\Delta^2 a$ which are negatively oriented, and let $F_+$ be the union of the connected faces 
$\overline{\{\gamma\}\times W^{u}_{f}(q)}$ as $\gamma$ varies through those elements of $\Delta^2 a$ which are positively oriented.  Our
equivalence relation $\sim_{\Delta}$ induces an orientation-reversing diffeomorphism $\phi\co F_-\to F_+$ which maps $\{\gamma\}\times W^{u}_{f}(q)$ to $\{\gamma'\}\times W^{u}_{f}(q)$ by the identity on $W^{u}_{f}(q)$ whenever $\gamma\sim_{\Delta}\gamma'$ and $\gamma$ is negatively-oriented while $\gamma'$ is positively oriented.  Thus we may apply Lemma \ref{fuse} to glue $F_-$ to $F_+$, resulting in a new oriented manifold with corners $Y_{0}^{\phi}$.  

The faces of $Y_{0}^{\phi}$ include (the images under the projection $\pi\co Y_{0}\to Y_{0}^{\phi}$ of) the faces $\overline{\{\gamma\}\times W^{u}_{f}(q)}$ where $\gamma\in \Delta^1a$ (and so $q=q_j$ for some $j$), as well as unions of images under $\pi$ of faces \linebreak $\overline{\mathcal{M}(p,r;f)\times W^{u}_{f}(r)}$ where $|r|_f=k-1$ and $\mathcal{M}(p_j,r;f)\neq\varnothing$ (in some cases, different faces of this form have been joined together along their boundary by the gluing process that created $Y_{0}^{\phi}$ from $Y_0$).  

Lemma \ref{resolve} then gives a smooth oriented manifold with boundary $Y_a$ and a smooth homeomorphism $\pi_1\co Y_a\to Y_{0}^{\phi}$. Since the evaluation map $E\co Y_0\to M$ descends to a smooth map $\bar{E}\co Y_{0}^{\phi}\to M$ by Lemma \ref{fuse}, the composition $\alpha_a=\bar{E}\circ \pi_1\co Y_a\to M$ is smooth.  We will now show that $\alpha_a$ is a pseudochain and that $\alpha_a|_{\partial Y_a}$ is a pseudoboundary.  In other words we must show that the $\Omega$-limit sets $\Omega_{\alpha_a}$ and $\Omega_{\alpha_a|_{\partial Y_a}}$ have dimensions at most $k-1$ and $k-2$ respectively.  

Now evidently $\Omega_{\alpha_a}=\Omega_{E}$ and $\Omega_{\alpha_a|_{\partial Y_a}}=\Omega_{E|_{\pi^{-1}(\pi_1(\partial Y_a))}}$.  Any divergent sequence in $Y_0$ has a subsequence which converges in the compactification $\widehat{Y}$ to a point which is sent by the evaluation map to an element of an unstable manifold $W^{u}_{f}(s)$ where $|s|_f\leq k-1$; it quickly follows from this that $\Omega_{E}$ (and hence also $\Omega_{\alpha_a}$) has dimension at most $k-1$.

As for $\Omega_{\alpha_a|_{\partial Y_a}}=\Omega_{E|_{\pi^{-1}(\pi_1(\partial Y_a))}}$, note that $\pi_1(\partial Y_a)$ is just the union of the boundary and corner strata of $Y_{0}^{\phi}$, and so $\pi^{-1}(\pi_1(\partial Y_a))$ is the union of all of the boundary and corner strata of $Y_0$ \emph{except} those of the form $\{\gamma\}\times W^{u}_{f}(q)$ where $\gamma\in \Delta^2a(q)$.  If $\{x_n\}_{n=1}^{\infty}$ is a divergent sequence in $\pi^{-1}(\pi_1(\partial Y_a))$, then after passing to a subsequence either each $x_n$ belongs to some $\{\gamma\}\times \overline{W^{u}_{f}(q_j)}$ where $\gamma\in \Delta^1a(q_j)$ (and where the closure is taken in $Y_0$, not in $\widehat{Y}$), or else each $x_n$ belongs to some $\overline{\mathcal{M}(p_i,r;f)\times W^{u}_{f}(r)}$ where $|r|_f=k-1$ and where $\mathcal{M}(q_j,r;f)\neq\varnothing$ for some $j$.  Now in view of the codimension-two strata that were included in $Y_0$ (all of which are still contained in $\pi^{-1}(\pi_1(\partial Y_a))$, though some of them will project to subsets of $\partial^{{}^{\circ}1}Y_{0}^{\phi}$), if such a sequence diverges in $\pi^{-1}(\pi_1(\partial Y_a))$ then, considering it now as a sequence in the compact space $\widehat{Y}$, it must have a subsequence which converges to a point which is sent by the evaluation map to an element of an unstable manifold $W^{u}_{f}(s)$ where $|s|_f\leq k-2$.  Thus indeed $\Omega|_{\alpha|_{\partial Y_a}}$ has dimension at most $k-2$.

We have now proven property (i) of Lemma \ref{chainconstruct}; the other properties follow quickly from the construction.  Indeed property (ii) follows directly from the facts that $\alpha_a(Y_a)\subset E(Y_0)$, that $\alpha_a(\partial Y_a)\subset E(\overline{\partial^{{}^{\circ}1} Y_0})$, and that for $p\in Crit_l(f)$ the closure of $W^{u}_{f}(p)$ is (thanks in part to the Morse--Smale property) contained in the union of unstable manifolds of critical points of index at most $l$.  This latter fact also implies that for each $\gamma\in \Delta^2 a(q)$ the face $\{\gamma\}\times W^{u}_{f}(q)$ is disjoint from some neighborhood $V$ of the index $k+1$ critical points, and therefore the evaluation maps $E\co Y_0\to M$ and $\bar{E}\co Y_{0}^{\phi}\to M$ are coincident near each $p\in Crit_{k+1}(f)$.  Moreover the region on which $\pi_1\co Y_a\to Y_{0}^{\phi}$ fails to be a diffeomorphism (namely, the preimage of the corner locus of $Y_{0}^{\phi}$) is also disjoint from a neighborhood of $\alpha_{a}^{-1}(Crit_{k+1}(f))$, in view of which $\alpha_a$ is coincident to $\bar{E}$, and so also to $E$, near each $p\in Crit_{k+1}(f)$.  This immediately implies the first sentences of both (iii) and (iv).

The second sentences of (iii) and (iv) follow similarly, since any point of $\partial Y_a$ which is mapped to a suitably small neighborhood of $Crit_k(f)$ is contained in the preimage under $\pi_1$ of the image under $\pi$ of a face of the form $\{\gamma\}\times W^{u}_{f}(q_j)$ where $\gamma\in \Delta^1 a(q_j)$, and $\pi_{1}^{-1}\circ\pi$ is an orientation-preserving diffeomorphism onto its image when restricted to such a face.  

Finally, $\alpha_a(\partial Y_a)$ contains each of the points $q_j$ since we assume $z_j\neq 0$ for all $j$, while any point   $x\in\alpha_a(\partial Y_a)$ lies either on an unstable manifold $W^{u}_{f}(q_j)$ or on an unstable manifold $W^{u}_{f}(r)$ where $\mathcal{M}(q_j,f;r)\neq \varnothing$ for some $j$.  Since $f$ decreases along its negative gradient flowlines, in either case we will have $f(x)\leq f(q_j)$ for some $j$, proving (v).
\end{proof}

\begin{prop}\label{twoboundaries}
Let $a_-\in CM_{n-k}(-f;\Z)$ and $a_+\in CM_{k+1}(f;\Z)$, giving via Lemma \ref{chainconstruct} pseudochains $\alpha_{a_-}\co Y_{a_-}\to M$ and $\alpha_{a_+}\co Y_{a_+}\to M$ (using the Morse function $-f$ for the former and $f$ for the latter).  Write $B_{\pm}=\partial Y_{a_{\pm}}$, so that $b_-:=\alpha_{a_-}|_{B_-}$ is a $(n-k-1)$-pseudoboundary and $b_+:=\alpha_{a_+}|_{B_+}$ is a $k$-pseudoboundary.  These pseudoboundaries satisfy the following properties:
\begin{itemize} \item[(i)] $\min(f|_{\overline{ b_-(B_-)}})-\max(f|_{\overline{b_+(B_+)}})=-\ell_{-f}(d_{-f,n-k}a_-)-\ell_f(d_{f,k+1}a_+)$.
\item[(ii)] The linking number of the pseudoboundaries $b_-$ and $b_+$ is well-defined, and given by \[ lk(b_-,b_+)=\Lambda(d_{-f,n-k}a_-,d_{f,k+1}a_+).\]
\item[(iii)] For all $\phi$ belonging to a  $C^{\infty}$-residual subset of $Diff(M)$, our given Morse--Smale locally trivial Riemannian metric $h$ is generic with respect to $f,\phi\circ b_+,\phi\circ b_-$ in the sense of Definition \ref{genwrt}, so we have a well-defined map $I_{\phi\circ b_+,\phi\circ b_-}\co CM_{n}(f;\K)\to CM_{0}(f;\K)$.  If additionally $\phi$ is sufficiently $C^1$-close to the identity then $I_{\phi\circ b_+,\phi\circ b_-}$ is equal to zero.
\end{itemize}
\end{prop}

\begin{proof} Write $z_-=d_{-f,n-k}a_-$ and $z_+=d_{f,k+1}a_+$.  The statement (i) follows directly from Lemma \ref{chainconstruct}(v), as
\[ \max(\pm f|_{\overline{b_{\pm}(B_{\pm}) }})=\ell_{\pm f}(b_{\pm}),\] and so \[ \min(f|_{\overline{b_-(B_-)}})-\max(f|_{\overline{b_+(B_+)}})=-\ell_{-f}(z_-)-\ell_f(z_+).\]

Turning to (ii),  by Lemma \ref{chainconstruct}(ii) $\overline{b_+(B_+)}$ is contained in the union of the unstable manifolds of the critical points of $f$ with index at most $k$, while $\overline{b_-(B_-)}$ is contained in the union of the unstable manifolds of the critical points of $-f$ with index at most $n-k-1$ (\emph{i.e.}, the stable manifolds of the critical points of $f$ with index at least $k+1$).  The Morse--Smale condition therefore implies that $\overline{b_+(B_+)}\cap  \overline{b_-(B_-)}=\varnothing$, and so these two pseudoboundaries have a well-defined linking number, given by \[ lk_{\K}(b_-,b_+)=\#(Y_{a_+}{}_{\alpha_{a_+}}\times_{b_-}B_-).\]  
Now $\overline{\alpha_{a_+}(Y_{a_+})}$ is contained in the union of the unstable manifolds of critical points of $f$ with index at most $k+1$; again by the Morse--Smale condition we have, if $p,q\in Crit(f)$ obey $|p|_f\leq k+1\leq |q|_f$, then \[ W^{u}_{f}(p)\cap W^{s}_{f}(q)=\left\{\begin{array}{ll} \{p\} & \mbox{if }p=q\mbox{ and }|p|_f=|q|_f=k+1, \\ \varnothing & \mbox{otherwise}. \end{array}\right.\]  Let us write $a_+=\sum_i a_{i,+}p_i$ and $z_-=\sum_j z_{j,-}q_j$.  It then follows from Lemma \ref{chainconstruct}(iii) and (iv) and the fact that $W^{u}_{-f}(q_j)=W^{s}_{f}(q_j)$ as oriented manifolds that 
\[ \#(Y_{a_+}{}_{\alpha_{a_+}}\times_{b_-}B_-)=\sum_{i,j}a_{i,+}z_{j,-}\#_{\K}(W^{u}_{f}(p_i){}_{i_{u,p_i}}\times_{i_{s,q_j}}W^{s}_{f}(q_j)).\]  By our orientation conventions and index considerations, $W^{u}_{f}(p_i){}_{i_{u,p_i}}\times_{i_{s,q_j}}W^{s}_{f}(q_j)$  consists of a single positively-oriented point if $p_i=q_j$ and is empty otherwise.  We thus have \[ lk(b_-,b_+)=\sum_{\{(i,j)|p_i=q_j\}}a_{i,+}z_{j,-}=\Pi(z_-,a_+)=\Lambda(z_-,z_+), \] proving (ii).

As for (iii), the fact that  $h$ is generic with respect to $f,\phi\circ b_+,\phi\circ b_-$ for a $C^{\infty}$-residual set of $\phi\in Diff(M)$ follows straightforwardly by applying Lemma \ref{diffu} to the various relevant fiber products.  If the final statement of the proposition were false, then we could find a sequence $\{\phi_n\}_{n=1}^{\infty}$ in $Diff(M)$ which $C^1$-converges to the identity, critical points $p,q\in Crit(f)$, and a sequence $(\gamma_n,T_n)\in\tilde{\mathcal{M}}(p,q;f)\times(0,\infty)$ such that $\gamma_n(0)\in  \phi_n(\overline{b_+(B_+)})$ and $\gamma_n(T_n)\in \phi_n(\overline{b_-(B_-)})$.  A standard compactness result (\emph{e.g.} \cite[Proposition 2.35]{S93}) would then give a possibly-broken Morse trajectory for $f$ which passes first through $\overline{b_+(B_+)}$ and then, either strictly later or at precisely the same time, through $\overline{b_-(B_-)}$.  But since $\overline{b_+(B_+)}$ is contained in the union of the unstable manifolds of critical points with index at most $k$, while $\overline{b_-(B_-)}$
is contained in the union of the stable manifolds of critical points with index at least $k+1$, this is forbidden by the Morse--Smale property.  This contradiction completes the proof.
\end{proof}

We can now finally complete the proof of Theorem \ref{alggeom} and thus Theorem \ref{main1}.  For clarity we will, unlike elsewhere in the paper, incorporate the ring over which we are working into the notation for the Morse boundary operator and the Morse-theoretic linking pairing: thus we have maps $d_{f,k+1}^{\K}\co CM_{k+1}(f;\K)\to CM_{k}(f;\K)$ and $\Lambda_{\K}\co Im(d_{-f,n-k}^{\K})\times Im(d_{f,k+1}^{\K})\to\K$. We first make the following almost-obvious algebraic observation:

\begin{lemma} \label{extboundary} Let $0\neq z\in d_{f,k+1}^{\K}(CM_{k+1}(f;\K))$.  Then there are $z_1,\ldots,z_N\in d_{f,k+1}^{\Z}(CM_{k+1}(f;\Z))$ and $r_1,\ldots,r_N\in \K$ such that $z=\sum_{i=1}^{N}z_i\otimes r_i$ and $\ell_f(z_i)\leq \ell_f(z)$ for all $z$. 
\end{lemma}

\begin{proof}  The lemma amounts to the statement that, for all $\lambda\in\R$, the natural map \[ \left(Im(d_{f,k+1}^{\Z})\cap CM^{\lambda}_{k}(f;\Z)\right)\otimes \K \to Im(d_{f,k+1}^{\K})\cap CM^{\lambda}_{k}(f;\K) \] is surjective.  Write $A=Im(d_{f,k+1}^{\Z})$ and $B=CM_{k}^{\lambda}(f;\Z)$ and view them as submodules of the $\Z$-module $CM_{k}(f;\Z)$; we then have $Im(d_{f,k+1}^{\K})=A\otimes \K$ and $CM^{\lambda}_{k}(f;\K)=B\otimes \K$, and so we wish to show that the natural map \[ j_{\K}\co (A\cap B)\otimes \K\to (A\otimes\K)\cap (B\otimes\K) \] is surjective.  But this is true on quite general grounds: there is a short exact sequence \[ 0\to A\cap B\to A\oplus B\to A+B\to 0 \] where the first map is $x\mapsto (x,x)$ and the second is $(a,b)\mapsto a-b$.  The right exactness of the tensor product functor then shows that the induced sequence \[ (A\cap B)\otimes \K\to (A\otimes\K)\oplus(B\otimes\K)\to (A+B)\otimes \K\to 0 \] is exact, and exactness at the second term implies that $j_{\K}$ is surjective.
\end{proof}

\begin{proof}[Proof that $\beta^{alg}\leq \beta^{geom}$ in Theorem \ref{alggeom}] 
First of all we observe that, for any nontrivial ring $\K$ and any grading $k$, we have $\beta^{geom}_{k}(f;\K)\geq 0$. Indeed, in any  coordinate chart $U\subset M$ it is straightforward to construct smooth maps $\alpha_-\co B^{n-k}\to U$, $\alpha_+\co B^{k+1}\to U$ (where $B^l$ denotes the closed $l$-dimensional unit ball), the images of whose boundaries are disjoint, such that \linebreak $lk(\alpha_{-}|_{\partial B^{n-k}},\alpha_+|_{\partial B^{k+1}})=1$ (and so since $\K$ is a nontrivial ring $lk_{\K}(\alpha_{-}|_{\partial B^{n-k}},\alpha_+|_{\partial B^{k+1}})\neq 0)$.  For any $\ep>0$, by taking the coordinate chart $U$ so small that $\max f|_{\bar{U}}-\min f|_{\bar{U}}<\ep$ we guarantee that $\min(f|_{\alpha_-(\partial B^{n-k})})-\max(f|_{\alpha_+(\partial B^{k+1})})>-\ep$.  This proves that $\beta^{geom}_{k}(f;\K)\geq 0$.

So for the rest of the proof we may assume that $\beta^{alg}_{k}(f;\K)>0$, since otherwise the inequality $\beta^{alg}_{k}\leq \beta^{geom}_{k}$ is immediate.  Since $\beta^{alg}_{k}(f;\K)$ is independent of the choice of Morse--Smale metric, we may use one which is locally trivial near $Crit(f)$, allowing us to use the constructions of Lemma \ref{chainconstruct}.   Let $a_-\in CF_{n-k}(-f;\K)$ and $a_+\in CF_{k+1}(f;\K)$ be such that, where $z_-=d_{-f,n-k}^{\K}a_-$ and $z_+=d_{f,k+1}^{\K}a_+$, we have $\Lambda_{\K}(z_-,z_+)\neq 0$ (such $a_{\pm}$ do exist, since  $\beta^{alg}_{k}(f;\K)>0$).
By Lemma \ref{extboundary} we may write \[ z_-=\sum_{i=1}^{N_-}z_{-,i}\otimes r_i\qquad z_+=\sum_{i=1}^{N_+}z_{+,i}\otimes s_i \] where $r_i,s_i\in \K$, $z_{-,i}\in Im(d_{-f,n-k}^{\Z})$, $z_{+,i}\in Im(d_{f,k+1}^{\Z})$, and \begin{equation}\label{elldec}\ell_{\pm f}(z_{\pm,i})\leq \ell_{\pm f}(z_{\pm}) \end{equation}  for all $i$. We then have \[ 0\neq \Lambda_{\K}(z_-,z_+)=\sum_{i,j}\Lambda_{\Z}(z_{-,i},z_{+,j})r_is_j,\] so there must be some indices $i_0,j_0$ such that, where $\ep_{\K}\co \Z\to \K$ denotes the unique unital ring morphism, $\ep_{\K}(\Lambda_{\Z}(z_{-,i_0},z_{+,j_0}))\neq 0$.

Applying Proposition \ref{twoboundaries} to $z_{-,i_0}$ and $z_{+,j_0}$ gives an $(n-k-1)$-pseudoboundary $b_-$ and a $k$-pseudoboundary $b_+$ such that  $lk_{\K}(b_-,b_+)=\ep_{\K}(lk(b_-,b_+))=\ep_{\K}(\Lambda_{\Z}(z_{-,i_0},z_{+,j_0}))\neq 0$ and such that \[ \min(f|_{\overline{Im(b_-)}})-\max(f|_{\overline{Im(b_+)}})=-\ell_{-f}(z_{-,i_0})-\ell_f(z_{+,j_0})\geq -\ell_{-f}(z_-)-\ell_{f}(z_+)\] where the last inequality uses (\ref{elldec}).  
 
 Since $z_-\in Im(d_{-f,n-k}^{\K})$ and $z_+\in Im(d_{f,k+1}^{\K})$ were arbitrary elements subject to the condition that $\Lambda_{\K}(z_-,z_+)\neq 0$, it immediately follows that $\beta^{geom}_{k}(f;\K)\geq \beta^{alg}_{k}(f;\K)$.
\end{proof}

We also obtain the following, which shows that Corollary \ref{linkcor1} is sharp and completes the proof of Theorem \ref{main2}.

\begin{cor}\label{cormain2}
Let $\K$ be a field, and let $h$ be a metric such that the gradient flow of $f$ with respect to $h$ is Morse--Smale and such that $h$ is locally trivial.  Then the rank of the operator $d_{f,k+1}^{\K}\co CM_{k+1}(f;\K)\to CM_{k}(f;\K)$ is the largest integer $m$ such that there exist $b_{1,-},\ldots,b_{s,-}\in\mathcal{B}_{n-k-1}(M)$, $b_{1,+},\ldots,b_{r,+}\in\mathcal{B}_{k}(M)$ with the properties that for each $i,j$ we have $(b_{i,+},b_{j,-})\in\mathcal{T}_k(M,f)$ and the metric $h$ is generic with respect to $f,b_{i,+},b_{j,-}$, and that the matrix $L$ with entries given by \[ L_{ij}=lk_{\K}(b_{j,-},b_{i,+})-(-1)^{(n-k)(k+1)}\Pi(M_{-f},I_{b_{i,+},b_{j,-}}M_f)\] has rank $m$.  Moreover, given an integer $m$, if any such $b_{i,+}$ and $b_{j,-}$ exist, they may be chosen in such a way that $\Pi(M_{-f},I_{b_{i,+},b_{j,-}}M_f)=0$.
\end{cor}

\begin{proof}  The statement that the rank of $d_{f,k+1}^{\K}$ is at least equal to $m$ is proven in Corollary \ref{linkcor1}.  For the reverse inequality, note first that if the inequality holds for some field $\K_0$, then it must also hold for all field extensions of $\K_0$ since the relevant ranks are not affected by the field extension.  Therefore for the rest of the proof we may assume that $\K$ is equal either to $\mathbb{Q}$ or to $\mathbb{Z}/p\mathbb{Z}$ for some prime $p$, since any field is an extension of one of these.

Denote $m=rank(d_{f,k+1}^{\K})$.  Of course since $d_{-f,n-k}^{\K}$ is adjoint to $d_{f,k+1}^{\K}$ by (\ref{adjpi}), we also have $m=rank(d_{-f,n-k}^{\K})$.  Now the linking pairing $\Lambda_{\K}\co Im(d_{-f,n-k}^{\K})\times Im(d_{f,k+1}^{\K})\to\K$ is nondegenerate by the same argument as in the proof of Proposition \ref{alg-betaprop}: if $z=\sum_q z_q q\in Im(d_{f,k+1}^{\K})\setminus\{0\}$, then choosing any $q_0$ such that $z_{q_0}\neq 0$, we have $\Lambda_{\K}(d_{-f,n-k}q_0,z)\neq 0$.  Consequently since $\K$ is a field there are $x_{1,-},\ldots,x_{m,-}\in Im(d_{-f,n-k}^{\K})$ and $x_{1,+},\ldots,x_{m,+}\in Im(d_{f,k+1}^{\K})$ such that \begin{equation}\label{kron} \Lambda_{\K}(x_{j,-},x_{i,+})=\left\{\begin{array}{ll} 1 & \mbox{if }i=j,\\ 0 & \mbox{if }i\neq j.\end{array}\right. \end{equation}
Suppose that $\K=\mathbb{Q}$, so we may consider $Im(d_{-f,n-k}^{\Z})$ and $Im(d_{f,k+1}^{\Z})$ as subgroups of $Im(d_{-f,n-k}^{\K})$ and $Im(d_{f,k+1}^{\K})$, respectively.  Then for some nonzero integer $N$ each of the elements $z_{i,\pm}=Nx_{i,\pm}$ will belong to $Im(d_{-f,n-k}^{\Z})$ or  $Im(d_{f,k+1}^{\Z})$.  Apply Proposition \ref{twoboundaries} (using primitives $a_{i,\pm}$ for $z_{i,\pm}$) to obtain pseudoboundaries $b_{i,\pm}^{0}\co B_{i,\pm}\to M$ so that \[ lk(b_{j,-}^{0},b_{i,+}^{0})=\Lambda_{\Z}(z_{j,-},z_{i,+})=\left\{\begin{array}{ll} N^2 & \mbox{if }i=j,\\ 0 & \mbox{if }i\neq j,\end{array}\right.\] and, for generic diffeomorphisms $\phi$ which are $C^1$-close to the identity, $I_{\phi\circ b_{i,+}^{0},\phi\circ b_{j,-}^{0}}=0$.  Of course, for such a diffeomorphism $\phi$ we will have $lk(\phi\circ b_{j,-}^{0},\phi\circ b_{i,+}^{0})=lk(b_{j,-}^{0},b_{i,+}^{0})$.  So where $b_{i,+}=\phi\circ b_{i,+}^{0}$ and $b_{j,-}=\phi\circ b_{j,-}^{0}$, the matrix $L$ described in the proposition is $N^2$ times the $m\times m$ identity, and in particular has rank $m$.  This completes the proof in the case that $\K=\mathbb{Q}$.

Finally suppose that $\K=\mathbb{Z}/p\mathbb{Z}$ where $p$ is prime.  We again have $x_{i,\pm}$ as in (\ref{kron}).  Choose $a_{i,-}\in CM_{n-k}(-f;\Z)$ and $a_{i,+}\in CM_{k+1}(f;\Z)$ such that $d_{-f,n-k}^{\Z}a_{i,-}$ and $d_{f,k-1}^{\Z}a_{i,+}$ reduce modulo $p$ to, respectively, $x_{i,-}$ and $x_{i,+}$.  Applying Proposition \ref{twoboundaries} to obtain pseudoboundaries $b_{i,\pm}^{0}$, and then letting $b_{i,\pm}=\phi\circ b_{i,\pm}^{0}$ for a suitably generic diffeomorphism $\phi$ which is $C^1$-close to the identity, we see that $I_{b_{i,+},b_{j,-}}=0$ (over $\Z$, and hence also over $\Z/p\Z$), and \[ lk(b_{j,-},b_{i,+})=lk(b_{j,-}^{0},b_{i,+}^{0})=\Lambda_{\Z}(d_{-f,n-k}^{\Z}a_{j,-},d_{f,k-1}^{\Z}a_{i,+}).\]  But $\Lambda_{\Z}(d_{-f,n-k}^{\Z}a_{j,-},d_{f,k+1}^{\Z}a_{i,+})$ reduces modulo $p$ to $\Lambda_{\Z/p\Z}(x_{j,-},x_{i,+})$, which is $1$ when $i=j$ and $0$ otherwise.  Thus the matrix $L$ described in the proposition is the $m\times m$ identity, which has rank $m$.\end{proof}

\section{Some technical proofs} \label{app}

This final section contains proofs of Lemmas \ref{diffu}, \ref{genmet}, and \ref{mark2}.

 \begin{proof}[Proof of Lemma \ref{diffu}]  
This is a fairly standard sort of application of the Sard--Smale theorem \cite{S}; as in \cite{MS} a minor complication is caused by the fact that $\Diff_{S}(Y)$ is not a Banach manifold, but this is easily circumvented by first considering the Banach manifold $\Diff^{k}_{S}(Y)$ of $C^k$ diffeomorphisms supported in $S$ for sufficiently large integers $k$.

Namely, for any positive integer $k>\dim M+\dim N-\dim Y$ consider the map \begin{align*} \Theta\co \Diff^{k}_{S}(Y)\times M\times  N&\to Y\times Y
\\ (\phi,m,n)&\mapsto \left(\phi(f(m)),g(n)\right). \end{align*}  This is a $C^k$ map of $C^k$-Banach manifolds and we will show presently that it is transverse to $\Delta\subset Y\times Y$.  

Let $(\phi,m,n)\in \Theta^{-1}(\Delta)$, so that $\phi(f(m))=g(n)$.  If $\phi(f(m))\notin int(S)$, then since $\{y|\phi(y)\neq y\}$ is an open subset contained in $S$ we must have $\phi\left(\phi(f(m))\right)=\phi(f(m))$, and therefore $f(m)=\phi(f(m))\in Y\setminus int(S)$. Now since $\phi$ is the identity on the open set $Y\setminus S$, the linearization $\phi_*$ acts as the identity at every point of $Y\setminus S$, and therefore (by continuity) also at every point of $\overline{Y\setminus S}=Y\setminus int(S)$.  In particular $\phi_*\co T_{f(m)}Y\to T_{f(m)}Y$ is the identity.  Consequently our assumption on $S$ implies that $\left((\phi\times f)\times g\right)_*\co T_m M\times T_n N\to T_{(f(m),f(m))}Y\times Y$ is already transverse to $\Delta$, and so $\Theta$ is certainly transverse to $\Delta$ at $(\phi,m,n)$.

There remains the case that $\phi(f(m))\in int(S)$.  But then a small perturbation of $\phi$ in $\Diff_{S}^{k}(Y)$ can be chosen which moves $\phi(f(m))$ in an arbitrary direction in $Y$; in other words, there are elements of form $(\xi,0,0)\in T_{\phi}\Diff_{S}^{k}(Y)\oplus T_{m}M\oplus T_n N$ such that $\Theta_*(\xi,0,0)$ is equal to an arbitrary element of $T_{\phi(f(m))}Y\times\{0\}\leq T_{(\phi(f(m)),\phi(f(m)))}(Y\times Y)$.   So since 
$T_{\phi(f(m))}Y\times\{0\}$ is complementary to $T_{(\phi(f(m)),\phi(f(m)))}\Delta$ in $T_{(\phi(f(m)),\phi(f(m)))}(Y\times Y)$ this proves that $\Theta$ is transverse to $\Delta$.

Consequently the implicit function theorem for Banach manifolds shows that $\Theta^{-1}(\Delta)$ is a $C^k$-Banach submanifold of $\Diff^{k}_{S}(Y)\times M\times  N$.  The projection $\pi\co \Theta^{-1}(\Delta)\to \Diff^{k}_{S}(Y)$ is Fredholm of index $\dim M+\dim N-\dim Y$ (which we arranged to be less than $k$), and so the Sard--Smale theorem applies to show that the set of regular values of $\pi$ is residual in $\Diff^{k}_{S}(Y)$.  Moreover a standard argument (see for instance the proof of \cite[Proposition 2.24]{S93}) shows that $\phi\in \Diff^{k}_{S}(Y)$ is a regular value for $\pi$ if and only if the restriction $\Theta|_{\{\phi\}\times M\times N}$ is transverse to $\Delta$.

This shows that, for all positive integers $k>\dim M+\dim N-\dim Y$, the set $\mathcal{S}^k$ of $\phi\in \Diff^{k}_{S}(Y)$ such that $(m,n)\mapsto (\phi(f(m)),g(n))$ is transverse to $\Delta$ is residual in $\Diff^{k}(S)$.  To complete the proof of the lemma it remains only to replace the integer $k$ by $\infty$, which we achieve by an argument adapted from \cite[p. 53]{MS}.  Write $M=\cup_{r=1}^{\infty}M_r$ and  $N=\cup_{s=1}^{\infty}N_s$ where each $M_r$ and $N_s$ is compact, and let \[ \mathcal{S}_{rs}=\{\phi\in \Diff_{S}(Y)|\left((\phi\circ f)\times g\right)\mbox{ is transverse to $\Delta$ at all points of $M_r\times N_s$}\}.\]  For each $r,s\in\mathbb{Z}_+$, $\mathcal{S}_{rs}$ is easily seen to be open in  the $C^1$ (and so also the $C^k$ for all $1\leq k\leq \infty$) topology on $\Diff_{S}(Y)$.     Likewise the set \[  
\mathcal{S}^{k}_{rs}=\{\phi\in \Diff^{k}_{S}(Y)|\left((\phi\circ f)\times g\right)\mbox{ is transverse to $\Delta$ at all points of $M_r\times N_s$}\}
\] is open in the $C^k$-topology on $\Diff^{k}_{S}(Y)$.

We now show that $\mathcal{S}_{rs}$ is dense in $\Diff_{S}(Y)$.  Let $\phi_{\infty}\in \Diff_{S}(Y)$ be arbitrary.  For any sufficiently large integer $k$, since $\mathcal{S}^k=\cap_{r,s}\mathcal{S}^{k}_{rs}$ is residual and therefore dense in $\Diff^{k}_{S}(Y)$ there is $\phi_k\in \mathcal{S}^k$ such that $d_{C^k}(\phi_k,\phi_{\infty})<3^{-k}$, where $d_{C^k}$ denotes $C^k$ distance (with respect to an arbitrary auxiliary Riemannian metric; since our diffeomorphisms are the identity off a fixed compact set, different choices of Riemannian metrics will result in uniformly equivalent distances $d_{C^k}$).  Now the smooth diffeomorphisms $\Diff_{S}(Y)$ are dense in $\Diff_{S}^{k}(Y)$, and $\mathcal{S}^{k}_{rs}$ is open, so there is $\phi'_k\in \mathcal{S}_{rs}=\mathcal{S}^{k}_{rs}\cap \Diff_{S}(Y)$ arbitrarily $C^k$-close to $\phi_k$; in particular this allows us to arrange that $d_{C^k}(\phi'_k,\phi_{\infty})< 2^{-k}$.  Letting $k$ vary, we have constructed a sequence $\{\phi'_k\}$ in $\mathcal{S}_{rs}$ such that $d_{C^k}(\phi'_k,\phi_{\infty})< 2^{-k}$,   which implies that the $\phi'_k$ converge to $\phi_{\infty}$ in the $C^{\infty}$ topology.  Thus $\mathcal{S}_{rs}$ is indeed dense in $\Diff_{S}(Y)$.  Since we have already shown that $\mathcal{S}_{rs}$ is open, this proves that the countable intersection $\mathcal{S}=\cap_{r,s}\mathcal{S}_{rs}$ is residual, as desired.
\end{proof}

\begin{proof}[Proof of Lemma \ref{genmet}] The argument is similar to that in \cite[Lemma 4.10]{S99}.   Let $p,q\in Crit(f)$.  Of course the fiber product is empty in case $p=q$, so from now on we assume $p\neq q$.  In \cite[Appendix A]{S93} Schwarz constructs a Banach manifold $\mathcal{P}_{p,q}^{1,2}(\mathbb{R},M)$ consisting of class $H^{1,2}$ maps $\gamma\co \R\to M$ suitably asymptotic to $p$ as $t\to -\infty$ and to $q$ as $t\to +\infty$.  Moreover there is a vector bundle $\mathcal{E}_{p,q}\to  \mathcal{P}_{p,q}^{1,2}(\mathbb{R},M)$ whose fiber over $\gamma\in \mathcal{P}_{p,q}^{1,2}(\mathbb{R},M)$ is $\gamma^* TM$, and the section \begin{align*} \Phi\co \mathcal{G}\times \mathcal{P}_{p,q}^{1,2}(\mathbb{R},M)&\to L^{2}_{\mathbb{R}}(\mathcal{E}_{p,q})  \\ (h,\gamma)&\mapsto \dot{\gamma}+(\nabla^h f)\circ \gamma \end{align*} is shown to be smooth as a map of Banach manifolds and to be transverse to the zero-section on \cite[p. 47]{S93}.

Write \[ \tilde{\mathcal{M}}^{univ}(p,q;f)=\{(h,\gamma)\in \Phi\co \mathcal{G}\times \mathcal{P}_{p,q}^{1,2}(\mathbb{R},M)| \Phi(h,\gamma)=0\}.\]  In other words, $\tilde{\mathcal{M}}^{univ}(p,q;f)$ consists of those pairs $(h,\gamma)$ where $\gamma$ is a negative $h$-gradient flowline of $f$ asymptotic in large negative time to $p$ and in large positive time to $q$.  Since $\Phi$ is transverse to the zero-section, $\tilde{\mathcal{M}}^{univ}(p,q;f)$ is a smooth Banach manifold.  Where $\pi_{p,q}\co \tilde{\mathcal{M}}^{univ}(p,q;f)\to \mathcal{G}$ is the projection, the statement that $h\in\mathcal{G}$ is a regular value of $\pi_{p,q}$ for each $p,q$ is equivalent to the statement that the negative gradient flow of $f$ with respect to the metric $h$ is Morse--Smale (see \cite[pp. 43--45]{S93} for more details).  

We  have a map $\tilde{E}_k\co \tilde{\mathcal{M}}^{univ}(p,q;f)\times \mathbb{R}_{+}^{k-1}\to M^k$ defined by \[ \tilde{E}_k(\gamma,h,t_1,\ldots,t_{k-1})=\left(\gamma(0),\gamma(t_1),\ldots,\gamma\left(\sum_{i=1}^{k-1}t_i\right)\right),\] and we now claim that $\tilde{E}_k$ is a submersion.  Indeed, more specifically, we claim that for any $(\gamma,h,t_1,\ldots,t_{k-1})\in \tilde{\mathcal{M}}^{univ}(p,q;f)\times \mathbb{R}_{+}^{k-1}\to M^k$, writing $s_j=\sum_{i=1}^{j}t_i$ for $j=0,\ldots,k-1$, the linearization of $\tilde{E}_k$ at $(\gamma,h,t_1,\ldots,t_{k-1})$ restricts to $T_{(\gamma,h)}\tilde{\mathcal{M}}^{univ}(p,q;f)\times\{\vec{0}\}$ as a surjection to $\prod_{j=0}^{k-1}T_{\gamma(s_{j})}M$.  As in \cite[(4.15)]{S99}, with respect to a suitable frame along $\gamma$ the linearization of the operator $\Phi\co  \mathcal{G}\times \mathcal{P}_{p,q}^{1,2}(\mathbb{R},M)\to L^{2}_{\mathbb{R}}(\mathcal{E}_{p,q})$ takes the form $\Phi_*(\xi,A)=\dot{\xi}+S(t)\xi + A\cdot\nabla^h f$.  Here $\xi$ varies through $H^{1,2}(\gamma^*TM)\cong H^{1,2}(\mathbb{R},\R^n)$ and $A$ varies through a Banach space consisting of smooth sections (and containing  in particular all compactly supported smooth sections) of the bundle of symmetric endomorphisms of $\gamma^*TM$.  Moreover $t\mapsto S(t)$ is a certain smooth path of symmetric operators on $\mathbb{R}^n$.  To prove our claim we need to check that if $v_j\in T_{\gamma(s_j)}M$ are arbitrary vectors then there is an element $(\xi,A)\in \ker\Phi_*$ such that $\xi(s_j)=v_j$ for each $j=0,\ldots,k-1$.  Now $\gamma$ is a nonconstant (since $p\neq q$) flowline of $-\nabla^h f$, and so the points $\gamma(s_j)$ are all distinct, and $\nabla^h f$ is nonvanishing at each $\gamma(s_j)$.  But then we can simply choose $\xi\in H^{1,2}(\gamma^*TM)$ to be an arbitrary smooth section which is compactly supported in a union of small disjoint neighborhoods of the various $s_j$, and such that $\xi(s_j)=v_j$.  Having chosen this $\xi$, since $\nabla^h f$ is nonvanishing on the support of $\xi$ it is straightforward to find a section $A$ of the bundle of symmetric endomorphisms of $\gamma^*TM$, having the same compact support as $\xi$, with the property that $A\cdot\nabla^h f=-\dot{\xi}-S(t)\xi$ everywhere.  This pair $(\xi,A)$ will be as desired, confirming that $\tilde{E}_k$ is a submersion.

In view of this, given our maps $g_i\co V_i\to M$, the fiber product \[ \mathcal{V}^{univ}(p,q,f,g_0,\ldots,g_{k-1};h)=\left( V_0\times\cdots V_{k-1} \right){}_{g_0\times \cdots\times g_{k-1}}\times_{\tilde{E}_k}\left(\tilde{\mathcal{M}}^{univ}(p,q;f)\times \mathbb{R}_{+}^{k-1}\right) \] is cut out transversely, and so is a Banach manifold.  If the metric $h\in\mathcal{G}$ is a regular value for the projection $\pi_{\mathcal{V},p,q}\co \mathcal{V}^{univ}(p,q,f,g_0,\ldots,g_{k-1};h)\to \mathcal{G}$, then the original fiber product \linebreak $\mathcal{V}(p,q,f,g_0,\ldots,g_{k-1};h)$
appearing in the proposition will be cut out transversely.   Using the Sard--Smale theorem, the residual subset of the proposition is then given by the intersection of the sets of regular values of $\pi_{\mathcal{V},p,q}$ as $p$ and $q$ vary through $Crit(f)$ with the sets of regular values of the maps $\pi_{p,q}$ from the second paragraph of the proof.
\end{proof}

\begin{proof}[Proof of Lemma \ref{mark2}]
First we need to construct $\overline{\tilde{\mathcal{M}}(p,q;f)\times (0,\infty)}$ as a manifold with boundary by providing collars for the various parts of the boundary $C_1,\ldots,C_6$, in such a way that $E_1$ extends smoothly  to the boundary in accordance with the formulas given in the lemma.    To prepare for this, let us recall some features of the trajectory spaces $\tilde{\mathcal{M}}(x,y;f)$ and of the gluing map constructed in \cite[Section 2.5]{S93}.  

Assuming that $x,y\in Crit(f)$ with $x\neq y$ and $\tilde{\mathcal{M}}(x,y;f)\neq\varnothing$, so that in particular $f(y)<f(x)$, choose a regular value $a$ for $f$ with $f(y)<a<f(x)$.  Where for a trajectory $\gamma\in \tilde{\mathcal{M}}(x,y;f)$ we denote its equivalence class in $\mathcal{M}(x,y;f)$ by $[\gamma]$, the choice of $a$ induces an orientation-preserving diffeomorphism $\alpha_{a,x,y}\co \tilde{\mathcal{M}}(x,y;f)\to \mathcal{M}(x,y;f)\times\R$ defined by $\alpha_{a,x,y}(\gamma)=\left([\gamma],s_{a.\gamma}\right)$, where $s_{a,\gamma}$ is the real number characterized by the property that $f(\gamma(-s_{a,\gamma}))=a$.  For any $s\in\R$ and $\gamma\in \tilde{\mathcal{M}}(x,y;f)$ define $\sigma_s\gamma\in\tilde{\mathcal{M}}(x,y;f)$ by \[ (\sigma_s\gamma)(t)=\gamma(s+t).\]  Then if for an element $[\gamma]\in \mathcal{M}(x,y;f)$ we write $\gamma_0$ for the unique representative of $[\gamma]$ such that $f(\gamma_0(0))=a$, the inverse of $\alpha_{a,x,y}$ is given by $\alpha_{a,x,y}^{-1}([\gamma],s)=\sigma_s\gamma_0$.

Now let $r\in Crit(f)$ be any critical point distinct from $p$ and $q$ such that $\mathcal{M}(p,r;f)\times\mathcal{M}(r,q;f)$ is nonempty.  Choose regular values $a$ and $b$ of $f$ such that $f(q)<a<f(r)<b<f(p)$.  Then if $V$ is any open subset of $\mathcal{M}(p,r;f)\times\mathcal{M}(r,q;f)$ such that $\bar{V}$ is compact, \cite[Proposition 2.56]{S93} gives a number $\rho_V>0$ and a smooth embedding $\#_V\co (\rho_V,\infty)\times V\to \mathcal{M}(p,q;f)$ having the following features.  For an element $([\gamma],[\eta])\in V\subset \mathcal{M}(p,r;f)\times\mathcal{M}(r,q;f)$ choose the unique representatives $\gamma\in\tilde{\mathcal{M}}(p,r;f)$ and $\eta\in\tilde{\mathcal{M}}(r,q;f)$ such that $\gamma(0)=b$ and $\eta(0)=a$.  Then a suitable representative $\gamma\#_{\rho}\eta$ of $\#(\rho,[\gamma],[\eta])$ has the property that, on any fixed compact subset of $\R$, 
$\sigma_{-\rho}(\gamma\#_{\rho}\eta)\to \gamma$ uniformly exponentially fast as $\rho\to\infty$, and $\sigma_{\rho}(\gamma\#_{\rho}\eta)\to \eta$ uniformly exponentially fast as $\rho\to\infty$ (with the constants independent of the choice of $([\gamma],[\eta])$ from the precompact subset $V$).  

Furthermore, if $V_1$ and $V_2$ are two open subsets of $\mathcal{M}(p,r;f)\times\mathcal{M}(r,q;f)$ each with compact closure, the gluing maps $\#_{V_1}$ and $\#_{V_2}$ coincide on their common domain of definition (as follows from examination of the construction and was also noted in \cite[Proof of Lemma 4.4]{S99}).  Consequently if $\{\chi_{\beta}\}$ is a partition of unity subordinate to an open cover $\{V_{\beta}\}$ of $\mathcal{M}(p,r;f)\times\mathcal{M}(r,q;f)$ by open sets with compact closure, and if we define $\rho_0\co \mathcal{M}(p,r;f)\times\mathcal{M}(r,q;f)\to \R$ by  $\rho_0=\sum_{\beta}\chi_{\beta}\rho_{\beta}$, then the gluing maps $\#_{V_{\beta}}$ piece together to give a smooth map \begin{equation}\label{glue} \#\co \{(\rho,[\gamma],[\eta])\in \R\times \mathcal{M}(p,r;f)\times\mathcal{M}(r,q;f)|\rho>\rho_0([\gamma],[\eta])\}\to \mathcal{M}(p,q;f),\end{equation} which (in view of the convergence properties of the $\gamma\#_{\rho}\eta$) can be arranged to be an embedding after possibly replacing $\rho_0$ by a larger smooth function.  

With respect to our orientation conventions from Section \ref{or:morse}, the gluing map $\#$  can be seen to affect the orientation by multiplication by $(-1)^{|p|_f-|r|_f-1}$ (see also \cite[A.1.14]{BC}).

We now use these facts to produce collars for the parts $C_1,\ldots,C_6$ of $\partial \overline{\tilde{\mathcal{M}}(p,q;f)\times (0,\infty)}$.
More specifically, for each $i$ we will construct, for a suitable smooth function $\ep_i\co C_i\to (0,\infty)$, a smooth embedding \[ \psi_i\co \{(t,x)\in\R\times C_i|0<t<\ep_i(x)\}\times C_i\to \tilde{\mathcal{M}}(p,q;f)\times (0,\infty),\] such that $E_1\circ\psi_i$ extends smoothly to $\{0\}\times C_i$ in a way that agrees with the formulas for $\bar{E}_1|_{C_i}$ in the statement of the lemma. Let $\hat{C}_i=\{(t,x)\in\R\times C_i|0\leq t<\ep_i(x)\}$, so that $\hat{C}_i$ has the structure of a manifold with boundary $\{0\}\times C_i$. We can then set \[ \overline{\tilde{\mathcal{M}}(p,q;f)\times (0,\infty)}=\frac{\hat{C}_1\sqcup \hat{C}_2\sqcup\cdots\sqcup \hat{C}_6\sqcup \left(\tilde{\mathcal{M}}(p,q;f)\times(0,\infty)\right)}{z\sim \psi_i(z)\mbox{ for }z\in \hat{C}_i\setminus \partial\hat{C}_i,\,i=1,\ldots,6}.\]  
This will be a Hausdorff topological space, since our formulas imply that the continuous extension $\bar{E}_1\co\overline{\tilde{\mathcal{M}}(p,q;f)\times (0,\infty)}\to M\times M$ of $E_1$ is injective, and any space that admits an injective continuous map to a Hausdorff space is Hausdorff.  The $\psi_i$ will be diffeomorphisms to their images by dimensional considerations, so 
$\overline{\tilde{\mathcal{M}}(p,q;f)\times (0,\infty)}$ will inherit a smooth manifold-with-boundary atlas from $\tilde{\mathcal{M}}(p,q;f)\times (0,\infty)$ and from the $\hat{C}_i$, making $\bar{E}_1$ a smooth function.  Since if $C_i$ is oriented, one has $\partial\hat{C}_i=-C_i$ as oriented manifolds (as we use the outer-normal-first convention), the boundary orientation of $C_i$ induced by 
the orientation of $\tilde{\mathcal{M}}(p,q;f)\times (0,\infty)$ will be the orientation of $C_i$ that makes $\psi_i$ into an orientation-\emph{reversing} embedding.

So we now construct the $\psi_i$, starting with $\psi_1$. Let $r\in Crit(f)$ with $|r|_f=|p|_f-1$ and $\mathcal{M}(p,r;f)\times\tilde{\mathcal{M}}(r,q;f)\neq\varnothing$, and let $a$ and $b$ be  regular values of $f$ with $f(q)<a<f(r)<b<f(p)$, inducing an orientation-preserving diffeomorphism $\alpha_{a,r,q}\co\tilde{\mathcal{M}}(r,q;f)\to \mathcal{M}(r,q;f)\times\R$ and a gluing map $\#$ as in (\ref{glue}).  Recall that the map $\#$ lifts to a map into $\tilde{\mathcal{M}}(p,q;f)$, given by $(\rho,[\gamma],[\eta])\mapsto \gamma\#_{\rho}\eta$ where the representatives $\gamma$ and $\eta$ are chosen so that $f(\gamma(0))=b$ and $f(\eta(0))=a$.  Using $\alpha_{a,r,q}$ to identify $\tilde{\mathcal{M}}(r,q;f)$ with $\mathcal{M}(r,q;f)\times\R$, the part of our collar $\psi_1$ corresponding to the critical point $r$ is the map \[ 
 \left\{\left.(\delta,[\gamma],[\eta],s,T)\in (0,\infty)\times\mathcal{M}(p,r;f)\times\mathcal{M}(r,q;f)\times \R\times(0,\infty)\right| 0<\delta<\frac{1}{\rho_0([\gamma],[\eta])}\right\}\to  \tilde{\mathcal{M}}(p,q;f)\times (0,\infty) \] defined by \[ \psi_1\left(\delta,[\gamma],[\eta],s,T\right)=\left(\sigma_{s+\delta^{-1}}(\gamma\#_{\delta^{-1}}\eta),T\right).\]  The fact that the map $\#$ of (\ref{glue}) is a smooth embedding readily implies that $\psi_1$ is a smooth embedding as well (at least after possibly lowering the upper limit on $\delta$ to prevent overlap between the images of maps from overlapping for different choices of the finitely many $r$).  Our identification of $\mathcal{M}(r,q;f)\times\R$ with $\tilde{\mathcal{M}}(r,q;f)$ has $([\eta],s)$ corresponding to $\sigma_s\eta$, so the fact that $\sigma_{s+\delta^{-1}}(\gamma\#_{\delta^{-1}}\eta)$ converges exponentially quickly on any compact subset of $\R$ to $\sigma_s\eta$ as $\delta^{-1}\to\infty$ readily implies that the function $E_1\circ\psi_1$ extends smoothly to $\{0\}\times C_1\subset \hat{C}_1$ by the formula stated in the lemma.  (The exponential nature of the convergence yields, on compact subsets of $C_1$, uniform estimates $dist(E_1\circ\psi_1(\delta,z),\bar{E}_1|_{C_1}(z))\leq Be^{-\beta/\delta}$, which ensures smoothness up to the boundary, with normal derivatives of all orders vanishing.)  As for the orientation, using the orientation preserving identification $\alpha_{a,p,q}\co \tilde{\mathcal{M}}(p,q;f)\cong \mathcal{M}(p,q;f)\times \R$ and the fact that (since $|p|_f=|r|_f+1$ in this case) the gluing map $(\delta,[\gamma],[\eta])\mapsto ([\gamma\#_{\delta^{-1}}\eta])$ is orientation-reversing, it is clear that $\psi_1$ is orientation-reversing.  Consequently $C_1$'s orientation as part of the boundary of $\overline{\tilde{\mathcal{M}}(p,q;f)\times(0,\infty)}$ coincides with its usual orientation.
 
The construction of $\psi_2$ is very similar to that of $\psi_1$: Given $r\in Crit(f)$ with $|r|_f=|q|_f+1$ and $\mathcal{M}(r,q;f)\neq\varnothing$, choose a regular value $b$ with $f(r)<b<f(p)$, inducing an identification $\alpha_{b,p,r}\co \tilde{\mathcal{M}}(p,r;f)\cong \mathcal{M}(p,r;f)\times\R$.  With respect to this identification, for $0<\delta<\frac{1}{\rho_0([\gamma],[\eta])}$ define \begin{equation}\label{psi2} \psi_2(\delta,[\gamma],s,[\eta],T)=(\sigma_{s-\delta^{-1}}(\gamma\#_{\delta^{-1}}\eta),T) \end{equation} where the representatives $\gamma$ and $\eta$ are chosen just as in the definition of $\psi_1$.  The exponential convergence of $\sigma_{-\delta^{-1}}(\gamma\#_{\delta^{-1}}\eta)$ to $\gamma$ on compact subsets can be seen to imply that this $\psi_2$ has the properties that we require.  The boundary orientation of $C_2$ may be computed by switching the positions of the parameters $s$ and $[\eta]$ in the domain and using the facts that the gluing map $\#$ of (\ref{glue}) affects the orientation by a sign $(-1)^{|p|_f-|r|_f-1}=(-1)^{|p|_f-|q|_f}$, and that $\mathcal{M}(p,r;f)$ is zero-dimensional.

As for $\psi_3$, now let $r$ be any critical point distinct from $p$ and $q$ such that $\tilde{\mathcal{M}}(p,r;f)\times\tilde{\mathcal{M}}(r,q;f)\neq\varnothing$, and as usual choose regular values $a$ and $b$ with $f(q)<a<f(r)<b<f(p)$.  This induces an orientation-preserving diffeomorphism \[ \alpha_{b,p,r}\times  \alpha_{a,r,q}\co \tilde{\mathcal{M}}(p,r;f)\times \tilde{\mathcal{M}}(r,q;f)\cong \mathcal{M}(p,r;f)\times\R\times \mathcal{M}(r,q;f)\times\R.\]
With respect to this identification, define, for $([\gamma],[\eta])\in \tilde{\mathcal{M}}(p,r;f)\times \tilde{\mathcal{M}}(r,q;f)$ and $0<\delta<\rho_0([\gamma],[\eta])$, \[ \psi_3(\delta,[\gamma],s,[\eta],u)=\left(\sigma_{s-\delta^{-1}}(\gamma\#_{\delta^{-1}}\eta),2\delta^{-1}-s+u\right),\] where as usual the representatives $\gamma$ and $\eta$ are chosen so that $f(\gamma(0))=b$ and $f(\eta(0))=a$.  The  convergence of $\sigma_{-\delta^{-1}}(\gamma\#_{\delta^{-1}}\eta)$ to $\gamma$ on compact subsets gives that $\sigma_{s-\delta^{-1}}(\gamma\#_{\delta^{-1}}\eta)(0)$ converges to $\sigma_s\gamma(0)$ as $\delta\to 0$ for all $s$, and the  convergence of $\sigma_{\delta^{-1}}(\gamma\#_{\delta^{-1}}\eta)$ to $\eta$ on compact subsets gives that $\sigma_{s-\delta^{-1}}(\gamma\#_{\delta^{-1}}\eta)(2\delta^{-1}-s+u)$ converges to $\sigma_u\eta(0)$ as $\delta\to 0$ for all $u$.  This implies that $E_1\circ\psi_3$ extends continuously to $C_3\times\{0\}$ in the manner asserted in the statement of the lemma.  The fact that $\psi_3$ is an embedding (at least after appropriately shrinking the domain) and that the extension of $E_1$ is smooth follows just as in the case of $\psi_1$.  To compute the boundary orientation of $C_3$, note that moving the parameter $s$ past $[\eta]$ in the domain leads to a sign $(-1)^{|r|_f-|q|_f-1}$, which when combined with the usual sign coming from the gluing map $\#$ leads to the boundary orientation of $C_3$ being $(-1)^{|p|_f-|q|_f}$ times its usual orientation, as stated in the lemma.

For $i=4,5,6$ we have $C_i=\tilde{\mathcal{M}}(p,q;f)$, and we can use the following rather simpler collars $\psi_i\co (0,1)\times \tilde{\mathcal{M}}(p,q;f)\to  \tilde{\mathcal{M}}(p,q;f)\times(0,\infty)$:
\begin{align*}
\psi_4(\delta,\gamma)&=(\gamma,\delta)  \\ 
\psi_5(\delta,\gamma)&=(\sigma_{-\delta^{-1}}\gamma,\delta^{-1}) \\
\psi_6(\delta,\gamma)&=(\gamma,\delta^{-1})
\end{align*}

That $\psi_4,\psi_5,\psi_6$ satisfy the required properties and induce the stated orientations is in each case straightforward; perhaps the only point to mention is that the fact that the extension of $E_1$ is smooth up to the boundary along $C_5$ and $C_6$ follows from the fact that any $\gamma\in \tilde{\mathcal{M}}(p,q;f)$ has $\gamma(t)\to p$ exponentially fast as $t\to -\infty$, and $\gamma(t)\to q$ exponentially fast as $t\to\infty$ (see \emph{e.g.} \cite[Lemma 2.10]{S93}).

This completes the construction of $\overline{\tilde{\mathcal{M}}(p,q;f)\times(0,\infty)}$; it remains to show that the $\Omega$-limit set of $\bar{E}_1$ is as described.  In other words we need to show that if $\{(\gamma_n,t_n)\}_{n=1}^{\infty}$ is any sequence in $\tilde{\mathcal{M}}(p,q;f)\times(0,\infty)$ then after passing to a subsequence $\{(\gamma_n,t_n)\}_{n=1}^{\infty}$ will either converge in 
$\overline{\tilde{\mathcal{M}}(p,q;f)\times(0,\infty)}$ or else will have the property that $E_1(\gamma_n,t_n)=(\gamma_n(0),\gamma_n(t_n))$ converges to a point in one of the sets described in (i)-(iv) of the statement of the Lemma.  (Since $\tilde{\mathcal{M}}(p,q;f)\times(0,\infty)$ is dense in $\overline{\tilde{\mathcal{M}}(p,q;f)\times(0,\infty)}$ we need only consider sequences in $\tilde{\mathcal{M}}(p,q;f)\times(0,\infty)$).

So let $\{(\gamma_n,t_n)\}_{n=1}^{\infty}$ be a sequence in $\tilde{\mathcal{M}}(p,q;f)\times(0,\infty)$.  By the basic compactness result \cite[Proposition 2.35]{S93}, we may pass to a subsequence such that, for some $\nu\in\{1,\ldots,|p|_f-|q|_f\}$, some critical points $p=p_0,p_1,\ldots,p_{\nu}=q$ of $f$, some trajectories $\gamma^j\in \tilde{\mathcal{M}}(p_{j-1},p_j;f)$, and some sequences $\{\tau_{n,j}\}_{n=1}^{\infty}$ in $\R$ for $j=1,\ldots,\nu$, we have for each $j$, \begin{equation}\label{recvg} \sigma_{\tau_{n,j}}\gamma_n\to \gamma^j \quad \mbox{uniformly with all derivatives on each compact subset of $\mathbb{R}$}.\end{equation}
(In this case $\{\gamma_n\}_{n=1}^{\infty}$ is said to converge weakly to the broken trajectory $(\gamma^1,\gamma^2)$.)

These conditions continue to hold if we remove all constant trajectories $\gamma^j$ from consideration, so without loss of generality we assume that each $\gamma^j$ is nonconstant, so that $p_j\neq p_{j-1}$ for all $j$.

Since the values $f(\gamma_n(0))$ and $f(\gamma_n(t_n))$ are confined to the compact interval $[f(q),f(p)]$ and have $f(\gamma_n(0))>f(\gamma_n(t_n))$, and since $f$ is exhausting, by passing to a further subsequence we may assume that $\gamma_n(0)\to x_0$ and $\gamma_n(t_n)\to x_T$ for some $x_0,x_T\in M$ with $f(q)\leq f(x_T)\leq f(x_0)\leq f(p)$.  

We may then choose $j\in \{1,\ldots,\nu\}$ such that $f(p_j)\leq f(x_0)\leq f(p_{j-1})$.  Passing to a further subsequence, we may assume that $\{\tau_{n,j}\}_{n=1}^{\infty}$  either converges to a limit $-\tau_0$ or diverges to $+\infty$ or diverges to $-\infty$.  In the first case we obtain by (\ref{recvg}) that \[ \gamma_n(0)=\sigma_{\tau_{n,j}}\gamma_n(-\tau_{n,j})\to \gamma^j(\tau_0)\mbox{ as }n\to\infty,\]  and thus $x_0\in e_{p_{j-1},p_j}\left(\tilde{\mathcal{M}}(p_{j-1},p_j;f)\right)$.  Suppose that instead $\tau_{n,j}\to +\infty$.  Then for any given $t\in\R$, for large enough $n$ we will have \[ f\left(\gamma_n(0)\right)=f\left((\sigma_{\tau_{n,j}}\gamma_n)(-\tau_{n,j})\right)\geq f\left((\sigma_{\tau_{n,j}}\gamma_n)(t)\right).\] Thus $f(x_0)\geq f(\gamma^j(t))$ for all $t\in\R$, and since $j$ was chosen so that $f(x_0)\leq f(p_{j-1})$ this forces $f(x_0)=f(p_{j-1})$.  
We will now show that, continuing to assume that $\tau_{n,j}\to +\infty$, we in fact have $x_0=p_{j-1}$.  For any small open ball $B$ around $p_{j-1}$ and any $\ep>0 $ there is $T>0$ such that $\gamma^j(-T)\in B$ and  $f(\gamma^j(-T))>f(p_{j-1})-\ep$, and therefore for large enough $n$ we will have  $\gamma_n(\tau_{n,j}-T)\in B$ and $f\left(\gamma_n(\tau_{n,j}-T)\right)>f(p_{j-1})-\ep$.  So since $\gamma_n$ is a negative gradient trajectory of $f$ and $f(\gamma_n(0))<f(p_{j-1})+\ep$ for large enough $n$ we have  \[ \int_{0}^{\tau_{n,j}-T}\|\dot{\gamma_n}(t)\|^2dt=f(\gamma_n(0))-f(\gamma_n(\tau_{n,j}-T))<2\ep \] for all sufficiently large $n$.  Now  if $x_0=\lim \gamma_n(0)$ were not equal to $p_{j-1}$,  we
 could find\footnote{Specifically, choose disjoint balls around all of the critical points of $f$ with critical value at most $f(p_{j-1})+1$ and also a ball around $x_0$, let $\delta$ be the infimum of $\|\nabla f\|$  in $\{x|f(x)\leq f(p_{j-1})+1\}$ off of these balls, and let $D$ be the minimal distance between any two of the balls.} $\ep$-independent constants $\delta,D>0$ and disjoint balls $B$ around $p_{j-1}$ and $B'$ around $x_0$ such that for any path $\eta\co [0,R]\to \{x|f(x)\leq f(p_{j-1})+1\}$ beginning in $B'$ and ending in $B$ there would be a segment $\eta|_{[r_1,r_2]}$ having length at least $D$ and such that $\|\nabla f(\eta(t))\|\geq \delta$ for all $t\in [r_1,r_2]$.
 In particular for large $n$ this would apply to $\eta=\gamma_n|_{[0,\tau_{n,j}-T]}$, where $T$ has been chosen based on an arbitrary $\ep>0$ as above. We would then obtain \[ 2\ep>\int_{r_1}^{r_2}  \|\nabla f(\gamma_n(t))\|^2dt\geq (r_2-r_1)\delta^2,\] so $r_2-r_1<\frac{2\ep}{\delta^2}$.  But  if $C$ is the maximum of $\|\nabla f\|$ on $\{f\leq f(p_{j-1})+1\}$, $\gamma_n|_{{[r_1,r_2]}}$ would then have length at most $\frac{2C\ep}{\delta^2}$, which if we choose $\ep$ sufficiently small is a contradiction with the fact that 
$\eta|_{[r_1,r_2]}$ needs to have length at least $D$.  This contradiction shows that we must indeed have $x_0=p_{j-1}$.

The same argument shows that if $\tau_{n,j}\to -\infty$ then $x_0=p_j$. Moreover, applying the same argument to the sequence $\{\gamma_n(t_n)\}_{n=1}^{\infty}$ in place of $\{\gamma_n(0)\}_{n=1}^{\infty}$ shows that, if $k$ is chosen so that $f(p_k)\leq f(x_T)\leq f(p_{k-1})$, then $x_T=\lim_{n\to\infty}\gamma_n(t_n)$ is given by \[ x_T=\left\{\begin{array}{ll} \gamma^k(\tau_T) & \mbox{ if }\lim_{n\to\infty}\tau_{n,k}-t_n=-\tau_T \\ p_{k-1}&\mbox{ if } \lim_{n\to\infty}\tau_{n,k}-t_n=+\infty \\ p_{k}&\mbox{ if } \lim_{n\to\infty}\tau_{n,k}-t_n=-\infty \end{array}\right.\] (and of course we may and do pass to a subsequence such that one of the above three alternatives holds).

So we can now check case-by-case based on the number $\nu$ of trajectories that appear in the limit and on the behavior of the sequences $\{\tau_{n,j}\}_{n=1}^{\infty}$ that, having passed to this subsequence, either $\{(\gamma_n,t_n)\}_{n=1}^{\infty}$ converges in $\overline{\tilde{\mathcal{M}}(p,q;f)\times(0,\infty)}$ or else $\{E_1(\gamma_n,t_n)\}_{n=1}^{\infty}$ converges to a point in one of the sets (i)-(iv) in the statement of the lemma.

First suppose that $\nu=1$.  If neither $\{\tau_{n,1}\}_{n=1}^{\infty}$ nor $\{t_n-\tau_{n,1}\}_{n=1}^{\infty}$ converges in $\mathbb{R}$ then it follows from the last few paragraphs that both $x_0=\lim_{n\to\infty}\gamma_n(0)$ and $x_T=\lim_{n\to\infty}\gamma_n(t_n)$ converge to $p$ or $q$ and so $E_1(\gamma_n,t_n)$ converges to a point of (iv) (allowing the possibilities $a=b=p$ or $a=b=q$).  If $\{\tau_{n,1}\}_{n=1}^{\infty}$ converges, say to $-\tau_0$, then since $\sigma_{\tau_{n,1}}\gamma_n\to \gamma^1$ it follows that $\gamma_n\to \sigma_{\tau_0}\gamma^1\in \tilde{\mathcal{M}}(p,q;f)$.  Thus  if $\{t_n\}_{n=1}^{\infty}$ converges to a positive real number $T$ then $\{(\gamma_n,t_n)\}_{n=1}^{\infty}$ converges to a point (namely $(\sigma_{\tau_0}\gamma^1,T)$) of $\tilde{\mathcal{M}}(p,q;f)\times(0,\infty)$; if $t_n\to 0$ then (as follows directly from the formula for $\psi_4$) $\{(\gamma_n,t_n)\}_{n=1}^{\infty}$ lies in  the image of $\psi_4$ for large $n$ and finally converges to a point of $C_4$; and if $\{t_n\}_{n=1}^{\infty}$ diverges to $\infty$ then $\{(\gamma_n,t_n)\}_{n=1}^{\infty}$ similarly converges to a point of $C_6$.  The only remaining possibility when $\nu=1$ is that $\{\tau_{n,1}\}_{n=1}^{\infty}$ diverges but
$\{t_n-\tau_{n,1}\}_{n=1}^{\infty}$ converges, say to $\tau_T$.  So in this case $t_n\to\infty$ and 
 $\sigma_{t_n}\gamma\to \sigma_{\tau_T}\gamma^1$, in view of which $\{(\gamma_n,t_n)\}_{n=1}^{\infty}$ lies in the image of $\psi_5$ for large $n$ and converges to the element $\sigma_{\tau_T}\gamma^1$ of $C_5$.
 
Now suppose $\nu=2$; thus the trajectories $\gamma_n$ converge weakly to the broken trajectory $(\gamma^1,\gamma^2)$, where for some $r\in Crit(f)$ distinct from $p$ and $q$, $\gamma^1\in \tilde{\mathcal{M}}(p,r;f)$ and $\gamma^2\in \tilde{\mathcal{M}}(r,q;f)$.  Now the analysis above shows that, where $x_0=\lim_{n\to\infty}\gamma_n(0)$ and $x_T=\lim_{n=0}^{\infty}\gamma_n(t_n)$, we have $x_0,x_T\in \{p,q,r\}\cup\gamma^1(\R)\cup\gamma^2(\R)$.  If either $x_0$ or $x_T$ belongs to $\{p,q,r\}$ then $E_1(\gamma_n,t_n)$ converges to a point in a set in (iv) of the statement of the lemma.  Also, if $x_0=x_T$, then $E_1(\gamma_n,t_n)$ converges to a point in a set in (iii) of the statement of the lemma.  Thus we may assume that $x_0$ and $x_T$ are distinct points, each lying on $\gamma^1(\R)\cup\gamma^2(\R)$.  Also, in the case that  both $x_0,x_T\in \gamma^1(\R)$,  if $|r|_f>|q|_f+1$ then $E_1(\gamma_n,t_n)$ converges to a point in a set in (i) of the statement of the lemma; the same also holds if $x_0,x_T\in \gamma^2(\R)$ and $|r|_f<|p|_f-1$.  So if $x_0,x_T\in \gamma^1(\R)$ we may assume that $|r|_f=|q|_f+1$, and if $x_0,x_T\in \gamma^2(\R)$ we may assume that $|r|_f=|p|_f-1$.

Suppose that $x_0,x_T\in \gamma^1(\R)$.  As noted earlier, this implies that the sequences $\{\tau_{n,1}\}_{n=1}^{\infty}$ and $\{\tau_{n,1}-t_n\}_{n=1}^{\infty}$ both converge, say to $-\tau_0$ and $-\tau_T$, respectively, and in this case we have $x_0=\gamma^1(\tau_0)$ and $x_T=\lim_{n\to\infty}\gamma^1(\tau_T)$, so $\tau_T>\tau_0$ since $x_T\neq x_0$.  Recall that in defining the collar $\psi_2$ for $C_2$ we made a choice of regular values $a$ and $b$ such that $f(q)<a<f(r)<b<f(q)$.  By the last sentence of \cite[Proposition 2.57]{S93}, for large $n$ the equivalence class $[\gamma_n]$ of $\gamma_n$ will lie in the image of the gluing map (\ref{glue}); thus there will be $\gamma^{1}_{n}\in \tilde{\mathcal{M}}(p,r;f)$ and  $\gamma^{2}_{n}\in \tilde{\mathcal{M}}(r,q;f)$ such that $f(\gamma^{1}_{n})=b$, $f(\gamma^{2}_{n})=a$, and sequences of real numbers $\rho_n,u_n$ such that $\rho_n\to\infty$ and \[ \gamma_n=\sigma_{u_n}\left(\gamma^{1}_{n}\#_{\rho_n}\gamma^{2}_{n}\right),\] with $\rho_n$ remaining in an interval $[\bar{\rho},\infty)$ and $([\gamma^{1}_{n}],[\gamma^{2}_{n}])$ remaining in a fixed compact subset of $\mathcal{M}(p,r;f)\times \mathcal{M}(r,q;f)$.  

Consequently, keeping in mind that the formula for $\psi_2$ used the identification of $\tilde{\mathcal{M}}(p,r;f)$ with $\mathcal{M}(p,r;f)\times\R$ determined by the regular value $b$, we obtain that for large $n$, \[ (\gamma_n,t_n)=\psi_2\left(\rho_{n}^{-1},\sigma_{\rho_n+u_n}\gamma_{n}^{1},[\gamma^{2}_{n}],t_n\right),\] and in particular our sequence eventually enters and never leaves the collar around $C_2$.  By considering the properties of the function $\bar{E}_1$ on the image of $C_2$, the weak convergence properties of the $\gamma_n$ then imply that, as $n\to\infty$, \[ \psi_{2}^{-1}(\gamma_n,t_n)\to (0,\sigma_{\tau_0}\gamma^1,[\gamma^2],\tau_T-\tau_0)\in \{0\}\times \tilde{\mathcal{M}}(p,r;f)\times\mathcal{M}(r,q;f)\times(0,\infty),\] proving that the sequence $\{(\gamma_n,t_n)\}_{n=1}^{\infty}$ converges to a point of $C_2\subset \overline{\tilde{\mathcal{M}}(p,q;f)\times (0,\infty)}$ when $\nu=2$ and $x_0,x_T\in \gamma^1(\mathbb{R})$.

In the case that $\nu=2$ and $x_0,x_T\in \gamma^2(\R)$, an identical analysis based on the sequences $\{\tau_{n,2}\}_{n=1}^{\infty}$ and 
$\{\tau_{n,2}-t_n\}_{n=1}^{\infty}$ shows that $\{(\gamma_n,t_n)\}_{n=1}^{\infty}$ converges to a point of $C_1\subset \overline{\tilde{\mathcal{M}}(p,q;f)\times (0,\infty)}$.

The remaining case when $\nu=2$ is where $x_0\in \gamma^1(\R)$ and $x_T\in \gamma^2(\R)$ (since $t_n>0$ the opposite is impossible).  Then the sequence $\{\tau_{n,1}\}_{n=1}$ converges (say to $-\tau_0$) since $x_0\in \gamma^1(\R)$, and the sequence $\{\tau_{n,2}-t_n\}_{n=1}^{\infty}$ converges (say to $-\tau_T$) since $x_T\in\gamma^2(\R)$.  For large enough $n$, the weak convergence of $\{\gamma_n\}_{n=1}^{\infty}$ and \cite[Proposition 2.57]{S93} give large real numbers $\rho_n$ and trajectories $\gamma^{1}_{n}\in \tilde{\mathcal{M}}(p,r;f)$ and $\gamma^{2}_{n}\in \tilde{\mathcal{M}}(r,q;f)$ with $f(\gamma^{1}_{n}(0))=b$ and $f(\gamma^{2}_{n}(0))=a$ such that \[ \gamma_n=\sigma_{u_n}\left(\gamma_{n}^{1}\#_{\rho_n}\gamma_{n}^{2}\right)\] for some real numbers $u_n$, with $\rho_n$ remaining in an interval $[\bar{\rho},\infty)$ and $([\gamma^{1}_{n}],[\gamma^{2}_{n}])$ remaining in a fixed compact subset of $\mathcal{M}(p,r;f)\times \mathcal{M}(r,q;f)$.    From this one obtains that, for large $n$, \[ (\gamma_n,t_n)=\psi_3\left(\rho_{n}^{-1}, \sigma_{\rho_n+u_n}\gamma_{n}^{1},\sigma_{u_n+t_n-\rho_n}\gamma_{n}^{2}\right).\]  Thus our sequence eventually enters and never leaves the collar around $C_3$, and the weak convergence properties of the sequence imply that \[ \psi_{3}^{-1}(\gamma_n,t_n)\to (0,\sigma_{\tau_0}\gamma^1,\sigma_{\tau_T}\gamma^2)\in\{0\}\times \tilde{\mathcal{M}}(p,r;f)\times\tilde{\mathcal{M}}(r,q;f).\]  This completes the proof in case $\nu=2$. 

Finally suppose that $\nu>2$. Since all of the trajectories $\gamma^j$ are nonconstant and so (by the Morse-Smale condition) $1\leq |p_j|_{f}-|p_{j+1}|_f\leq |p|_f-|q|_f-2$ for all $j$, and since $x_0,x_T\in Crit(f)\cup\gamma^1(\R)\cup\cdots\cup \gamma^{\nu}(\R)$, it is straightforward to see that in any case $(x_0,x_T)=\lim_{n\to\infty}E_n(\gamma_n,t_n)$ belongs to one of the sets (i)-(iv).
\end{proof}

\end{document}